\documentclass[12pt,a4paper]{report}
\usepackage[english]{babel}
\usepackage{newlfont}
\usepackage{color}
\usepackage[utf8]{inputenc}
\usepackage{graphicx}
\usepackage{amsfonts}
\usepackage{amssymb}
\usepackage{mathtools}
\usepackage{mathrsfs}
\usepackage{amsthm}
\newtheorem{theorem}{Theorem}[section]
\newtheorem{proposition}{Proposition}[section]
\newtheorem{corollary}{Corollary}[theorem]
\newtheorem{lemma}[proposition]{Lemma}
\theoremstyle{definition}
\newtheorem{definition}{Definition}[section]
\theoremstyle{remark}
\newtheorem{remark}{Remark}[section]
\usepackage{pst-node}
\usepackage{tikz-cd}
\usepackage{tikz}
\usetikzlibrary{positioning, arrows}
\usepackage{leftindex}
\newtheorem{exmp}{Example}[section]
\usepackage{appendix}
\usepackage{hyperref}
\usepackage{url}

\begin{document}
\begin{titlepage}
%
%
%
%
\begin{center}
{{\Large{\textsc{Alma Mater Studiorum $\cdot$ University of  Bologna}}}} 
\rule[0.1cm]{15.8cm}{0.1mm}
\rule[0.5cm]{15.8cm}{0.6mm}
\\\vspace{3mm}
{\small{\bf School of Science \\
Department of Physics and Astronomy\\
Master Degree in Physics}}
\end{center}

\vspace{23mm}

\begin{center}\textcolor{black}{
%
%
{\LARGE{\bf On quantum $G$-structures}}\\
}\end{center}

\vspace{50mm} \par \noindent

\begin{minipage}[t]{0.47\textwidth}
%
%
{\large{\bf Supervisor: \vspace{2mm}\\\textcolor{black}{
Prof. Emanuele Latini}\\\\
%
%
%
\textcolor{black}{
\bf Co-supervisor: 
\vspace{2mm}\\
Dr. Antonio Del Donno}\\\\
}}\end{minipage}
\hfill
\begin{minipage}[t]{0.47\textwidth}\raggedleft \textcolor{black}{
{\large{\bf Submitted by:
\vspace{2mm}\\
%
%
\textcolor{black}{
Giovanni Gava}}}
}
\end{minipage}

\vspace{40mm}

\begin{center}
%
%
Academic Year \textcolor{black}{ 2023/2024}
\end{center}

\end{titlepage}

\newpage
\thispagestyle{empty} 
\mbox{} 
\newpage

\begin{abstract}
We examine principal fiber bundles in classical differential geometry, with a focus on the frame bundle and its $G$-structures. Extending these concepts to noncommutative geometry, we explore their quantum analogues within the framework of quantum principal bundles. In particular, we define a quantum frame resolution and investigate its quantum reduction. Moreover, we study in detail the case of quantum Hermitian structures on the quantum plane.
\end{abstract}


\newpage
\thispagestyle{empty} 
\mbox{} 
\newpage

\clearpage
\pagenumbering{roman}
\tableofcontents
\clearpage
\pagestyle{plain} 

\chapter*{Introduction} \label{intoduction}
\addcontentsline{toc}{chapter}{Introduction}

Differential geometry has long played a fundamental role in theoretical physics, providing the mathematical framework underlying both gauge theories and general relativity. In particular, in this setting, gauge fields can be understood as connections on certain principal fiber bundles, which in turn induce covariant derivatives on sections of associated bundles; physically, this is to be interpreted as the gauge mediator-matter field minimal coupling. In general relativity, the Levi-Civita connection arises as the unique torsion-free, metric-compatible linear connection on a Lorentzian manifold $M$. Similarly, this can be induced from a connection on the principal $SO(1,3)$-bundle of pseudo-orthonormal frames. The latter is an example of a $G$-structure, i.e., a reduction of the frame bundle $FM$---a principal $GL(n,\mathbb{R})$-bundle over an $n$-dimensional manifold $M$---to a Lie subgroup $G \subset GL(n,\mathbb{R})$. \\
Much of the essential information about the underlying geometries one considers is encoded in $G$-structures. Here, in a sense, the Lie group $G$ serves as the symmetry group of the geometry being modelled: for instance, to endow a manifold $M$ with a metric structure, one considers an $O(n)$-structure, a $2n$-dimensional manifold $M$ with an almost symplectic structure corresponds to an $Sp(2n,\mathbb{R})$-structure, while a $U(n)$-structure defines a Hermitian structure on a complex manifold. This aligns with the idea in Klein's \textit{Erlangen Program}, which generalizes Euclidean geometry by studying spaces in terms of their symmetries. A Klein geometry, like Euclidean geometry, is homogeneous, meaning that for any two points, there exists a symmetry that maps one to the other. Cartan geometry generalizes this concept by introducing "curvature" to arbitrary Klein geometries, just as Riemannian geometry introduces curvature to Euclidean geometry. This idea is illustrated in the diagram below (Fig. \ref{fig:sharpediagram}).
\begin{figure}[h]
    \centering
    \begin{tikzpicture}[
        node distance=2cm and 3cm,
        every node/.style={draw, text width=3cm, align=center, rounded corners, font=\small},
        every path/.style={->, thick},
        textnode/.style={draw=none, fill=none, font=\small} 
    ]
    
    \node (E) {Euclidean \\ geometry};
    \node (K) [right=of E] {Klein \\ geometries};
    \node (R) [below=of E] {Riemannian \\ geometry};
    \node (C) [below=of K] {Cartan \\ geometries};

    \draw (E) -- (K) node[midway, above, textnode] {generalize symmetry group};
    \draw (R) -- (C) node[midway, below, textnode] {generalize tangent space geometry};
    \draw (E) -- (R) node[midway, left, textnode] {allow curvature};
    \draw (K) -- (C) node[midway, right, textnode] {allow curvature};

    \end{tikzpicture}
    \caption{Relation between different geometries. This diagram was originally presented by Sharpe in the Preface of \cite{sharpe_dg}.}
    \label{fig:sharpediagram}
\end{figure}
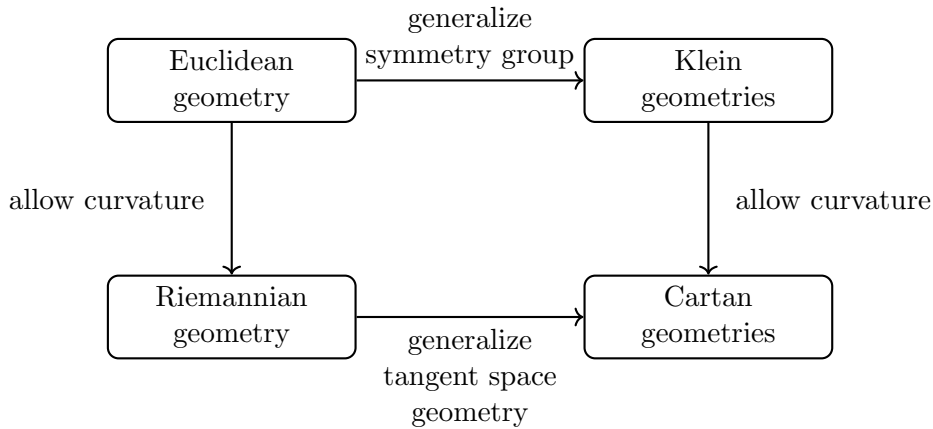
Significantly, $G$-structures are deeply connected to Cartan geometries \cite{cap_slovak}.
In the notable case of $O(p,q)$-structures, there is an important equivalence between the category of pseudo-Riemannian manifolds of signature $(p,q)$ and the category of torsion-free Cartan geometries of type $\big(\mathbb{R}^n \rtimes O(p,q), O(p,q) \big)$.
The elegance and generality of Cartan geometries also find important applications in physics. A significant one in general relativity, is discussed in \cite{wise_cartan}: a Cartan connection offers a notion of parallel transport between different homogeneous spacetime models, extending beyond the Minkowski model typically used in Einstein's theory to include the de Sitter and anti-de Sitter models, and providing a clearer understanding of the MacDowell-Mansouri formulation of general relativity.

The geometry of smooth manifolds, which we use to describe spacetime, has proven its validity up to its limits. However, it is widely accepted that at the Planck scale, where quantum gravitational effects become significant, the very nature of spacetime geometry may be fundamentally different. This is where noncommutative geometry, or quantum geometry, offers a possible path forward \cite{majid_hopfphysicsplanck}. In fact, within this framework, geometry can be extended to encompass discrete spaces and finite dimensional algebras. The classical geometry of spacetime is then expected to emerge as suitable low-energy limit \cite{majid_insights}.\\
The idea of noncommutative geometry is that, instead of working with manifolds, one adopts an algebraic perspective by considering the algebra of functions defined on them. Inspired by the noncommutativity of the algebra of phase space coordinate functions in quantum mechanics --- namely, the canonical commutation relations $[x,p]=i\hbar$ --- one then allows these function algebras to become noncommutative, typically through $q$-deformations.\\
As a result of this noncommutativity, these algebras no longer correspond to functions on an underlying classical space. However, in the algebraic framework, the existence of such a space is not a prerequisite, making it particularly well-suited for regimes where spacetime itself loses its classical meaning. In addition, by the Gel'fand-Naimark theorem, every commutative unital $C^{\ast}$-algebra is isometrically-isomorphic to $C(X)$, the algebra of continuous functions on some compact Hausdorff space $X$, allowing for the generalization of classical geometry into the noncommutative setting.
The central example in noncommutative geometry is that of Hopf algebras and quantum groups \cite{kassel_qgroups}, \cite{sweedler_hopf}. The latter can be viewed as $q$-deformations of algebras of functions over Lie groups. In physics, Hopf algebras have found a role in the interpretation of renormalization in QFT \cite{connes_hopf} and in quantum computing \cite{ercolessi_quantum}. From the philosophical standpoint, the interest in studying Hopf algebras is motivated by the \textit{principle of representation-theoretic self-duality} \cite{majid_principalselfduality}, which argues that a complete theory of physics should have a self-dual structure (in this precise representation-theoretic sense). In fact, the category of Hopf algebras is the simplest example of self-dual category after that of Abelian groups, as established in a known theorem by Pontryagin.\\

Even if it is not crucial to this project, it is worth to mention the Alain Connes' spectral triple approach to noncommutative geometry \cite{connes_ncg}, that provides a path to the unification of geometry and quantum field theory.
In this thesis, on the other hand, we work within the scope of quantum differential geometry \cite{majid_qriemanniangeo}, the other main approach to the subject. In this setting, one adopts the more general algebraic point of view, describing differential calculus over algebras. In particular, this theory provides noncommutative generalizations of principal bundles and their associated bundles.
A natural braiding occurs, enabling the definition of quantum analogues of connections, gauge transformations and metrics are given.\\
In this work, we review the arguments that are necessary to introduce the \textit{quantum frame resolution} and the noncommutative generalization of the soldering form, by which one expresses gravity as gauge theory of the frame bundle. Our original contribution lies in defining quantum $G$-structures as quantum reductions of quantum frame resolutions and proving that they, too, constitute quantum frame resolutions. Importantly, this is done in terms of a general covariant first order differential calculus, from which we induce a quotient calculus on the reduced principal comodule algebra. Furthermore, we illustrate an example of quantum frame resolution in studying the case of the quantum group of affine motions on the quantum plane, using the smashed product calculus on the smashed product algebra $\mathbb{C}_q^2\#GL_q(2)$. We then discuss its quantum reduction to the quantum Hermitian structure $\mathbb{C}_q^2\#U_q(2)$.

\section*{Outline}
In Chapter \ref{principalfiberbundles}, we provide an overview of the principal bundle formalism in differential geometry. We discuss associated bundles and reductions of principal bundles. A connection is defined in terms of a $G$-equivariant horizontal distribution $\mathcal{H}P = \ker{\omega}$, for some $\mathfrak{g}$-valued $1$-form $\omega$, the connection $1$-form. We show how gauge fields are obtained as the pullback by local sections of the principal connection $\omega$, and how one recovers the gauge transformation formula. We then discuss the important example of the frame bundle $FM \to M$ over an $n$-dimensional manifold $M$ and its canonical form $\theta \in \Omega^1_{hor}(FM;\mathbb{R}^n)^{GL(n,\mathbb{R})}$. Particular attention is given to $G$-structures, which can be viewed as principal bundles $P \to M$ with the together with a $1$-form $\Theta \in \Omega^1_{hor}(P;\mathbb{R}^n)^{G}$, the equivalent of the soldering form, obtained as the pullback of the latter by the bundle morphism defining the reduction. Their direct link to Cartan geometries is illustrated in the affine case.

In Chapter \ref{Hopfalgebras}, we shift our discussion to the algebraic perspective. We define algebras and coalgebras. A bialgebra is understood as a vector space with both an algbera and a coalgebra structure. An Hopf algebra is then defined as a bialgebra with the addition of an antipode, playing the role of group inversion. We then present comodules and comodule algebras, the latter being the non commutative analogues of total spaces in the quantum principal bundle framework. We study trivial and cleft extensions as special cases of Hopf-Galois extensions. We then focus on the trivial case and on the corresponding smashed product algebra. Finally, we present the theory of first order differential calculi on algebras, which provides a notion of a differential structure in the quantum-geometric setting. We define first order covariant calculi. Importantly, we examine the case of the pullback and quotient calculi, as well as their covariant versions. A definition of the quantum Maurer-Cartan form is given. We show the construction of a smashed product calculus on the smashed product algebra.

In Chapter \ref{quantumprincipalbundles}, the theory of quantum principal bundles is introduced. They are understood as faithfully flat Hopf-Galois extensions $B := A^{coH} \subseteq A$, where $B$ is the subalgebra of coinvariant elements of the comodule algebra $A$ under its coaction. Two equivalent definitions of associated quantum vector bundles $\mathcal{E}$ are presented.
A definition of base space calculus $(\Gamma_B,\mathrm{d}_B)$ on $B$ and of horizontal forms $A\Gamma_B$ is provided. Lastly, we illustrate the definition of quantum reductions of such quantum principal bundles. In particular, given a principal $H$-comodule algebra $A \subseteq B$, a reduction is a principal $H_0$-comodule algebra $B_0 := A_0^{coH_0} \subseteq A_0$, with $H_0 := H/J$ for some Hopf ideal $J$. This is obtained via a morphism of $H_0$-comodule algebras $\phi \colon A \to A_0$ with the property that its restriction to $B$ gives an isomorphism of algebras $B \cong B_0$.

In Chapter \ref{quantumframebundle} we provide an axiomatic definition of frame bundle as a principal bundle together with a canonical form $\theta \in \Omega^1_{hor}(FM;\mathbb{R}^n)^{GL(n,\mathbb{R})}$. An important correspondence between equivariant functions on a principal bundle $P\to M$ with values in a suitable vector space $V$ and sections of the corresponding associated bundle is shown. This is later extended to the level of differential forms as the correspondence between the space of forms on the base space $M$ with values on the associated bundles $P \times_G V$ and that of horizontal and equivariant forms with values in $V$. Such correspondence, in light of the case of the frame bundle, justifies the definition of a frame resolution of the tangent bundle $TM$ as a principal bundle with an horizontal and equivariant form $\theta$ inducing a bundle isomorphism $TM \cong P \times_G V$. We then provide noncommutative counterparts to such arguments. The translation map is introduced, which is then used to demonstrate the correspondence between strongly tensorial forms $\theta \colon A\Gamma_B \to V$ and left $B$-module maps $s \colon \mathcal{E} = (A \otimes V)^{coH} \to \Gamma_B$. From this, a quantum frame resolution of $(B,\Gamma_B)$ is defined as a tuple $(A,H,V,\theta)$, with $A$ a principal $H$-comodule algebra,  such that $\theta$ induces an isomorphism of left $B$-modules $\mathcal{E} \cong \Gamma_B$.
As discussed earlier, we study the case of the quantum group of affine motions on the quantum plane, providing an original example of quantum frame resolution in terms of the smashed product algebra $\mathbb{C}^2_q \# GL_q(2)$ and the smashed product calculus on it. Finally, we introduce the definition of quantum $G$-structures. Precisely, given a quantum frame resolution $(A,H,V,\theta)$, a quantum $G$-structure $(A_0,H_0,V,\theta_0)$ is obtained as a quantum reduction $\phi \colon A \to A_0$ to the Hopf algebra $H_0$, together with a strongly tensorial $\theta_0 := \Phi_{\Gamma} \circ \theta$. Here, $\Phi_{\Gamma}$ is the morphism of differential calculi between the assumed covariant first order differential calculus $(\Gamma_A,\mathrm{d}_A)$ on $A$ and the quotient calculus $(\Gamma_{A_0},\mathrm{d}_{A_0})$ on $A_0$ induced by $\phi$. We prove that $(A_0, H_0, V, \phi)$ is a quantum frame resolution itself, utilizing the fact that the restriction of $\Phi_{\Gamma}$ to the base space calculus gives an isomorphism of differential calculi $\Gamma_{B} \cong \Gamma_{B_0}$. At last, we provide an example by considering $A_0 = \mathbb{C}_q^2\#U_q(2)$ and $H_0 = U_q(2)$.

\section*{Acknowledgements} I am deeply grateful to my supervisors, Prof. Emanuele Latini and Dr. Antonio Del Donno, for their genuine support and patience. Every discussion and meeting with them has been a truly enjoyable experience. Their passion for the subject was highly inspiring, and this enthusiasm will undoubtedly stay with me in the years to come. I also want to thank Dr. Thomas Weber for his valuable insights throughout this work.

\chapter{Principal fiber bundles} \label{principalfiberbundles}
\pagenumbering{arabic}
In this chapter, we provide a review of principal fiber bundles within the framework of differential geometry. This serves as a reference for the discussion of their quantum geometric counterparts, i.e., quantum principal bundles, which will be presented in later chapters. Our main references are \cite{farrill_gauge}, \cite{sharpe_dg} and \cite{cap_slovak}.

In Section \ref{principalbundles}, we introduce principal bundles, their associated bundles, and reductions. Section \ref{connections} focuses on connections: gauge fields, the mediators of the fundamental interactions we encounter in physics, are understood as connections on particular principal bundles. We show both the interpretation of a connection as an horizontal distribution and as a $1$-form with values in the Lie algebra $\mathfrak{g}$ associated to the structure group of the principal bundle. In \ref{theframebundle}, we analyze a special case, that of the frame bundle, highlighting one of its remarkable features, i.e., the soldering form. Section \ref{Gstructures} covers $G$-structures, or reductions of the frame bundle to a structure group $G \subset GL(n,\mathbb{R})$. We show that these can be viewed as principal bundles equipped with a horizontal and equivariant $1$-form $\Theta$. Finally, in \ref{cartangeometries}, we demostrate the connection between $G$-structures and Cartan geometry in the affine case.

\section{Principal bundles} \label{principalbundles}

\begin{definition}
    A \textsl{principal fiber bundle} $(P,M,G,\pi)$ is the information of:
    \begin{itemize}
        \item A smooth manifold $P$, called \textsl{total space}.
        \item A smooth manifold $M$, called \textsl{base space}. The base space is diffeomorphic to the quotient $M \cong P/G$. Points of $M$ are called orbits.
        \item A smooth surjection $\pi \colon P \to M$, called \textsl{projection map}. For every $m \in M$, the submanifold $\pi^{-1}(m) \subset P$ is called the \textsl{fiber} over $m$. 
        \item A Lie group $G$, called the \textsl{structure group}, acting freely and transitively on $P$ from the right, via the action 
        \begin{equation}
            \begin{split}
                R \colon &P \times G \to P,\\
                &(p, g) \mapsto pg = R_g(p),  
            \end{split}
            \label{eq:rightGactioninthedef}
        \end{equation}
        such that
        \begin{enumerate}
            \item $G$ acts freely, i.e., every element in $G$ except the identity $e$ moves every point in $P$: if $pg = p$ then $g = e$;
            \item $G$ acts transitively on fibers, i.e., for any two points $p_1$ and $p_2$ in the same fiber there exists a structure group element $g \in G$ such that $p_2=p_1 g$; 
            \item the fibers are $G$-invariant, meaning $\pi(pg) = \pi(p)$ for any $p \in P$ and $g \in G$.
        \end{enumerate}
    \end{itemize}
    Moreover, a principal $G$-bundle $\pi \colon P \to M$ is \textsl{locally trivial}: the base space $M$ admits an open cover $\{ U_{\alpha} \}$ for which there exist diffeomorphisms 
    \begin{equation}
        \begin{split}
            \psi_{\alpha} \colon &\pi^{-1}(U_{\alpha}) \to U_{\alpha} \times G,\\
            &p \mapsto \psi_{\alpha}(p):=\big(\pi(p),g_{\alpha}(p)\big)
        \end{split}
        \label{eq:trivializationmaps}
    \end{equation}
    called \textsl{trivialization maps}, where $g_{\alpha} \colon \pi^{-1}(U_{\alpha}) \to G$ are $G$-equivariant fiberwise diffeomorphisms, such that the following diagram
    \[
\begin{tikzcd}
\pi^{-1}(U_{\alpha}) \arrow[rd, "\pi"'] \arrow[rr, "\psi_{\alpha}"] &            & U_{\alpha} \times G \arrow[ld, "pr_1"] \\
                                                                    & U_{\alpha} &                                       
\end{tikzcd}
    \]
    commutes. By $G$-equivariance we mean that $g_{\alpha}(pg)=g_{\alpha}(p)g$ for any $p \in P$ and $g \in G$.\\
    We say that a principal fiber bundle is \textsl{trivial} when there exists a (globally defined) diffeomorphism 
    \begin{equation}
        \begin{split}
            \psi \colon &P \to M \times G,\\
            &p \mapsto \psi(p):=\big(\pi(p), \chi(p)\big)
        \end{split}
        \label{eq:triavialbundleforclarity}
    \end{equation}
    for a $G$-equivariant differomorphism $\chi \colon P \to G$.
    \label{def:principalfiberbundle}
\end{definition}
\begin{definition}
    A \textsl{section} of a principal fiber bundle $\pi \colon P \to M$ is a smooth map
    \begin{equation}
        s \colon M \to P
        \label{eq:sectiondef}
    \end{equation}
    such that $\pi \circ s = id_M$. Similarly, smooth maps
    \begin{equation}
        s_{\alpha} \colon U_{\alpha} \to \pi^{-1}(U_{\alpha}),
        \label{eq:localsections}
    \end{equation}
    satisfying $\pi \circ s_{\alpha} = id_M$ are called \textsl{local sections}.
    \label{def:section}
\end{definition}
Note that (global) sections $s \colon M \to P$ are rare. The following proposition is given.
\begin{proposition}
    A principal fiber bundle is trivial if and only if it admits a (global) section.
    \label{pro:trivialbundleglobalsection}
\end{proposition}
\begin{proof}
    Let us first show that if a principal fiber bundle admits a section then it is trivial. Let $s \colon M \to P$ be such section, picking out a distinguished point $s(m) \in \pi^{-1}(m)$ for every orbit in the base space. By transitivity of the group action, there exists a unique $g \in G$ such that we can write $p= s\big(\pi(p)\big)g =: s\big(\pi(p)\big)\chi(p)$, where $\chi$ is a smooth map which selects that specific group element for each point in $p \in P$. Moreover, $\chi$ is $G$-equivariant, since we can write $pg = s\big(\pi(pg) \big)\chi(pg) = s\big(\pi(p) \big)\chi(pg)$ but also $pg=s\big(\pi(p)\big)\chi(p)g$. This implies we can define a (global) trivialization $\psi \colon P \to M \times G$ as $\psi(p)=\big(\pi(p), \chi(p)\big) = (m, g)$, which is smooth and has smooth inverse. The inverse is in fact $\phi \colon M \times G \to P$ defined as $\phi(m,g)=s(m)g$: clearly, for every $p=s(m)g \in P$, $\psi(\phi(m,g))=\psi(p)=(m,g)$ and $\phi(\psi(p))=p$.\\
    Conversely, if a bundle is trivial then there exists a diffeomorphism $\psi \colon P \to M \times G, p \mapsto \psi(p)=(\pi(p),\chi(p))$ such that $\chi(pg)=\chi(p)g$ for some $g \in G$. 
    Consider the map $s \colon M \to P, m \mapsto s(m) := \psi^{-1}(m,e)$, where $e$ is the neutral element of $G$. Clearly $s$ is a section because it is smooth and we can write $\pi \circ s(m)=\pi(\psi^{-1}(m,e))=m$, since $\psi^{-1}(m,e) \in \pi^{-1}(m)$.
\end{proof}
A principal fiber bundle always admits local sections $s_{\alpha} \colon U_{\alpha} \to \pi^{-1}(U_{\alpha})$ as in Definition \ref{def:section}. In particular, its local triviality property can be expressed in terms of such local sections. In fact, they are canonically associated to the trivialization maps as
\begin{equation}
    \psi_{\alpha}\big(s_{\alpha}(m)\big)=(m, e),
    \label{eq:aggiungoperchiarezza}
\end{equation}
i.e., $s_{\alpha}(m)=\psi_{\alpha}^{-1}(m,e)$ for any $m \in M$. In other words, $g_{\alpha} \circ s_{\alpha} \colon U_{\alpha} \to G$ is the constant function sending every orbit $m \in U_{\alpha}$ to the neutral element $e$ of the structure group. \\
Summing things up, this means that given any $p \in \pi^{-1}(m)$, by transitivity of the right $G$-action, there is a unique group element $g_{\alpha}(p) \in G$ such that we can write
\begin{equation}
    p = s_{\alpha}(m)g_{\alpha}(p)
    \label{eq:quellochemihafattocapire}
\end{equation}
and
\begin{equation}
    \psi_{\alpha} (p) = \Big(\pi\big(s_{\alpha}(m)g_{\alpha}(p)\big), g_{\alpha}\big(s_{\alpha}(m)g_{\alpha}(p)\big)\Big) = \Big(\pi\big(s_{\alpha}(m)\big), g_{\alpha}(p)\Big).
    \label{eq:generaltrivialization}
\end{equation}
Clearly, for $p=s_{\alpha}(m)$ then $g_{\alpha}(p)=e$.\\

Let $U_{\alpha}, U_{\beta} \subset M$ such that $U_{\alpha\beta} := U_{\alpha} \cap U_{\beta} \neq \emptyset$. For $m \in U_{\alpha\beta}$ and $p \in \pi^{-1}(m)$, on $\pi^{-1}(U_{\alpha\beta})$, we can then trivialize the bundle in two different ways: by $\psi_{\alpha}(p)=\big(m, g_{\alpha}(p)\big)$ and by $\psi_{\beta}(p)=\big(m, g_{\beta}(p)\big)$. There must exists a group element $\bar{g}_{\alpha\beta}(p) \in G$ such that $g_{\alpha}(p)=\bar{g}_{\alpha\beta}(p)g_{\beta}(p)$, that is:
\begin{equation}
    \bar{g}_{\alpha\beta}(p)=g_{\alpha}(p)g_{\beta}(p)^{-1}.
    \label{eq:over}
\end{equation}
We now notice that $\bar{g}_{\alpha\beta} \colon \pi^{-1}(U_{\alpha\beta}) \to G$ is constant on each fiber. In fact, by $G$-equivariance of $g_{\alpha}$ and $g_{\beta}$ we have that 
\begin{equation}
    \bar{g}_{\alpha\beta}(p g)= g_{\alpha}(pg)g_{\beta}(pg)^{-1} = g_{\alpha}(p)g g^{-1} g_{\beta}(p)^{-1}= g_{\alpha}(p) g_{\beta}(p)^{-1}=\bar{g}_{\alpha\beta}(p).
    \label{eq:proofthatgbaralphabetaisconstantonfibers}
\end{equation}
We can translate this into the following definition.
\begin{definition}
    There exist smooth maps
    \begin{equation}
    g_{\alpha\beta} \colon U_{\alpha\beta} \to G,
    \label{eq:transitionfunctions}
    \end{equation}
    called \textsl{transition functions}, on any non-empty overlap $U_{\alpha\beta}$, such that $\bar{g}_{\alpha\beta}(p)=g_{\alpha\beta}(\pi(p))$ for all $p \in \pi^{-1}(U_{\alpha\beta})$.
    \label{def:transitionfunctions}
\end{definition}
\begin{lemma}
    Transition functions satisfy the following cocycle conditions:
    \begin{equation}
        g_{\alpha\beta}(m)g_{\beta\alpha}(m)=e,\;\;\; \forall m \in U_{\alpha\beta};
        \label{eq:cocy1}
    \end{equation}
    \begin{equation}
        g_{\alpha\beta}(m)g_{\beta\gamma}(m)g_{\gamma\alpha}(m)=e,\;\;\; \forall m \in U_{\alpha\beta\gamma}:=U_{\alpha} \cap U_{\beta} \cap U_{\gamma}.
        \label{eq:cocy2}
    \end{equation}
    \label{lem:cocycleconditionsfortransitionfunction}
\end{lemma}
\begin{proof}
    Let $p \in \pi^{-1}(m)$. From (\ref{eq:over}) we get:
    \begin{equation}
        g_{\alpha\beta}(m)g_{\beta\alpha}(m)= \bar{g}_{\alpha\beta}(p)\bar{g}_{\beta\alpha}(p) = g_{\alpha}(p)g_{\beta}(p)^{-1}g_{\beta}(p)g_{\alpha}(p)^{-1} = e.
    \label{eq:proofcocy1}
    \end{equation}
    Similarly,
    \begin{equation}
        \begin{split}
            &g_{\alpha\beta}(m)g_{\beta\gamma}(m)g_{\gamma\alpha}(m) = \bar{g}_{\alpha\beta}(p)\bar{g}_{\beta\gamma}(p)\bar{g}_{\gamma\alpha}(p) = \\
            &= g_{\alpha}(p)g_{\beta}(p)^{-1}g_{\beta}(p)g_{\gamma}(p)^{-1}g_{\gamma}(p)g_{\alpha}(p)^{-1} = e.
        \end{split}
        \label{eq:proofcocy2}
    \end{equation}
\end{proof}
Transition functions determine the relation between different local sections defined on non-empty overlaps $U_{\alpha\beta}$.
\begin{lemma}
    Let $s_{\alpha} \colon U_{\alpha} \to \pi^{-1}(U_{\alpha})$ and $s_{\beta} \colon U_{\beta} \to \pi^{-1}(U_{\beta})$ be two different local setions defined on $U_{\alpha\beta} = U_{\alpha} \cap U_{\beta} \neq \emptyset$. Then
    \begin{equation}
        s_{\beta}(m)=s_{\alpha}(m)g_{\alpha\beta}(m), 
        \label{eq:sexrel}
    \end{equation}
    for all $m \in U_{\alpha\beta}$.
    \label{lem:relationbetweendifferentlocalsections}
\end{lemma}
\begin{proof}
    For $p \in \pi^{-1}(m)$, $m \in U_{\alpha\beta}$ we have that $p=s_{\alpha}(m)g_{\alpha}(p)=s_{\beta}(m)g_{\beta}(p)$, which implies
    \begin{equation}
        s_{\beta}(m)=s_{\alpha}(m)g_{\alpha}(p)g_{\beta}(p)^{-1}=s_{\alpha}(m)\bar{g}_{\alpha\beta}(p)=s_{\alpha}(m)g_{\alpha\beta}(\pi(p))=s_{\alpha}(m)g_{\alpha\beta}(m).
        \label{eq:proofsextionsalphabeta}
    \end{equation}
\end{proof}
\begin{remark}
    One can construct the total space from an open cover $\{U_{\alpha}\}$ and transition functions $\{g_{\alpha\beta}\}$ obeying the cocycle conditions in Lemma \ref{lem:cocycleconditionsfortransitionfunction} as:
    \begin{equation}
        P = \cup_{\alpha} ( U_{\alpha} \times G ) \Big/ \sim ,
        \label{eq:union}   
    \end{equation}
    where $(m,g) \sim (m, g_{\alpha\beta}(m)g)$ for all $m \in U_{\alpha\beta}$ and $g \in G$. $\pi$ is induced by the projection on the first factor and the right $G$-action on the total space is induced by right multiplication on $G$.
    \label{rem:onhowtoconstructPfromMandG}
\end{remark}

\subsection{Associated bundles} \label{associatedbundles}

\begin{definition}
    Let $G$ act on a differentiable manifold $F$ through a representation $\rho \colon G \to \text{Aut}(F)$. We can use the data of a principal $G$-bundle $\pi \colon P \to M$ to define a bundle $\pi_F \colon  E \to M$ with standard fiber $F$, called the \textsl{associated fiber bundle} to $P$. The total space is defined as the quotient
    \begin{equation}
        E = P \times_{G} F := (P \times F) \Big/ \sim
        \label{eq:assfb}
    \end{equation}
    by the equivalence relation
    \begin{equation}
        (p,f) \sim (pg, \rho(g^{-1})f),
        \label{eq:equivalenceassbundlejusttobeclear}
    \end{equation}
    where $(p,f) \in P \times F$, and the projection map $\pi_F$ is induced by $\pi$ as 
    \begin{equation}
        \begin{split}
            \pi_F \colon &P \times_G F \to M,\\
            &[\![p,f]\!] \mapsto \pi_F([\![p,f]\!]) := \pi(p).
        \end{split}
        \label{eq:cosaimprecisa}
    \end{equation}
    When the standard fiber of the associated bundle is a vector space $V$, so that $G$ acts on it via a linear representation $G \to \text{Aut}(V)= GL(V)$, we will call $E = P \times_G V$ the \textsl{associated vector bundle}.
    \label{def:associatedfiberbundle}
\end{definition}
\begin{remark}
    The associated fiber bundle $P \times_G F$ can also be constructed locally from the local data defining $P$,  namely the open cover $\{ U_{\alpha} \}$ and the transition functions $\{g_{\alpha\beta}\}$ on double overlaps. Indeed we have that:
    \begin{equation}
        P \times_G F = \cup_{\alpha} (U_{\alpha} \times G) \Big/ \sim,
        \label{eq:assunion}
    \end{equation}
    with $(m, f) \sim (m, \rho(g_{\alpha\beta}(m))f)$ for all $m \in U_{\alpha\beta}$ and $f \in F$.
    \label{rem:constructtheassociatedbundlefromlocaldata}
\end{remark}

\subsection{Reductions} \label{reductions}
\begin{definition}
    A \textsl{principal bundle morphism} between two principal $G$-bundles $\pi \colon P \to M$ and $\pi^{\prime} \colon P^{\prime} \to M^{\prime}$ consists of a pair of smooth maps $(\Phi, \phi)$, with $\Phi \colon P \to P^{\prime}$ and $\phi \colon M \to M^{\prime}$, such that
    \begin{equation}
        \Phi(pg) = \Phi(p)g
        \label{eq:principalbundlemorphismcondition1}
    \end{equation}
    for all $p \in P$ and $g \in G$, and
    \begin{equation}
        \pi^{\prime} \circ \Phi = \phi \circ \pi.
        \label{eq:principalbundlemorphismcondition2}
    \end{equation}
    Whenever $\Phi$ and $\phi$ are diffeomorphisms, we say that $(\Phi,\phi)$ is a \textsl{principal bundle isomorphism}.
    \label{def:principalbundlemorphism}
\end{definition}
In the definition above, the base space map $\phi$ is often the identity, which is typically the case in most examples. A particularly important instance of morphisms of principal bundles arises in the context of reductions of the structure group.
\begin{definition}
    A \textsl{reduction} of the principal $G$-bundle $\pi \colon P \to M$ to the structure group $H \subset G$, with $H$ a Lie subgroup of $G$, is a principal $H$-bundle $\pi_r \colon R \to M$ together with a principal bundle morphism $\iota \colon R \to P$ with respect to the inclusion $i \colon H \to G$, that is
    \begin{equation}
        \pi \circ \iota = \pi_r
        \label{eq:reductionproperty1}
    \end{equation}
    and
    \begin{equation}
        \iota(ph) = \iota(p)i(h)
        \label{eq:reductionproperty2}
    \end{equation}
    for all $p \in R$ and $h \in H$, or, equivalently $\iota \circ R_h = R_{i(h)} \circ \iota$.
    \label{def:reduction}
\end{definition}
Restricting the right $G$-action to $H$, we obtain a free right action of $H$ on $P$, and one can show that the space $P/H$ of orbits of this action is a fiber bundle over $M$ with fiber $G/H$. An argument based on cocycle conditions of transition functions shows that reductions of $P$ to the structure group $H$ are in bijective correspondence with the set of global smooth sections of the fiber bundle $P/H \to M$.

\section{Connections} \label{connections}

Let $\pi \colon P \to M$ be a principal $G$-bundle. 
\begin{definition}
    The \textsl{vertical subspace} $\mathcal{V}_pP \subset T_pP$ is the kernel of the push forward (at $p$) of the projection map:
    \begin{equation}
        \mathcal{V}_pP=\ker{\pi_{\ast}},
        \label{eq:vertd}
    \end{equation}
    with $\pi_{\ast} \colon T_pP \to T_{\pi(p)}M$. A vector field $\xi \in \mathscr{X}(P)$ is called \textsl{vertical} if $\xi_p \in \mathcal{V}_pP$ for all $p \in P$.
    \label{def:verticalsubspace}
\end{definition}
\begin{remark}
     We view the vertical subspace $\mathcal{V}_pP \subset T_pP$ as the set of vectors in $T_pP$ which are tangent to the fiber $\pi^{-1}(m)$ over $m = \pi(p)$ at $p$. In fact, for a vector $x = \gamma^{\prime}(0) \in \mathcal{V}_pP$ tangent to the curve $\gamma \colon (a,b) \to \pi^{-1}(m)$ at a point $p = \gamma(0) \in \pi^{-1}(m)$, we have that 
    \begin{equation}
        \pi_{\ast}(x)= \pi_{\ast}(\gamma^{\prime}(0))=(\pi \circ \gamma)^{\prime}(0) = 0,
        \label{eq:toshowverticalvectorsarethosetangenttothefiber}
    \end{equation}
    since $(\pi \circ \gamma)$ is constant for any $p=\gamma(t) \in \pi^{-1}(m)$ (in particular, $\pi(\gamma(t))=m$ for any $t \in (a,b)$).
    \label{rem:onthefactthatverticalvectorsaretangenttothefiber}
\end{remark}
A \textsl{distribution} $\mathcal{D}$ on a manifold $M$ is a subset $\mathcal{D} \subset TM$ such that for each $x \in M$ the subset $ \mathcal{D}_x  = \mathcal{D} \cap T_xM$ is a vector subspace in $T_xM$. By linear algebra, each distribution can be defined as the kernel of a one-form $\omega$ with values in a suitable vector space $V$. Thus, we can think of vertical subspaces both as $\mathcal{V}_pP=\ker{\pi_{\ast}}$ and as defining a distribution $\mathcal{V}P \subset TP$, called the vertical distribution, on the manifold $P$: $\mathcal{V}_pP = \mathcal{V}P \cap T_pP$.\\
\begin{remark}
    Since $\pi \circ R_g = \pi$, with $R_g \colon P \to P, p \mapsto pg$, the vertical subspaces define a $G$-invariant distribution $\mathcal{V}P \subset TP$: 
    \begin{equation}
        (R_g)_{\ast}\mathcal{V}_pP= \mathcal{V}_{pg}P,
        \label{eq:Ginv}
    \end{equation}
    with $(R_g)_{\ast} \colon T_pP \to T_{pg}P$. In fact, for any $x \in \mathcal{V}_pP$, we can write $\pi_{\ast}\big((R_g)_{\ast}(x)\big)= (\pi \circ R_g)_{\ast}(x)=\pi_{\ast}(x) = 0$ which shows that $(R_g)_{\ast}(x) \in \mathcal{V}_{pg}P$ for every $x \in \mathcal{V}_pP$.
    \label{rem:ontheuniquenessoftheverticalsubspace}
\end{remark} 
The vertical distribution $\mathcal{V}P$ on the manifold $P$ is uniquely determined by the geometry of the bundle, as it consists of vectors tangent to the fibers, and does not depend on any additional choices. However, in general, there is no canonical way to select a complementary subspace to $\mathcal{V}_pP$ within $T_pP$ without introducing extra structure. This is precisely what a connection provides.

\subsection{The horizontal distribution and the connection $1$-form}

A connection on $P$ is a smooth choice of horizontal subspaces $\mathcal{H}_pP \subset T_pP$ complementary to the vertical subspace $\mathcal{V}_pP = \ker{\pi_{\ast}} \subset T_pP$ (which is unique for each $T_pP$) at each point:
\begin{equation}
    T_pP= \mathcal{V}_pP \oplus \mathcal{H}_pP.
    \label{eq:orsum1}
\end{equation}
\begin{definition}
    A \textsl{connection} on $P$ is a $G$-invariant distribution $\mathcal{H}P \subset TP$
    \begin{equation}
        (R_g)_{\ast}\mathcal{H}_pP= \mathcal{H}_{pg}P,
        \label{eq:orsum2}
    \end{equation}
    called the \textsl{horizontal distribution}, complementary to $\mathcal{V}P$.
    \label{def:connectionhorizontaldistribution}
\end{definition}
We now introduce the definition of fundamental vector fields, which provides a natural way to identify vertical vectors with elements of the Lie algebra $\mathfrak{g}$ associated with the structure group $G$ of the principal bundle. This will later enable us to define the connection $1$-form.
\begin{definition}
    The right $G$-action on the total space defines a map 
    \begin{equation}
        \begin{split}
            \zeta \colon &\mathfrak{g} \to \mathscr{X}(P),\\
            &X \mapsto \zeta(X).
        \end{split}
        \label{eq:justforcalrityfundvectorfields}
    \end{equation}
    The vector field $\zeta(X) \in \mathscr{X}(P)$ is called the \textsl{fundamental vector field}, and describes the infinitesimal action of $G$ on $P$. At each point $p \in P$, it is defined as
    \begin{equation}
        \zeta_p(X) = \frac{d}{dt}(p e^{tX}) \Bigr|_{t=0}.
        \label{eq:fundvf}
    \end{equation}
    \label{def:fundamentalvectorfield}
\end{definition}
We notice that 
\begin{equation}
    \pi_{\ast}\big(\zeta_p(X)\big)=\pi_{\ast}\big(\frac{d}{dt}(p e^{tX}) \Big|_{t=0}\big)= \frac{d}{ds}\big(\pi (p e^{sX})\big) \Big|_{s=0}=\frac{d}{ds} \big(\pi (p)\big) \Big|_{s=0} = 0, 
    \label{eq:proofsigmavertical}
\end{equation}
which tells us that $\zeta_p(X) \in \mathcal{V}_pP$, i.e. that $\zeta(X)$ is a vertical vector field. Thus, the map $\zeta_p \colon \mathfrak{g} \to \mathcal{V}_pP$ takes an element $X$ of the Lie algebra, which can be seen as a tangent vector to the identity point of the Lie group $G$, and assigns to it a vector $\zeta_p(X)$ tangent to the fiber $\pi^{-1}\big(\pi(p)\big)$ at $p$. This makes sense, since the fibers of a principal bundle are diffeomorphic to the structure group. By this argument, since $G$ acts freely and transitively on the total space, the map $\sigma_p \colon \mathfrak{g} \to \mathcal{V}_pP$ is an isomorphism between the Lie algebra and the vertical subspace $\mathcal{V}_pP$, for every $p \in P$.
\begin{lemma}
    For any $X \in \mathfrak{g}$, we have
    \begin{equation}
        (R_g)_{\ast}\big(\zeta(X)\big)= \zeta\big(ad_{g^{-1}}(X)\big),
        \label{eq:lemmasigma}
    \end{equation}
    with
    \begin{equation}
        \begin{split}
            ad_g \colon &\mathfrak{g} \to \mathfrak{g},\\
            &X \mapsto ad_g(X) := \frac{d}{dt} (g e^{tX} g^{-1}) \Bigr|_{t=0}.
        \end{split}
        \label{eq:adjointstuffadgmenouno}
    \end{equation}
    \label{lem:Rgstaronzeta}
\end{lemma}
\begin{proof}
    From the definition of fundamental vector field we can write, for any $p \in P$, that 
    \begin{equation}
        \begin{split}
            &(R_g)_{\ast}\big(\zeta_p(X)\big) = (R_g)_{\ast}\big(\frac{d}{dt}(p e^{tX}) \Big|_{t=0}\big) = \frac{d}{ds}\big(R_g(p e^{sX})\big)\Big|_{s=0} = \\
            &= \frac{d}{ds}(p e^{sX} g)\Big|_{s=0} = \frac{d}{ds}(p g g^{-1} e^{sX} g)\Bigr|_{s=0}.
        \end{split}
        \label{eq:Rgstaronzeta1}
    \end{equation}
    Since $ad_{g^{-1}}(X) =  \frac{d}{dt} (g^{-1} e^{tX} g) \Big|_{t=0}$ and 
    \begin{equation}
        \frac{d}{dt} (e^{t \cdot ad_{g^{-1}} (X)}) \Big|_{t=0} = ad_{g^{-1}}(X)e^{t \cdot ad_{g^{-1}}(X)}\Big|_{t=0} = ad_{g^{-1}}(X),
        \label{eq:Rgstaronzeta2}
    \end{equation}
    we have that Equation \ref{eq:Rgstaronzeta1} reads:
    \begin{equation}
        \frac{d}{ds}(p g g^{-1} e^{sX} g)\Big|_{s=0} = \frac{d}{dr}(p g  e^{r \cdot ad_{g^{-1}} (X)} )\Big|_{r=0} = \zeta_{pg}\big(ad_{g^{-1}} (X)\big),
        \label{eq:Rgstaronzeta3}
    \end{equation}
    which proves the lemma.
\end{proof}
Let $\text{dim}(\mathcal{V}_pP) = \text{dim}(\mathfrak{g}) = k$. Then the horizontal subspace $\mathcal{H}_pP \subset T_pP$, being a linear subspace, can be seen as the kernel of a $1$-form $\omega$ at $p$ with values in a suitable $k$-dimensional vector space, which can be naturally be chosen to be the Lie algebra $\mathfrak{g}$ of $G$. Since $\omega$ annihilates horizontal vectors, it is entirely determined by its action on the vertical vector fields. This leads to the following definition.
\begin{definition}
    Let $\mathcal{H}P$ be a connection on $P$. Then $\mathcal{H}P = \ker{\omega}$, for a $1$-form $\omega \in \Omega^1(P;\mathfrak{g})$, called the \textsl{connection} $1$\textsl{-form}, satisfying:
    \begin{equation}
        \omega(\xi) = 
        \begin{cases}
        X, \text{ if } \xi=\zeta(X) \\
        0, \text{ if } \xi \text{ horizontal, i.e. } \xi_p \in \mathcal{H}_pP \text{ for all } p \in P
        \end{cases}
        \;\;\;\; \text{,} 
        \label{eq:defcon}
    \end{equation}
    with $\xi \in \mathscr{X}(P)$.
    \label{def:connectiononeform}
\end{definition}
Note that one has a particular connection $1$-form for each connection, i.e. for each choice of an horizontal direction. Let $\mathcal{H}P = \ker{\omega}$. A different horizontal distribution $\mathcal{H}^{\prime}P$ will be identified by a different connection $1$-form $\omega^{\prime}$ as $\mathcal{H}^{\prime}P=\ker{\omega^{\prime}}$. The two connection $1$-forms $\omega$ and $\omega^{\prime}$ identify two different complements $\mathcal{H}_pP$ and $\mathcal{H}^{\prime}_pP$ to $\mathcal{V}_pP$ in $T_pP$:
    \begin{equation}
        T_pP = \mathcal{V}_pP \oplus \mathcal{H}_pP = \mathcal{V}_pP \oplus \mathcal{H}^{\prime}_pP.
        \label{eq:corressioni2}
    \end{equation}
More precisely, this means that a vector field $x \in \mathscr{X}(P)$ which is annihilated by $\omega$ is not annihilated by $\omega^{\prime}$ in general, or, in other words, $x$ has non-zero vertical component in the eyes of $\omega^{\prime}$.\\
Moreover, notice that a vertical vector field $\xi=\zeta(X)$ such that $\omega(\xi)=X$, is sent by any different connection $1$-form $\omega^{\prime}$, by definition, to the same Lie algebra element $X$:
    \begin{equation}
        \omega(\xi) = \omega^{\prime}(\xi) = X.
        \label{eq:corressioni1}
    \end{equation}
This implies that the difference $\tau = \omega - \omega^{\prime}$ of two connection $1$-forms is horizontal, i.e., it annihilates vertical vector fields.
\begin{remark}
    The space of sections for the vertical distribution $\mathcal{V}P$ is the set of vertical vector fields on $P$ (the fundamental vector field $\zeta(X)$ is an example). Analogously, for each horizontal distribution $\mathcal{H}P$, the space of its sections is the set of (horizontal) vector fields which are sent to zero by the connection $1$-form $\omega$ defining the horizontal distribution as $\mathcal{H}P = \ker{\omega}$.
    \label{rem:onsectionsoftheverticalandhorizontaldistributions}
\end{remark}
\begin{proposition}
    A connection one-form obeys the following $G$-equivariance condition:
    \begin{equation}
        (R_g)^{\ast} \omega = ad_{g^{-1}} \circ \omega.
        \label{eq:propRom}
    \end{equation}
    \label{pro:theconnectiononeformisequivariant}
\end{proposition}
\begin{proof}
    In order to prove the proposition we need to show that $((R_g)^{\ast} \omega) (v) = (ad_{g^{-1}} \circ \omega)(v)$ for any vector field $v \in \mathscr{X}(P)$, or, equivalently, for any $p \in P$, that
    \begin{equation}
        \big((R_g)^{\ast} \omega_{pg}\big)(v_p)=ad_{g^{-1}}\big(\omega_p(v_p)\big).
        \label{eq:proofconn1formequivariance1}
    \end{equation}
    Let us start by considering a horizontal vector field $w$, that is $w_p \in \mathcal{H}_pP$ for every $p \in P$. In this case the right-hand side of Equation (\ref{eq:proofconn1formequivariance1}) is zero since $\omega$ annihilates horizontal vector fields. The left-hand side instead reads
    \begin{equation}
        \big((R_g)^{\ast} \omega_{pg}\big) (w_p) = \omega_{pg}\big((R_g)_{\ast}(w_p)\big)=0,
        \label{eq:proofconn1formequivariance2}
    \end{equation}
    since the horizontal distribution $\mathcal{H}P$ is $G$-invariant by definition, hence $(R_g)_{\ast}w_p \in \mathcal{H}_{pg}P$ is again a horizontal tangent vector and is annihilated by $\omega$. \\
    Let now $v$ be a vertical vector field, that is $v_p \in \mathcal{V}_pP \subset T_pP$ for every $p$. Since $\zeta_p \colon \mathfrak{g} \to \mathcal{V}_pP$ is an isomorphism between the algebra and the vector subspaces for each point, we can always write $v_p=\zeta_p(X)$ for some Lie algebra element $X\in \mathfrak{g}$. Hence, for any vertical vector $v_p=\zeta_p(X)$, from Lemma \ref{lem:Rgstaronzeta}, we have that:
    \begin{equation}
        \begin{split}
            &\big((R_g)^{\ast} \omega_{pg}\big)(v_p) = \omega_{pg}\Big((R_g)_{\ast}\big(\zeta_p(X)\big)\Big)= \omega_{pg}\Big(\zeta_{pg}\big(ad_{g^{-1}}(X)\big)\Big)= \\
            &=  ad_{g^{-1}}(X) = ad_{g^{-1}}\Big(\omega_p\big(\zeta_p(X)\big)\Big) = ad_{g^{-1}}\big(\omega_p(v_p)\big).
        \end{split}
        \label{eq:proofconn1formequivariance3}
    \end{equation}
\end{proof}
For completeness, we note that one could alternatively start with a $1$-form $\omega \in \Omega^1(P;\mathfrak{g})$ satisfying $G$-equivariance as in Proposition \ref{pro:theconnectiononeformisequivariant} and such that $\omega\big(\zeta(X)\big) = X$, and, from this, define a connection, i.e. a $G$-invariant distribution $\mathcal{H}P$, as $\mathcal{H}P = \ker{\omega}$.

\subsection{Gauge fields}

Dealing with principal bundles, we always have a local section $s_{\alpha} \colon U_{\alpha} \to \pi^{-1}(U_{\alpha})$ for every open subset $U_{\alpha} \subset M$. This leads to the following definition.
\begin{definition}
    Taking the pullback of the connection $1$-form $\omega \in \Omega^1(P;\mathfrak{g})$ by local sections, defines the following $\mathfrak{g}$-valued $1$-forms on $U_{\alpha}$:
    \begin{equation}
        A_{\alpha} = s_{\alpha}^{\ast} \omega \in \Omega^1 (U_{\alpha}; \mathfrak{g}),
        \label{eq:gaugedef}
    \end{equation}
    which are called \textsl{gauge fields}.
    \label{def:gaugefield}
\end{definition}
We recall here that a vector field $v \in \mathscr{X}(G)$, with $G$ a Lie group, is called left-invariant if $(L_g)_{\ast}v_e=v_g$ for all $g \in G$, where $L_g \colon G \to G, h \mapsto L_g(h) = gh$. The vector space of left-invariant vector fields defines the Lie algebra $\mathfrak{g} \cong T_eG$ associated to the Lie group $G$.\\
\begin{definition}
    The \textsl{Maurer-Cartan form} is the $\mathfrak{g}$-valued $1$-form $\omega^{MC} \in \Omega^1(G;\mathbb{g})$ defined by:
    \begin{equation}
        \omega^{MC}_g :=(L_{g^{-1}})_{\ast} \colon T_gG \to \mathfrak{g}\cong T_eG.
        \label{MCform}
    \end{equation}
    For $G$ a matrix group,
    \begin{equation}
        \omega^{MC}_g = g^{-1} \mathrm{d} g
        \label{eq:matrixgroupMCform}
    \end{equation}
    \label{def:MaurerCartanformclassicaldefinition}
\end{definition}
If $v$ is a left-invariant vector field, then $\omega^{MC}_g(v_g)= v_e$, since $\omega^{MC}_g(v_g)=(L_{g^{-1}})_{\ast}(v_g) = (L_{g^{-1}})_{\ast}\big((L_g)_{\ast} v_e\big) = (L_{g^{-1}} \circ L_g)_{\ast}(v_e)=v_e$. This tells us that $\omega^{MC}_e \colon T_eG \to \mathfrak{g}$ is the natural identification between $T_eG$ and $\mathfrak{g}$.
\begin{lemma}
    The Maurer-Cartan form is equivariant under the right action of $G$, that is
    \begin{equation}
        R_g^{\ast} \omega^{MC} = ad_{g^{-1}} \circ \omega^{MC}.
        \label{eq:MCformequivariance}
    \end{equation}
    \label{lem:MaurerCartanisequvariant}
\end{lemma}
\begin{proof}
    Let $\xi \in \mathscr{X}(G)$. By definition,
    \begin{equation}
        R_g^{\ast} \omega^{MC}_{hg} (\xi_h) = \omega^{MC}_{hg}\big((R_g)_{\ast}\xi_h\big) = (L_{g^{-1}h^{-1}} \circ R_g)_{\ast}(\xi_h).
        \label{eq:proviamochelaMCèequivarianteclassica}
    \end{equation}
    Left and right multiplications commute, i.e.
    \begin{equation}
        L_{g^{-1}h^{-1}} \circ R_g = L_{g^{-1}} \circ L_{h^{-1}} \circ R_g = L_{g^{-1}} \circ R_g \circ L_{h^{-1}},
        \label{eq:leftandrightmultiplicationscommute}
    \end{equation}
    and since $(L_{g^{-1}} \circ R_g)$ is the conjugation by $g^{-1}$, Equation (\ref{eq:proviamochelaMCèequivarianteclassica}) reads
    \begin{equation}
        R_g^{\ast} \omega^{MC}_{hg} (\xi_h) = (R_g \circ L_{g^{-1}})_{\ast} \circ (L_{h^{-1}})_{\ast}(\xi_h) = ad_{g^{-1}} \circ \omega^{MC}_h(\xi_h).
        \label{eq:finiamostadimostrazioneMCcalssica}
    \end{equation}
\end{proof}
We now show a lemma which will be needed in the following discussion.
\begin{lemma}
    The $G$-equivariant diffeomorphisms $g_{\alpha} \colon \pi^{-1}(U_{\alpha}) \to G$ satisfy 
    \begin{equation}
        g_{\alpha\ast}\big(\zeta_p(X)\big)= (L_{g_{\alpha}(p)})_{\ast}(X),
        \label{eq:lemmaL}
    \end{equation}
    for any $X \in \mathfrak{g}$ and with $\zeta(X)$ being the fundamental vector field. 
    \label{lem:itwasanexercise}
\end{lemma}
\begin{proof}
    From equivariance of $g_{\alpha}$ we have:
    \begin{equation}
        \begin{split}
            &g_{\alpha\ast}\big(\zeta_p(X)\big) = \frac{d}{dt}\big(g_{\alpha}(p e^{tX})\big) \Big|_{t=0} = \frac{d}{dt}\big(g_{\alpha}(p) e^{tX}\big) \Big|_{t=0} = \\
            &= \frac{d}{dt} (L_{g_{\alpha}(p)} e^{tX}) \Big|_{t=0}        =(L_{g_{\alpha}(p)})_{\ast}(X).
        \end{split}
        \label{eq:whatoncewasanexercise}
    \end{equation}
\end{proof}
\begin{proposition}
    The restriction of the connection one-form $\omega$ to $\pi^{-1}(U_{\alpha})$ agrees with
    \begin{equation}
        \omega_{\alpha} := ad_{g_{\alpha}^{-1}} \circ \pi^{\ast} A_{\alpha} + g_{\alpha}^{\ast} \omega^{MC},
        \label{eq:localconn}
    \end{equation}
    where $\omega^{MC}$ is the Maurer-Cartan form.
    \label{pro:restrictionoftheconnectiononeform}
\end{proposition}
\begin{proof}
    We divide the proof into two steps. First, we verify that the restriction of the connection $1$-form $\omega$ to $\pi^{-1}(U_{\alpha})$ agrees with $\omega_{\alpha}$ along the image of a local section. Next, we check that they also agree along the fibers. Since the sections select a specific point in each fiber, and the right action of the structure group allows us to move along the fibers, these two conditions together imply that $\omega$ and $\omega_{\alpha}$ must agree on all of $\pi^{-1}(U_{\alpha})$.\\
    What we have to prove, at each point $p \in P$, reads:
        \begin{equation}
            (\omega_{\alpha})_p = ad_{g_{\alpha}(p)^{-1}} \circ \pi^{\ast} (A_{\alpha})_{\pi(p)} + g_{\alpha}^{\ast} \omega^{MC}_{g_{\alpha}(p)}.
        \label{eq:localconnindexes}
        \end{equation}
    Recall that $g_{\alpha} \colon \pi^{-1}(U_{\alpha}) \to G$ is the $G$-equivariant diffeomorphism defining the trivialization map.
    
    Let us start by showing that $\omega$ and $\omega_{\alpha}$ agree on the image of a local section $s_{\alpha} \colon U_{\alpha} \to \pi^{-1}(U_{\alpha})$. Let $m \in U_{\alpha}$ and consider a (trivialized) point $p=s_{\alpha}(m)$ in the image of the local section. Any tangent vector $\xi_p$ at such a point $p$ can always be written as
    \begin{equation}
        \xi_p=(s_{\alpha} \circ \pi)_{\ast}(\xi_p) + \bar{\xi_p},
        \label{eq:dsumex}
    \end{equation}
    for a particular vertical vector $\bar{\xi}_p \in \mathcal{V}_pP$. Since for $p=s_{\alpha}(m)$ we have $g_{\alpha}(p)=e$, we can write:
    \begin{equation}
        \begin{split}
        &(\omega_{\alpha})_p(\xi_p) = ad_{g_{\alpha}(p)^{-1}}\big( (\pi^{\ast} s_{\alpha}^{\ast} \omega_{s_{\alpha}(\pi(p))})(\xi_p)\big) + (g_{\alpha}^{\ast} \omega^{MC}_{g_{\alpha}(p)})(\xi_p) = \\
        &= ad_e\big( (\pi^{\ast} s_{\alpha}^{\ast} \omega_p)(\xi_p) \big) + (g_{\alpha}^{\ast} \omega^{MC}_e)(\xi_p) = (\pi^{\ast} s_{\alpha}^{\ast} \omega_p)(\xi_p) + (g_{\alpha}^{\ast} \omega^{MC}_e)(\xi_p) = \\
        &= \omega_p\big( s_{\alpha\ast} \pi_{\ast}(\xi_p)\big) + \omega^{MC}_e \big(g_{\alpha\ast} (\xi_p) \big),
        \end{split}
        \label{eq:proofagree1}
    \end{equation}
    where we used $ad_e (X) =X$ for any $X \in \mathfrak{g}$.
    Inserting the expression for $\xi_p$ from Equation (\ref{eq:dsumex}), then:
    \begin{equation}
        \begin{split}
        &(\omega_{\alpha})_p(\xi_p) =  \omega_p\Big( s_{\alpha\ast} \pi_{\ast}\big((s_{\alpha} \circ \pi)_{\ast}(\xi_p)\big)\Big) + \omega^{MC}_e \big(g_{\alpha\ast} (\bar{\xi}_p) \big) = \\
        &= \omega_p\big( (s_{\alpha} \circ \pi \circ s_{\alpha} \circ \pi)_{\ast}(\xi_p)\big) + \omega^{MC}_e \big(g_{\alpha\ast} (\bar{\xi}_p) \big) = \\
        &= \omega_p\big( s_{\alpha\ast} \pi_{\ast}(\xi_p)\big) +  \omega^{MC}_e \big(g_{\alpha\ast} (\bar{\xi}_p) \big),
        \end{split}
        \label{eq:proofagree2}
    \end{equation}
    where we used that $\pi \circ s_{\alpha} = id$ and that, since $g_{\alpha} \big(s_{\alpha}(m)\big)=e$ for all $m \in M$, $(g_{\alpha} \circ s_{\alpha})_{\ast}=0$. \\
    Let $\bar{\xi}_p = \zeta_p(X)$. Using Lemma \ref{lem:itwasanexercise}, we have
    \begin{equation}
        \begin{split}
            &\omega^{MC}_{e}\big(g_{\alpha\ast} (\bar{\xi}_p) \big) = \omega^{MC}_e\big((L_{g_{\alpha}(p)})_{\ast}(X)\big) = (L_e)_{\ast}(L_e)_{\ast}(X) = X = \omega_p(\bar{\xi}_p),
        \end{split}
        \label{eq:usingtheoncewasexercise}
    \end{equation}
    so that we can write Equation (\ref{eq:proofagree2}) as    
    \begin{equation}
        \begin{split}
            &(\omega_{\alpha})_p(\xi_p) = \omega_p\big( s_{\alpha\ast} \pi_{\ast}(\xi_p)\big) +  \omega^{MC}_e \big(g_{\alpha\ast} (\bar{\xi}_p) \big) =\\
            &= \omega_p\big( s_{\alpha\ast} \pi_{\ast}(\xi_p)\big) +  \omega_p(\bar{\xi}_p) = \omega_p(\xi_p),
        \end{split}
        \label{eq:proofagree3}
    \end{equation}
    showing that $(\omega_{\alpha})_p$ and $\omega_p$ agree on any $\xi_p \in T_pP$ whit $p$ in the image of a local section.
    
    Next, we show that they transform in the same way under the right action of the structure group $G$, that is they agree along fibers. Still, let $p = s_{\alpha}(m)$. Since $R_g\big(g_{\alpha}(p)\big)= g_{\alpha}\big(R_g(p)\big)$ and $\pi\big(R_g(p)\big) =\pi(p)$ we have:
    \begin{equation}
    \begin{split}
        &R_g^{\ast}(\omega_{\alpha})_{pg}=  ad_{g_{\alpha}(pg)^{-1}} \circ R_g^{\ast}\pi^{\ast} s_{\alpha}^{\ast} \omega_{s_{\alpha}(\pi(pg))}  + R_g^{\ast}g_{\alpha}^{\ast}\omega^{MC}_{g_{\alpha}(pg)}) =\\
        &= ad_{g^{-1}g_{\alpha}(p)} \circ \pi^{\ast} s_{\alpha}^{\ast} \omega_{s_{\alpha}(\pi(p))} + g_{\alpha}^{\ast}R_g^{\ast} \omega^{MC}_{g_{\alpha}(p)g} =  \\
        &= ad_{g^{-1}g_{\alpha}(p)} \circ \pi^{\ast} s_{\alpha}^{\ast} \omega_{s_{\alpha}(\pi(p))} + g_{\alpha}^{\ast}( ad_{g^{-1}} \circ \omega^{MC}_{g_{\alpha}(p)}) = \\
        &= ad_{g^{-1}} \circ \pi^{\ast} s_{\alpha}^{\ast} \omega_p +  ad_{g^{-1}} \circ  g_{\alpha}^{\ast}\omega^{MC}_e = ad_{g^{-1}} \circ (\pi^{\ast} s_{\alpha}^{\ast} \omega_p + g_{\alpha}^{\ast}\omega^{MC}_e) = \\
        &= ad_{g^{-1}} \circ (\omega_{\alpha})_p,
    \end{split} 
    \label{eq:righproof}
    \end{equation}
    where we used $G$-equivariance property for the Maurer-Cartan form $R_g^{\ast} \omega^{MC} = ad_{g^{-1}} \circ \omega^{MC}$.
\end{proof}
Considering a non-empty overlap $U_{\alpha\beta}$, we have that $\omega_{\alpha}=\omega_{\beta}$ on $\pi^{-1}(U_{\alpha\beta})$, since $\omega_{\alpha}$ and $\omega_{\beta}$ agree with the restriction of $\omega$ on $\pi^{-1}(U_{\alpha})$ and $\pi^{-1}(U_{\beta})$ respectively and $\omega$ is defined on the whole total space $P$. We can now show the relation between different gauge fields defined on $U_{\alpha\beta}$.
\begin{proposition}
    Let $U_{\alpha\beta} = U_{\alpha} \cap U_{\beta} \neq \emptyset$ and $s_{\alpha}$ and $s_{\beta}$ be two local sections defined on $U_{\alpha\beta}$. Then $A_{\alpha} = s_{\alpha}^{\ast}\omega$ and $A_{\beta} = s_{\beta}^{\ast}$ are related on $U_{\alpha\beta}$ as
    \begin{equation}
        A_{\alpha} = ad_{g_{\alpha\beta}} \circ A_{\beta} + g_{\beta\alpha} ^{\ast} \omega^{MC},
        \label{eq:relateAalphaandAbeta1}
    \end{equation}
    \label{pro:recoveringgaugetransformations}
\end{proposition}
\begin{proof}
    Let $m \in U_{\alpha\beta}$, $p=s_{\alpha}(m) \in \pi^{-1}(U_{\alpha\beta})$. By Proposition \ref{pro:restrictionoftheconnectiononeform} we have
    \begin{equation}
        \begin{split}
        &(A_{\alpha})_m=s_{\alpha}^{\ast} \omega_{s_{\alpha}(m)} =  s_{\alpha}^{\ast} (\omega_{\alpha})_{s_{\alpha}(m)} = s_{\alpha}^{\ast} (\omega_{\beta})_{s_{\alpha}(m)} = \\
        &= ad_{g_{\beta}(s_{\alpha}(m))^{-1}} \circ s_{\alpha}^{\ast} \pi^{\ast} (A_{\beta})_{\pi(s_{\alpha}(m))} + s_{\alpha}^{\ast} g_{\beta}^{\ast} \omega^{MC}_{g_{\beta}(s_{\alpha}(m))} = \\
        &= ad_{g_{\alpha\beta}(m)} \circ (\pi \circ s_{\alpha})^{\ast} (A_{\beta})_m + (g_{\beta} \circ s_{\alpha})^{\ast} \omega^{MC}_{g_{\beta\alpha}(m)} = \\
        &=ad_{g_{\alpha\beta}(m)} \circ (A_{\beta})_m + g_{\beta\alpha} ^{\ast} \omega^{MC}_{g_{\beta\alpha}(m)},
        \end{split} 
        \label{eq:relateproof}
    \end{equation}
    where we used that $g_{\beta\alpha}(m) = \bar{g}_{\beta\alpha}\big(s_{\alpha}(m)\big) = g_{\beta}\big(s_{\alpha}(m)\big)g_{\alpha}\big(s_{\alpha}(m)\big)^{-1} = (g_{\beta} \circ s_{\alpha})(m)$ and that, by Lemma \ref{lem:cocycleconditionsfortransitionfunction}, $g_{\beta\alpha}^{-1} = g_{\alpha\beta}$.
    Summing things up, without specifying the points, we can write:
    \begin{equation}
        A_{\alpha} = ad_{g_{\alpha\beta}} \circ A_{\beta} + g_{\beta\alpha} ^{\ast} \omega^{MC},
        \label{eq:relate1}
    \end{equation}
    as desired.
\end{proof}
For matrix groups, Equation (\ref{eq:relateAalphaandAbeta1}) becomes the more familiar:
\begin{equation}
    A_{\alpha} = g_{\alpha\beta} A_{\beta} g_{\alpha\beta}^{-1} - \mathrm{d} g_{\alpha\beta} g_{\alpha\beta}^{-1},
    \label{eq:relation1}
\end{equation}
which is exactly the gauge transformation law for mediators of the interactions in gauge theory. Fixing the gauge means then chosing a local trivialization, that is a local section.\\
What we did is we have defined gauge fields $A_{\alpha} \in \Omega^1(U_{\alpha}; \mathfrak{g})$ starting from the globally defined connection $1$-form $\omega \in \Omega^1(P;\mathfrak{g})$. Conversely, given a family of $1$-forms $A_{\alpha} \in \Omega^1(U_{\alpha}; \mathfrak{g})$ satisfying Proposition \ref{pro:recoveringgaugetransformations} on overlaps $U_{\alpha\beta}$, we
can construct a globally defined $\omega \in \Omega^1(P; \mathfrak{g})$ by Proposition \ref{pro:restrictionoftheconnectiononeform} such that $\omega$ is the connection $1$-form of a connection on $P$.

\section{The frame bundle} \label{theframebundle}

An important example of a principal bundle is the frame bundle $\pi \colon FM \to M$ of a smooth manifold $M$. Its fiber over each orbit $m \in M$ is the set of all bases of the tangent space $T_mM$. The structure group is $GL(n, \mathbb{R})$, with $\text{dim}(M) = n$: its action along fibers is to be interpreted as a basis change for $T_mM$.  We may equivalently view the fiber $FM_m$ over $m \in M$ as the space of all linear isomorphisms between $\mathbb{R}^n$ and $T_mM$: an element $u \in \pi^{-1}(m) \subset FM$ is an isomorphism $u \colon \mathbb{R}^n \xrightarrow[]{\sim} T_mM$, that is an assignment of a particular tangent vector to an $n$-tuple of components in $\mathbb{R}^n$. \\
A crucial feature of the frame bundle is the existence of a canonical $1$-form with values in the fundamental representation of $GL(n,\mathbb{R}^n)$, defined as follows.
\begin{definition}
    There exists a $1$-form $\theta \in \Omega^1(FM;\mathbb{R}^n)$, called the \textsl{canonical form} (or \textsl{soldering form}), defined on vector fields $\xi \in \mathscr{X}(FM)$ as
    \begin{equation}
        \theta_u (\xi_u) := u^{-1} (\pi_{\ast} \xi_u) \in \mathbb{R}^n,
        \label{eq:definitioncanonicalform}
    \end{equation}
    viewing $u \in FM$ as a linear isomorphism $u \colon \mathbb{R}^n \xrightarrow[]{\sim} T_{\pi(u)}M$. 
    \label{def:canonicalsolderingform}
\end{definition}
This means that the value of the canonical form at each frame $u \in FM$ for any $\xi_u \in T_uFM$ is defined as the components of the projection $\pi_{\ast}\xi_u \in \mathbb{R}^n$. By construction, the canonical form has the following properties.
\begin{proposition}
    The canonical form $\theta \in \Omega^1(FM;\mathbb{R}^n)$ is $GL(n;\mathbb{R})$-equivariant, that is
    \begin{equation}
        R_g^{\ast} \theta = g^{-1} \circ \theta
        \label{eq:canonicalformequivariance}
    \end{equation}
    for all $g \in GL(n,\mathbb{R})$, and it horizontal, that is
    \begin{equation}
        \theta(\xi) = 0
        \label{eq:canonicalformhorizontality}
    \end{equation}
    for any vertical vector field $\xi \in \mathcal{V}FM$. Hence, we write $\theta \in \Omega^1_{hor}(FM;\mathbb{R}^n)^{GL(n,\mathbb{R})}$.
    \label{pro:thecanonicalformishorizontalandequivarint}
\end{proposition}
\begin{proof}
    The fact that $\theta$ is horizontal is obvious from the definition \ref{def:canonicalsolderingform} of canonical form, since $\pi_{\ast}$ annihilates vertical vectors.\\
    $GL(n,\mathbb{R})$-equivariance follows from the fact that, viewing frames as linear isomorphisms $u \colon \mathbb{R}^n \xrightarrow[]{\sim} T_{\pi(u)}M$, the right $GL(n, \mathbb{R})$-action gives:
    \begin{equation}
        \begin{split}
            &R_g^{\ast}\theta_{ug}(\xi_{u})= \theta_{ug}\big((R_g)_{\ast}\xi_u\big)=(ug)^{-1} \big(\pi_{\ast} (R_g)_{\ast}\xi_u\big) = \\
            &= g^{-1} \circ u^{-1}(\pi_{\ast}\xi_u) = g^{-1} \circ \theta_u(\xi_u),
        \end{split}
        \label{eq:proofcanonicalformisequivariant}
    \end{equation}
    where we used that $\pi \circ R_g = \pi$ and that $GL(n,\mathbb{R})$ acts on $\mathbb{R}^n$ via the fundamental representation.
\end{proof}

\subsection{$G$-structures} \label{Gstructures}
Reductions of the frame bundle to subgroups of $GL(n;\mathbb{R})$ are called $G$-structures. They provide simple examples of geometric structures on manifolds.
        Riemannian spin structures may be similarly interpreted as G-structures corresponding to the homomorphism $i \colon Spin(m,\mathbb{R})\to GL(m,\mathbb{R})$. This is an important example in which the structure group is not a subgroup of $GL(m,\mathbb{R})$ but a covering of such a subgroup.
\begin{definition}
    Let $H \subset GL(n,\mathbb{R})$ a Lie subgroup and $M$ be a smooth manifold, with $\text{dim}(M) = n$. A $G$\textsl{-structure} on $M$ is a reduction $\iota \colon P \to FM$ of the frame bundle of $M$ to a structure group $H$. One may call such a reduction an $H$-structure.
    \label{def:Gstructure}
\end{definition}
The above is saying that $P \to M$ is a principal $H$-bundle and $\iota \colon P \to FM$ is a principal bundle morphism with respect to the inclusion $i \colon H \to GL(n,\mathbb{R})$ as in Definition \ref{def:reduction}.
\begin{exmp}
    A Riemannian structure is the reduction of $FM$ to orthonormal frames, i.e. to $H = O(n, \mathbb{R}) \subset GL(n,\mathbb{R})$. The group $GL(n, \mathbb{R})$ acts transitively on the space of all inner products on $\mathbb{R}^n$ and the isotropy group of the standard inner product is $O(n)$, so $GL(n,\mathbb{R})/O(n)$ is the space of all inner products on $R^n$. Hence, the associated bundle $FM/H=FM/O(n)$ is the bundle of all inner products on the tangent spaces of $M$, i.e. smooth sections of this bundle are exactly Riemannian metrics on $M$.
    \label{ex:exampleofGstructure}
\end{exmp}
Equivalently, one may characterize an $H$-structure as a principal $H$-bundle $P$ together with a $1$-form $\Theta \in \Omega^1_{hor}(P; \mathbb{R}^n)^{H}$, that is horizontal and $H$-equivariant, i.e. satisfying the analogues of properties of the canonical form in Definition \ref{def:canonicalsolderingform} for elements $g \in H$.
\begin{proposition}
    A $G$-structure $\pi \colon P \to M$ on $M$ with structure group $H \subset GL(n, \mathbb{R})$ is a principal $H$-bundle $\pi \colon P \to M$ together with an $H$-equivariant and horizontal $1$-form $\Theta \in \Omega^{1}_{hor}(P;\mathbb{R}^n)^H$. The converse is also true: a principal $H$-bundle $\pi \colon P \to M$ equipped with a $\Theta \in \Omega^1_{hor}(P;\mathbb{R}^n)^H$ is a $G$-structure on $M$.
    \label{pro:thepowerofGstructures}
\end{proposition}
\begin{proof}
    Let $\pi \colon P \to M$ be a reduction of the frame bundle $\tilde{\pi} \colon FM \to M$ to the structure group $H \subset GL(n,\mathbb{R})$, given by the principal bundle morphism $\iota \colon P \to FM$ with respect to the inclusion $i \colon H \to GL(n, \mathbb{R})$. By definition of reduction, we know that $\tilde{\pi} \circ \iota = \pi$ and that $\iota(p h)=\iota(p)i(h)$ for all $p \in P$ and $h \in H$. We know that the frame bundle is always equipped with the canonical form $\theta \in \Omega^1_{hor}(FM; \mathbb{R}^n)^{GL(n;\mathbb{R})}$, defined as $\theta_u(\xi_u):=u^{-1}(\tilde{\pi}_{\ast}\xi_u)$, for any $\xi \in \mathscr{X}(FM)$.\\
    Consider the $1$-form
    \begin{equation}
        \Theta := \iota^{\ast}\theta,
        \label{eq:Thetaasthepullbackbyiota}
    \end{equation}
    which is clearly is an element of $\Omega^1(P;\mathbb{R}^n)$. We have to prove it is horizontal and $H$-equivariant.\\
    It is horizontal since for any vertical vector $\xi_p \in \mathcal{V}_pP$  we have:
    \begin{equation}
        \begin{split}
            &\Theta_p(\xi_p) = \iota^{\ast} \theta_{\iota(p)} (\xi_p) = \theta_{\iota(p)}(\iota_{\ast} \xi_p)_{\iota(p)}= \iota(p)^{-1}(\tilde{\pi}_{\ast}\iota_{\ast} \xi_p) = \\
            &= \iota(p)^{-1}\big((\tilde{\pi} \circ \iota)_{\ast} \xi_p\big) = \iota(p)^{-1}\big(\pi_{\ast} \xi_p\big) = 0
        \end{split}
        \label{eq:steptoshowThetavertical}
    \end{equation}
    with $FM \ni \iota(p) \colon \mathbb{R}^n \xrightarrow[]{\sim} T_{\tilde{\pi}(\iota(p))}M = T_{\pi(p)}M$.\\
    It is $H$-equivariant since, knowing that $R_g^{\ast}\theta=g^{-1}\circ \theta$, we can write;
    \begin{equation}
        \begin{split}
            &R_h^{\ast} \Theta_{ph} = R_h^{\ast} (\iota^{\ast} \theta_{\iota(ph)}) = (\iota \circ R_h)^{\ast} \theta_{\iota(ph)} = (R_{i(h)} \circ \iota)^{\ast}\theta_{\iota(p)i(h)} =\\
            &= \iota^{\ast}R_{i(h)}^{\ast}\theta_{\iota(p)i(h)} = i(h)^{-1} \circ \iota^{\ast} \theta_{\iota(p)} = i(h)^{-1} \circ \Theta_p,
        \end{split}
        \label{eq:trytoproveThetaHequivariant}
    \end{equation}
    where we have used that $\iota \circ R_h = R_{i(h)} \circ \iota$. Summing things up, we have shown that a $G$-structure on $M$ is a principal $H$-bundle equipped with a $\Theta = \iota^{\ast}\theta \in \Omega^1_{hor}(P;\mathbb{R}^n)^H$.
    
    Conversely, let us now start from a principal $H$-bundle $\pi \colon P \to M$ equipped with an horizontal and $H$-equivariant one-form $\Theta \in \Omega^1_{hor}(P;\mathbb{R}^n)^H$. We want to show that $\pi \colon P \to M$ is a reduction of the frame bundle $\pi^{\prime} \colon FM \to M$. Since $H \subset GL(n,\mathbb{R})$ is a Lie subgroup, we obtain the natural inclusion $i \colon H \to GL(n,\mathbb{R})$. We are left with the task of finding a principal bundle morphism $\iota \colon P \to FM$, i.e. we need a smooth map $\iota$ assigning a frame $u =\iota(p) \in FM$ to each point $p \in P$ in a unique way. \\
    To do this, we have to use $\Theta \in \Omega^1_{hor}(P;\mathbb{R}^n)$. At each point $p \in P$, $\Theta_p$ is a map $\Theta_p \colon T_pP \to \mathbb{R}^n$. This is not an isomorphism, since it has non-trivial kernel, i.e. the vertical subspace $\mathcal{V}_pP$. We then need to consider the restriction
    \begin{equation}
        \Theta_p\big| \colon T_pP/\mathcal{V}_pP \cong \mathcal{H}_pP \xrightarrow[]{\sim} \mathbb{R}^n,
        \label{eq:restrictiondoesthejob}
    \end{equation}
    which is an isomorphism since we have removed the kernel from the domain. We can then observe that $T_pP/\mathcal{V}_pP \cong T_{\pi(p)}M$ since we can consider $\pi_{\ast} \colon T_pP \to T_{\pi(p)}M$, which is surjective, and remove its kernel (which is exactly $\mathcal{V}_pP$), to obtain an isomorphism $T_pP/\mathcal{V}_pP \cong T_{\pi(p)}M$. This implies that the restriction of $\Theta$ as written above identifies an isomorphism $T_{\pi(p)}M \cong \mathbb{R}^n$. \\
    Thus, we can define the principal bundle morphism $\iota \colon P \to FM$, inducing it from $\Theta_p\big|$, as:
    \begin{equation}
        \begin{split}
            \iota \colon &P \to FM,\\
            &p \mapsto \iota(p) := (\Theta_p\big|)^{-1} \colon \mathbb{R}^n \xrightarrow[]{\sim} T_{\pi(p)}M.
        \end{split}
        \label{eq:Ithinkthisisthecorrectprocedure}
    \end{equation}
    This, by construction, means that $\iota(p) \in FM$ is a basis for $T_{\tilde{\pi}(\iota(p))}M = T_{\pi(p)}M$, that is $\pi^{\prime}(\iota(p))=\pi(p)$ for all $p \in P$.\\
    In order to show $\iota$ is actually the principal bundle morphism characterizing the reduction of the frame bundle $FM$ to the principal $H$-bundle $P$, we still have to prove $\iota$ is equivariant. This can be done using $H$-equivariance of $\Theta$, namely $R_h^{\ast} \Theta_{ph} = h^{-1} \circ \Theta_p = i(h)^{-1} \circ \Theta_p$ for all $h \in H$. This still holds for the restrictions, so that we can write, for any $\xi_p \in \mathcal{H}_pP$, that: $R_h^{\ast}\Theta_{ph}\big|(\xi_p) = \Theta_{ph}\big|\big((R_h)_{\ast}\xi_p\big)= i(h)^{-1} \circ \Theta_p\big|(\xi_p)$. Renaming $(R_h)_{\ast}\xi_p=\xi^{\prime}_{ph}$ such that $\xi_p = (R_{h^{-1}})_{\ast}\xi^{\prime}_{ph}$ we have that:
    \begin{equation}
        \Theta_{ph}\big|(\xi^{\prime}_{ph}) = i(h)^{-1} \circ \Theta_p\big|((R_{h^{-1}})_{\ast}\xi^{\prime}_{ph}),
        \label{eq:fiberpreservingmorphismproof1}
    \end{equation}
    that is
    \begin{equation}
        \Theta_{ph}\big|= i(h)^{-1} \circ \Theta_p\big| \circ (R_{h^{-1}})_{\ast}.
        \label{eq:fiberpreservingmorphismproof2}
    \end{equation}
    From (\ref{eq:Ithinkthisisthecorrectprocedure}) and (\ref{eq:fiberpreservingmorphismproof2}):
    \begin{equation}
        \begin{split}
            &\iota(ph)^{-1} (\pi_{\ast}\xi_{ph}) = i(h)^{-1} \circ \iota(p)^{-1}\Big(\pi_{\ast}\big((R_{h^{-1}})_{\ast}\xi_{ph}\big)\Big) =\\
            &=i(h)^{-1} \circ \iota(p)^{-1}\big((\pi \circ R_{h^{-1}})_{\ast}\xi_{ph}\big) = i(h)^{-1} \circ \iota(p)^{-1}\big(\pi_{\ast}\xi_{ph}\big),
        \end{split}
        \label{eq:fiberpreservingmorphismproof3}
    \end{equation}
    whence
    \begin{equation}
        \iota(ph) = \iota(p) i(h),
        \label{eq:fiberpreservingmorphismproof4}
    \end{equation}
    completing the proof.
\end{proof}

\subsection{Cartan geometries: the affine model and its reductions} \label{cartangeometries}

Given a connection $1$-form $\gamma \in \Omega^1(FM; \mathfrak{gl}(n, \mathbb{R}))$ on the frame bundle $\pi \colon FM \to M$, with $\text{dim}(M) = n$, we may consider
\begin{equation}
    \omega := \theta + \gamma \in \Omega^1(FM; \mathbb{R}^n \oplus \mathfrak{gl}(n, \mathbb{R})),
    \label{eq:sumofcanonicalandprincipal}
\end{equation}
where $\theta \in \Omega^1_{hor}(FM;\mathbb{R}^n)^{GL(n,\mathbb{R})}$ is the canonical form in Definition \ref{def:canonicalsolderingform}.
The $1$-form $\omega$ reproduces the generators of fundamental vector fields: in fact, $\gamma$ does so, being a principal connection, and the canonical form $\theta$ is horizontal. Moreover, since the kernel of $\theta$ in a point is the vertical subspace and $\gamma$ is injective on the vertical bundle, it follows that for each $u \in FM$, the restriction of $\omega$ to $T_uFM$ is injective, and therefore a linear isomorphism.\\
In this picture, we can nicely describe the affine $n$-space $A^n$ as the homogeneous model of such a structure. As a set, $A^n = \mathbb{R}^n$ and the group $A(n, \mathbb{R})$ of affine motions is the group of all maps from $\mathbb{R}^n$ to itself, which are of the form
\begin{equation}
    x \mapsto Ax+b, \text{ for } A \in GL(n, \mathbb{R}), b \in \mathbb{R}^n.
    \label{eq:defintionofgroupofaffinemotions}
\end{equation}
This group can be seen as the semidirect product
\begin{equation}
    A(n,\mathbb{R}) = \mathbb{R}^n \rtimes GL(n,\mathbb{R}),
    \label{eq:affinesemidirect}
\end{equation}
meaning that $A(n,\mathbb{R}) = \mathbb{R}^n \times GL(n,\mathbb{R})$ as a set and its multiplication is given by $(b, A)(b^{\prime}, A^{\prime}) = (b + Ab^{\prime}, AA^{\prime})$ for any $b,b^{\prime} \in \mathbb{R}^n$ and $A,A^{\prime} \in GL(n,\mathbb{R})$. Equivalently, viewing $A^n$ as the affine hyperplane $x_1 = 1$ in $\mathbb{R}^{n+1}$, the affine motions are exactly the elements of $GL(n + 1, \mathbb{R})$ which map this affine hyperplane to itself, i.e.,
\begin{equation}
    A(n, \mathbb{R}) = \big\{  \begin{pmatrix} 1 & 0 \\ b & A  \end{pmatrix} \; \big| \; A \in GL(n, \mathbb{R}), b \in \mathbb{R}^n \big\} \subset GL(n+1, \mathbb{R}).
    \label{eq:affinehyperplaneshit}
\end{equation}
The group $A(m, \mathbb{R})$ acts transitively on $A^m$, and the isotropy subgroup of the first unit vector, i.e., what leaves $A^m$ invariant, is the subgroup of all elements with $b = 0$, so 
\begin{equation}
    A^n \cong A(n,\mathbb{R})/GL(n,\mathbb{R}).
    \label{eq:affinemspace}
\end{equation}
The Lie algebra $\mathfrak{a}(n, \mathbb{R})$ associated to the group of affine motions is
\begin{equation}
    \mathfrak{a}(n, \mathbb{R}) = \big\{  \begin{pmatrix} 0 & 0 \\ B & X  \end{pmatrix} \;\big|\; X \in \mathfrak{gl}(n, \mathbb{R}), B \in \mathbb{R}^n \big\} \cong \mathbb{R}^m \oplus \mathfrak{gl}(m, \mathbb{R}),
    \label{eq.defliealgebraassociatedtoaffinegroup}
\end{equation}
where $\cong$ indicates a vector space isomorphism. Thus, since the adjoint action of a matrix group is given by conjugation, one can see that the splitting $\mathbb{R}^n \oplus \mathfrak{gl}(n, \mathbb{R})$ is invariant under the restriction of the adjoint action to $GL(n,\mathbb{R})$. In particular, the action of $GL(n,\mathbb{R})$ on $\mathfrak{a}(m, \mathbb{R})$ is the direct sum of the standard representation and the adjoint action:  it acts on the translational part ($\mathbb{R}^n$) via the standard representation and on the linear part ($\mathfrak{gl}(m, \mathbb{R})$) using the adjoint representation. We say that splitting $\mathfrak{a}(m, \mathbb{R})= \mathbb{R}^m \oplus \mathfrak{gl}(m, \mathbb{R})$ is $GL(m, \mathbb{R})$-invariant. \\
Accordingly, the $1$-form $\omega$ as in Equation (\ref{eq:sumofcanonicalandprincipal}) is $GL(n,\mathbb{R})$-equivariant. In particular, $\omega$ can be viewes as an element of $\Omega^1(FM;\mathfrak{a}(n,\mathbb{R}))$, and its properties can be summarized as follows:
\begin{equation}
    \begin{split}
        &R_g^{\ast} \omega = ad_{g^{-1}} \circ \omega \text{ for all } g \in GL(n,\mathbb{R});\\
        &\omega\big(\zeta(X)\big) = X, \text{ for all } X \in \mathfrak{gl}(n, \mathbb{R})\subset \mathfrak{a}(n, \mathbb{R});\\
        &\omega_u \colon T_uFM \to \mathfrak{a}(n, \mathbb{R}) \text{ is a linear isomorphism for all } u \in FM,
    \end{split}
    \label{eq:propertiesofomegathesum}
\end{equation}
where $ad$ is the adjoint action of $A(n,\mathbb{R})$ (restricted to $GL(n,\mathbb{R})$) on its Lie algebra according to the splitting, as discussed.

The natural projection $\pi \colon A(n,\mathbb{R}) \to A(n, \mathbb{R})/GL(n,\mathbb{R})$ is a principal $GL(n,\mathbb{R})$-bundle. One can introduce Maurer-Cartan form $\omega^{MC} \in \Omega^1(A(n,\mathbb{R});\mathfrak{a}(n,\mathbb{R}))$ on the affine Lie group, and split $\omega^{MC} = \theta + \gamma$ according to $\mathfrak{a}(n,\mathbb{R}) = \mathbb{R}^n \oplus \mathfrak{gl}(n, \mathbb{R})$. Since this splitting is $GL(n,\mathbb{R})$-invariant, both $\theta$ and $\gamma$ are $GL(n, \mathbb{R})$-equivariant forms.
The form $\theta$ associates to each element $g \in A(n,\mathbb{R})$ a linear isomorphism $T_{\pi(g)}A^n \to \mathbb{R}^n$, i.e. a frame, and thus identifies $A(n, \mathbb{R})$ with the frame bundle $FA^m$. 
\\
Therefore, returning to a general manifold $M$, 
the analogue of the principal $GL(n, \mathbb{R})$-bundle $A(n,\mathbb{R}) \to A^n$ is the frame bundle $FM \to M$, 
and the $1$-form $\omega = \theta +  \gamma$ as in (\ref{eq:sumofcanonicalandprincipal}) is an analog of the Maurer-Cartan form $\omega_{MC}$ on $A(n, \mathbb{R})$. In particular, affine structures are to be interpreted as curved analogues (in the Cartan-geometric sense) of the homogeneous model $A(n,\mathbb{R}) \to A^n$. Moreover, the following lemma is given. 
\begin{lemma}
    Let $\pi \colon P \to M$ be an affine structure, i.e., a principal $GL(n,\mathbb{R})$-bundle together with a $1$-form $\omega \in \Omega^1(P;\mathfrak{a}(n,\mathbb{R}))$ satisfying the properties in (\ref{eq:propertiesofomegathesum}). Then, $P \cong FM$ as principal bundles.
    \label{lem:tosumupaffinestructures}
\end{lemma}
\begin{proof}
    At each point $p \in P$, the kernel of the $\mathbb{R}^n$ component of $\omega$ is exactly the vertical subspace $\mathcal{V}_pP$. This is guaranteed by the last two properties in Equation (\ref{eq:propertiesofomegathesum}). Its restriction to $T_pP/\mathcal{V}_pP$ gives a linear isomorphism $T_{\pi(p)}M \cong \mathbb{R}^n$. The latter corresponds to a unique frame $u \in FM$, so that we have, by the first property in (\ref{eq:propertiesofomegathesum}), an isomorphism of principal fiber bundles $\phi \colon P \to FM$. In particular, the $\mathbb{R}^n$ component of $\omega$ is the pullback by $\phi$ of the canonical form $\theta$ and the $\mathfrak{gl}(n,\mathbb{R})$ component is a principal connection on $P$.
\end{proof}
Let now $\pi \colon P \to M$ be an $H$-structure, i.e. a reduction $\iota \colon P \to FM$ of the frame bundle to the structure group $H$ with respect to the inclusion $i \colon H \to GL(n,\mathbb{R})$, and let the corresponding infinitesimal homomorphism $i^{\prime} \colon \mathfrak{g} \to \mathfrak{gl}(n,\mathbb{R})$ be injective. 
A connection on an $H$-structure $P$ is a connection $1$-form $\gamma \in \Omega^1(P;\mathfrak{h})$ on the principal $H$-bundle $P \to M$. \\
Consider the affine extension 
\begin{equation}
    B := \mathbb{R}^m \rtimes H
    \label{eq:affineextension}
\end{equation}
of $H$. This means that $B = \mathbb{R}^m \times H$ as a set and the multiplication is given by $(x, h)(y, h^{\prime}) = (x +i(h)(y), hh^{\prime})$ for every $x,y \in \mathbb{R}^n$ and $h,h^{\prime} \in H$. If $H$ is a closed subgroup of $ GL(n, \mathbb{R})$, then $B$ is a closed subgroup of $A(n, \mathbb{R})$, and in general there is an obvious inclusion $j \colon B \to A(m, \mathbb{R})$ such that $j^{\prime} \colon \mathfrak{b} \to \mathfrak{a}(m, \mathbb{R})$ is injective. Now we can equivalently view connections on $H$-structures as principal $H$-bundles $P \to M$ endowed with a $1$-form $\omega \in \Omega^1(P, \mathfrak{b})$ which satisfy the analogues of (\ref{eq:propertiesofomegathesum}) with respect to elements $g = j(h)$ with $h \in H$. In particular, as in the case of the affine space $A^m$ discussed above, one may look at the homogeneous space $B/H \cong \mathbb{R}^m$, and view the Maurer-Cartan form on $B$ as the connection on the $H$-structure $B \to \mathbb{R}^m$. Connections on generic $H$-structures can thus be thought of as curved analogues of this homogeneous space.
\begin{exmp}
    In the case $H = O(n) \subset GL(n,\mathbb{R})$, the affine extension $B$ is exactly the Euclidean group $\text{Euc}(n)$, and from Example \ref{ex:exampleofGstructure}, we see that this picture leads to viewing $n$-dimensional Riemannian manifolds as curved analogues of the Euclidean space $E^n$.
    \label{ex:exampleofcurvedanalogueGstructureofhomoaffine}
\end{exmp}
\begin{definition}
    Let $H \subset G$ be a Lie subgroup of the Lie group $G$. A \textsl{Cartan geometry} of type $(G,H)$ on a manifold $M$ is a principal $H$-bundle $\pi \colon P \to M$ together with a $\mathfrak{g}$-valued $1$-form $\omega \in \Omega^1(P;\mathfrak{g})$ called the \textsl{Cartan connection}, satisfying
    \begin{equation}
        \begin{split}
            &R_h^{\ast} \omega = ad_{h^{-1}} \circ \omega \text{ for all } h \in H;\\
            &\omega\big(\zeta(X)\big) = X, \text{ for all } X \in \mathfrak{h};\\
            &\omega_p \colon T_pP \to \mathfrak{g} \text{ is a linear isomorphism for all } p \in P.
        \end{split}
        \label{eq:propertiesofCartanconnection}
    \end{equation}
    The \textsl{homogeneous model} for a Cartan geometry of type $(G,H)$ is the bundle $G \to G/H$ endowed with the Maurer-Cartan form $\omega^{MC} \in \Omega^1(G;\mathfrak{g})$.
    \label{def:cartangeometries}
\end{definition}
In light of the above arguments, an $H$-structure $\pi \colon P \to M$ endowed with a principal connection $\gamma \in \Omega^1(P;\mathfrak{h})$ provides an example of Cartan geometry of type $(B,H)$, with $B = \mathbb{R}^n \rtimes H$ the affine extension, and Cartan connection given by
\begin{equation}
    \omega = \gamma + \Theta.
    \label{eq:cartanconnectionforanHstructure}
\end{equation}
Recall that $\Theta \in \Omega^1_{hor}(P;\mathbb{R}^n)^H$ is obtained as the pullback $\iota^{\ast}\theta$ of the canonical form by the reduction morphism $\iota \colon P \to FM$, as in Proposition \ref{pro:thepowerofGstructures}.\\

\chapter{Hopf algebras} \label{Hopfalgebras}
In this chapter, we introduce the reader to the realm of noncommutative geometry. Smooth manifolds are substituted by algebras of functions defined on them, which are then considered to be noncommutative in general. Without the need of a smooth manifold structure or even a topological space, one can formulate differential geometry over such algebras. Hopf algebras represent the algebraic generalization of the structure groups in classical differential geometry and principal fiber bundles, and thus play a fundamental role in the theory of quantum principal bundles which will be presented later.  Our main reference for this chapter is \cite{del_donno_durdevic}, with insights from \cite{kassel_qgroups}.

In Section \ref{algebrasandcoalgebras}, we define algebras and coalgebras, to later present in \ref{bialgebrasandHopfalgebras} the concept of an Hopf algebra as a bialgebra together with an antipode, which is the algebraic counterpart of group inversion in group theory. Examples of such Hopf algebras, or "quantum groups", can be obtained as deformations by a parameter $q$ of the algebra of functions over Lie groups: we show the example of $GL_q(2)$, which is relevant to later discussions in this thesis. In \ref{comodulealgebras}, we define comodules and comodule algebras, which play the role of total spaces in the quantum principal bundle formalism. In, particular, a quantum principal bundle will be defined as a Hopf-Galois extension. In Section \ref{HopfGaloisextensions}, we define the latter, together with cleft and trivial extensions, which will be shown in \ref{smashedproductalgebras} to correspond to the important case of the smashed product algebra. At last, in \ref{differentialcalculi}, the theory of first order differential calculi over algebras is provided, generalizing to noncommutative geometry the differentiable structure on a smooth manifold. In \ref{firstordercovariantcalculi}, we introduce covariant first order differential calculi with respect to an Hopf algebra $H$. This is done with particular focus on the case of the quotient and pullback calculi; in \ref{smashedproductcalculus} we show the construction of a first order differential calculus on the smashed product algebra \cite{latiniweber_projective}.

\section{Algebras and coalgebras} \label{algebrasandcoalgebras}
\begin{definition}
    An \textsl{algebra} $A$ over a field $\mathbb{K}$ is a $\mathbb{K}$-vector space together with $\mathbb{K}$-linear maps $\mu \colon A \otimes A \to A$, called \textsl{product}, and $\eta \colon \mathbb{K} \to A$, called \textsl{unit}, with $\eta(1) = 1_A$ being the unit element of the algebra, such that the following diagrams commute:\\
    \[
    \begin{tikzcd}
    A \otimes A \otimes A \arrow{r}{\mu \otimes id} \arrow[swap]{d}{id \otimes \mu} & A \otimes A \arrow{d}{\mu} \\
    A \otimes A \arrow{r}{\mu} & A
    \end{tikzcd}\;\;\;
    \begin{tikzcd}
    \mathbb{K} \otimes A \arrow{r}{\eta \otimes id} \arrow[swap]{rd}{\cong} & A \otimes A \arrow{d}{\mu} \\
    & A
    \end{tikzcd} \;\;\;
    \begin{tikzcd}
     A \otimes \mathbb{K} \arrow{r}{id \otimes \eta} \arrow[swap]{rd}{\cong} & A \otimes A \arrow{d}{\mu} \\
    & A 
    \end{tikzcd}.
    \]
    The first diagram gives \textsl{associativity}
    \begin{equation}
            \mu \circ (id \otimes \mu) = \mu \circ (\mu \otimes id),
        \label{eq:algebraassociativity}
    \end{equation}
    while the other two \textsl{unitality}
    \begin{equation}
        \mu \circ (\eta \otimes id) = id = \mu \circ (id \otimes \eta)
        \label{eq:algebraunitality}
    \end{equation}
    of the algebra $A$.
    \label{def:algebra}
\end{definition}
\begin{definition}
    A \textsl{morphism of algebras} is a $\mathbb{K}$-linear map $f \colon A \to A^{\prime}$ between $\mathbb{K}$-algebras $(A,\mu,\eta)$ and $(A^{\prime},\mu^{\prime},\eta^{\prime})$ such that the following diagrams commute:\\
    \[
    \begin{tikzcd}
    A \otimes A \arrow{r}{f \otimes f} \arrow[swap]{d}{\mu} & A^{\prime} \otimes A^{\prime} \arrow{d}{\mu^{\prime}} \\
    A \arrow{r}{f} & A^{\prime}
    \end{tikzcd}\;\;\;
    \begin{tikzcd}
     \mathbb{K} \arrow{r}{\eta} \arrow[swap]{rd}{\eta^{\prime}} &A \arrow{d}{f} \\
    & A^{\prime} 
    \end{tikzcd},
    \]
    that is
    \begin{equation}
        \begin{split}
            &f \circ \mu = \mu^{\prime} \circ (f \otimes f);\\
            &f \circ \eta = \eta^{\prime}.
        \end{split}
        \label{eq:morphismofalgebrasproperties}
    \end{equation}
    \label{def:algebramorphism}
\end{definition}
\begin{definition}
    A \textsl{coalgebra} $C$ over a field $\mathbb{K}$ is a $\mathbb{K}$-vector space together with $\mathbb{K}$-linear maps $\Delta \colon C \to C \otimes C$, called \textsl{coproduct}, and $\epsilon \colon C \to \mathbb{K}$, called \textsl{counit}, such that the following diagrams\\
    \[
    \begin{tikzcd}
    C \arrow{r}{\Delta} \arrow[swap]{d}{\Delta} & C \otimes C \arrow{d}{\Delta \otimes id} \\
    C \otimes C \arrow{r}{id \otimes \Delta} & C \otimes C \otimes C
    \end{tikzcd}\;\;\;
    \begin{tikzcd}
    C \arrow{r}{\Delta} \arrow[swap]{rd}{\cong} & C \otimes C \arrow{d}{\epsilon \otimes id} \\
    & \mathbb{K} \otimes C
    \end{tikzcd}\;\;\;
    \begin{tikzcd}
     C \arrow{r}{\Delta} \arrow[swap]{rd}{\cong} & C \otimes C \arrow{d}{id \otimes \epsilon} \\
    & \mathbb{K} \otimes C 
    \end{tikzcd},
    \]
    obtained by reversing the arrows in Definition \ref{def:algebra}, commute. The first diagram gives \textsl{coassociativity}
    \begin{equation}
        (id \otimes \Delta) \circ \Delta = (\Delta \otimes id) \circ \Delta
        \label{eq:coalgebracoassociativity}
    \end{equation}
    while the other two give \textsl{counitality}
    \begin{equation}
        (\epsilon \otimes id) \circ \Delta = id = (id \otimes \epsilon) \circ \Delta
        \label{eq:coalgebracounitality}
    \end{equation}
    of the coalgebra $C$.
    \label{def:coalgebra}
\end{definition}
\begin{remark}
    Throughout this thesis, we will use the so called "Sweedler notation", by which we denote
    \begin{equation}
        \Delta(c) = \sum_i c_{1_i} \otimes c_{2_i} =: c_1 \otimes c_2, \text{ for all } c \in C.
        \label{eq:sveedlernotation1}
    \end{equation}
    Using this notation, coassociativity reads
    \begin{equation}
        \begin{split}
            &(id \otimes \Delta) \circ \Delta(c)= c_1 \otimes \Delta(c_2) = c_1 \otimes c_{21} \otimes c_{22} = \\
            &=(\Delta \otimes id) \circ  \Delta(c) = c_{11} \otimes c_{12} \otimes c_2 = \\
            &= c_1 \otimes c_2 \otimes c_3,
        \end{split}
        \label{eq:sveedlercoassociativity}
    \end{equation}
    where we relable indexes from lowest to highest.\\
    On the other hand, counitality can be written as
    \begin{equation}
        \begin{split}
            &(\epsilon \otimes id) \circ \Delta(c) = \epsilon(c_1) \otimes c_2 = \epsilon(c_1)c_2 =\\
            &= (id \otimes \epsilon) \circ \Delta(c) = c_1\epsilon(c_2)=\\
            &=id(c)=c,
        \end{split}
        \label{eq:sveedlercounitality}
    \end{equation}
\end{remark}
\begin{definition}
    A \textsl{morphism of coalgebras} is a $\mathbb{K}$-linear map $g \colon C \to C^{\prime}$ between two coalgebras $(C,\Delta,\epsilon)$ and $(C^{\prime},\Delta^{\prime},\epsilon^{\prime})$ such that the following diagrams commute:\\
    \[
    \begin{tikzcd}
    C \arrow{r}{g} \arrow[swap]{d}{\Delta} & C^{\prime} \arrow{d}{\Delta^{\prime}} \\
    C \otimes C \arrow{r}{g \otimes g} & C^{\prime} \otimes C^{\prime}
    \end{tikzcd}\;\;\;
    \begin{tikzcd}
     C \arrow{r}{g} \arrow[swap]{rd}{\epsilon} &C^{\prime} \arrow{d}{\epsilon^{\prime}} \\
    & \mathbb{K} 
    \end{tikzcd},
    \]
    that is
    \begin{equation}
        \begin{split}
            &(g \otimes g) \circ \Delta = \Delta^{\prime} \circ g;\\
            &\epsilon = \epsilon^{\prime} \circ g.
        \end{split}
        \label{eq:morphismofcoalgebrasproperties}
    \end{equation}
    \label{def:coalgebramorphism}
\end{definition}

\section{Bialgebras and Hopf algebras} \label{bialgebrasandHopfalgebras}

\begin{definition}
    A $\mathbb{K}$-algebra $(B, \mu, \eta)$ is called a \textsl{bialgebra} if there is a $\mathbb{K}$-coalgebra structure $(\Delta, \epsilon)$ on $B$ such that $\Delta \colon B \to B \otimes B$ and $\epsilon \colon B \to \mathbb{K}$ are algebra morphisms or, equivalently, such that $\mu \colon B \otimes B \to B$ and $\eta \colon \mathbb{K} \to B$ are coalgebra morphisms. 
    \label{def:bialgebra}
\end{definition}
\begin{definition}
    A \textsl{morphism of bialgebras} is a $\mathbb{K}$-linear map $\phi \colon B \to B^{\prime}$ between bialgebras $(B,\mu,\eta,\Delta,\epsilon)$ and $(B^{\prime},\mu^{\prime},\eta^{\prime},\Delta^{\prime},\epsilon^{\prime})$ such that $\phi$ is both an algebra morphism and a coalgebra morphism.
    \label{def:bialgebramorphism}
\end{definition}
A Hopf algebra is a bialgebra with an additional element: the antipode $S$, which can be viewed as the algebraic counterpart of group inversion.
\begin{definition}
    A \textsl{Hopf algebra} is a bialgebra $(H, \mu, \eta, \Delta, \epsilon)$ together with a $\mathbb{K}$-linear map $S \colon H \to H$, called \textsl{antipode}, such that the following diagram commutes:\\
    \[
\begin{tikzcd}
                                                                    & H \otimes H \arrow[rr, "S \otimes id"] &                               & H \otimes H \arrow[rd, "\mu"]  &   \\
H \arrow[rr, "\epsilon"] \arrow[ru, "\Delta"] \arrow[rd, "\Delta"'] &                                        & \mathbb{K} \arrow[rr, "\eta"] &                                & H \\
                                                                    & H \otimes H \arrow[rr, "id \otimes S"] &                               & H \otimes H \arrow[ru, "\mu"'] &  
\end{tikzcd}
    \]
    This can be equivalently be written as
    \begin{equation}
        \mu \circ (S \otimes id) \circ \Delta = \eta \circ \epsilon = \mu \circ (id \otimes S) \circ \Delta,
        \label{eq:antipodeproperty}
    \end{equation}
    which in Sweedler notation reads
    \begin{equation}
        S(h_1)h_2 = \epsilon(h)=h_1S(h_2).
        \label{eq:robaseria}
    \end{equation}
    \label{def:Hopfalgebra} 
\end{definition}
\begin{definition}
    A \textsl{morphism of Hopf algebras} is a morphism of bialgebras $\varphi \colon H \to H^{\prime}$ between two Hopf algebras $(H,\mu,\eta,\Delta,\epsilon,S)$ and $(H^{\prime},\mu^{\prime},\eta^{\prime},\Delta^{\prime},\epsilon^{\prime},S^{\prime})$ such that the following diagram commutes:
    \[
    \begin{tikzcd}
    H \arrow{r}{\varphi} \arrow[swap]{d}{S} & H^{\prime} \arrow{d}{S^{\prime}} \\
    H \arrow{r}{\varphi} & H^{\prime}
    \end{tikzcd},
    \]
    that is
    \begin{equation}
        \varphi \circ S = S^{\prime} \circ \varphi.
        \label{eq:morphofHopfalg}
    \end{equation}
    \label{def:morphismofHopfalgebras}
\end{definition}
Following \cite{manin_qgroups}, we show the example of $GL_q(2)$, which will be relevant in the following sections
\begin{exmp}
    Let $\mathbb{K} = \mathbb{C}$ and $0 \neq q \in \mathbb{C}$. The ring of polynomial functions $\mathcal{O}_q(GL(2, \mathbb{C})) = GL_q(2)$ is a Hopf algebra $H =: GL_q(2)$ generated by $a,b,c,d$ and the formal inverse of a central element (i.e. it commutes with all elements in the algebra)
    \begin{equation}
        D = \underset{q}{det}\begin{pmatrix} a & b \\ c & d \end{pmatrix} = ad - qbc,
        \label{eq:quantumdeterminantGLq2}
    \end{equation}
    where $a,b,c,d$ satisfy the following $q$-commutation relations
    \begin{equation}
        \begin{split}
            & ab = qba, \;\; ac=qca, \;\; bd = qdb, \;\; cd=qdc, \\
            &bc=cb, \;\; ad-da=(q-q^{-1})bc.
        \end{split}
        \label{eq:GLq2commutationrelations}
    \end{equation}
    Coproduct $\Delta \colon H \to H \otimes H$, counit $\epsilon \colon H \to \mathbb{K}$ and antipode $S \colon H \to H$ are given as
    \begin{equation}
        \begin{split}
            &\Delta \colon \begin{pmatrix} a & b \\ c & d \end{pmatrix} \mapsto \begin{pmatrix} a & b \\ c & d \end{pmatrix} \dot{\otimes} \begin{pmatrix} a & b \\ c & d \end{pmatrix} = \begin{pmatrix} a \otimes a + b \otimes c & a \otimes b + b \otimes d \\ c \otimes a + d \otimes c & c \otimes b + d \otimes d \end{pmatrix};\\
            &\epsilon \colon \begin{pmatrix} a & b \\ c & d \end{pmatrix} \mapsto \begin{pmatrix} 1 & 0 \\ 0 & 1 \end{pmatrix}; \\
            &S \colon \begin{pmatrix} a & b \\ c & d \end{pmatrix} \mapsto D^{-1}\begin{pmatrix} d & -q^{-1}b \\ -qc & a \end{pmatrix}.
        \end{split}
        \label{eq:GLq2operations}
    \end{equation}
    To find how the coproduct acts on $D^{-1}$, we compute $\Delta(D) = \Delta(ad-qbc) = \Delta(a)\Delta(d)-q\Delta(b)\Delta(c)$ and we note that $\Delta(DD^{-1}) = \Delta(D)\Delta(D^{-1}) = 1_H \otimes 1_H$, hence $\Delta(D^{-1}) = \Delta(D)^{-1}$. One finds $\Delta(D) = D \otimes D$ so that $\Delta(D^{-1}) = D^{-1} \otimes D^{-1}$. Similarly, requiring the counit to be an algebra map, we have $\epsilon(D) = 1$ and thus $\epsilon(D^{-1}) = 1$. \\
    One can check that the above coalgebra operations are well defined. Let us show this for the generator $a$. Coassociativity reads:
    \begin{equation}
        \begin{split}
            &(id \otimes \Delta) \circ \Delta (a) =(id \otimes \Delta) \circ (a \otimes a + b \otimes c) =\\
            &=id(a) \otimes \Delta(a) +  id(b) \otimes \Delta(c)= \\
            &= a \otimes (a \otimes a + b \otimes c) + b \otimes (c \otimes a + d \otimes c) =\\
            &= a \otimes a \otimes a + a \otimes b \otimes c + b \otimes c \otimes a + b \otimes d \otimes c = \\
            &= (\Delta \otimes id) \circ \Delta(a) = \Delta(a) \otimes id(a) +  \Delta(b) \otimes id(c)= \\
            &= (a \otimes a + b \otimes c) \otimes a + (a \otimes b + b \otimes d) \otimes c = \\
            &= a \otimes a \otimes a + b \otimes c \otimes a + a \otimes b \otimes c + b \otimes d \otimes c.
        \end{split}
        \label{eq:GLq2coassociativityproof}
    \end{equation}
    Counitality reads:
    \begin{equation}
        \begin{split}
           &(\epsilon \otimes id) \circ \Delta(a) = (\epsilon \otimes id) \circ (a \otimes a + b \otimes c) =\\
           &=\epsilon(a) \otimes a + \epsilon(b) \otimes c = \\
           &=\epsilon(a)a+\epsilon(b)c=  a=\\
           &=(id \otimes \epsilon) \circ \Delta(a) = (id \otimes \epsilon) \circ (a \otimes a + b \otimes c)=\\
           &= a\epsilon(a) + b\epsilon(c) = a.
           \end{split}
        \label{eq:GLq2counitalityproof}
    \end{equation}
    Lastly, the antipode property is satisfied  (using the commutiation relations in Equation (\ref{eq:GLq2commutationrelations})):
    \begin{equation}
        \begin{split}
           & \mu \circ (S \otimes id) \circ \Delta(a) = \mu \circ (S \otimes id) \circ (a \otimes a + b \otimes c) = \\
           &= \mu \circ \big(S(a) \otimes a + S(b) \otimes c \big) = \mu \circ (D^{-1}d \otimes a  -D^{-1}q^{-1}b \otimes c) = \\
           &=D^{-1}(da -q^{-1}bc) = D^{-1}(ad - qbc + q^{-1}bc -q^{-1}bc)= \\
           &= D^{-1}D = 1_H =\eta\big( \epsilon(a) )= \\
           &= \mu \circ (id \otimes S) \circ \Delta(a) = \mu \circ \big(a \otimes S(a) + b \otimes S(c)\big)=\\
           &=aD^{-1}d + b(-D^{-1}qc) = D^{-1}(ad-qbc) = 1_H,
           \end{split}
        \label{eq:GLq2antipodeproof}
    \end{equation}
    where we used that since $D$ is central, i.e. $Dx = xD$ for any $x \in \{ a,b,c,d \}$, also $D^{-1}$ commutes with any element of the algebra, since $D^{-1}x = D^{-1}xDD^{-1} = D^{-1}DxD^{-1} = xD^{-1}$. One can similarly check the Hopf algebra axioms are satisfied on any of the generators.
    \label{ex:GLq2}
\end{exmp}
\begin{proposition}
    Let $H$ be a Hopf algebra. Then its antipode $S \colon H \to H$ is unique.
    \label{pro:unicityofantipode}
\end{proposition}
\begin{proof}
    Let $h \in H$ and $S_1$, $S_2$ be two antipode maps. Using counitality, i.e. $h = h_1\epsilon(h_2) = \epsilon(h_1)h_2$, we can write:
    \begin{equation}
        \begin{split}
            &S_1(h)=S_1\big(h_1\epsilon(h_2)\big)=S_1(h_1)\epsilon(h_2) = \\
            &=S_1(h_1)h_{21}S_2(h_{22})=S_1(h_1)h_2S_2(h_3)
        \end{split}
        \label{eq:proofofunicity1}
    \end{equation}
    and
    \begin{equation}
        \begin{split}
            &S_2(h)=S_2\big(\epsilon(h_1)h_2\big)=\epsilon(h_1)S(h_2)=\\
            &=S_1(h_{11})h_{12}S_2{h_2}=S_1(h_1)h_2S_2(h_3).
        \end{split}
        \label{eq:proofofunicity2}
    \end{equation}
\end{proof}
We now introduce the notions of convolution algebra and convolution invertibility to prove properties of the antipode $S$.
\begin{definition}
    Given an associative unital algebra $(A, \mu, \eta)$ and a coassociative counital coalgebra $(C, \Delta, \epsilon)$, we denote by $Hom_{\mathbb{K}}(C,A)$ the vector space of $\mathbb{K}$-linear maps $C \to A$. Define a $\mathbb{K}$-linear map
    \begin{equation}
        \begin{split}
            * \colon &Hom_{\mathbb{K}}(C,A) \otimes Hom_{\mathbb{K}}(C,A) \to Hom_{\mathbb{K}}(C,A),\\
            &(f,f^{\prime}) \mapsto f*f^{\prime}:=\mu \circ (f \otimes f^{\prime}) \circ \Delta.
        \end{split}
        \label{eq:convolutiondef}
    \end{equation}
    We call $(Hom_{\mathbb{K}}(C,A),*)$ the \textsl{convolution algebra}, with $*$ being its product.
    \label{def:convolutionalgebra}
\end{definition}
\begin{proposition}
     The convolution algebra $(Hom_{\mathbb{K}}(C,A),*)$ is an associative unital algebra with unit $\eta \circ \epsilon$. 
    \label{pro:convolutionalgebra}
\end{proposition}
\begin{proof}
    The product $*$ is bilinear. For every $f,g,f^{\prime},g^{\prime} \in Hom_{\mathbb{K}}(C,A)$ and $k,h,k^{\prime},h^{\prime} \in \mathbb{K}$ we have
    \begin{equation}
        \begin{split}
            &(kf + hg)*(k^{\prime}f^{\prime}+h^{\prime}g^{\prime}) (x)=\\
            &=\mu \circ \big((kf + hg) \otimes (k^{\prime}f^{\prime}+h^{\prime}g^{\prime})\big) \circ \Delta(x)=\\
            &=\mu \circ (kf \otimes k^{\prime}f^{\prime}+kf \otimes h^{\prime}g^{\prime}+hg\otimes k^{\prime}f^{\prime}+hg\otimes h^{\prime}g^{\prime}) \circ (x_1 \otimes x_2)=\\
            &=\mu \circ \big(kf(x_1) \otimes k^{\prime}f^{\prime}(x_2)+kf(x_1)\otimes h^{\prime}g^{\prime}(x_2)+\\
            &+hg(x_1)\otimes k^{\prime}f^{\prime}(x_2)+hg(x_1)\otimes h^{\prime}g^{\prime}(x_2)\big)=\\
            &=kk^{\prime}f(x_1)f^{\prime}(x_2) + kh^{\prime}f(x_1)g^{\prime}(x_2) + hk^{\prime}g(x_1)f^{\prime}(x_2)+hh^{\prime}g(x_1)g^{\prime}(x_2)=\\
            &=kk^{\prime}(f*f^{\prime})(x) + kh^{\prime}(f*g^{\prime})(x) + hk^{\prime}(g*f^{\prime})(x)+hh^{\prime}(g*g^{\prime})(x),
        \end{split}
        \label{eq:convolutionalgebrabilinearityproof}
    \end{equation}
    It satisfies associativity:
    \begin{equation}
        \begin{split}
            &\big((f*g)*h\big)(x)= \mu \circ \big((f*g)\otimes h\big) \circ \Delta(x) =  \mu \circ \big((f*g)(x_1) \otimes h(x_2)\big) =  \\
            &= \mu \circ \Big(\mu \circ \big(f(x_{11})\otimes g(x_{12})\big) \otimes h(x_2)\Big)= \mu \circ \big(f(x_{11})g(x_{12}) \otimes h(x_2)\big)=\\
            &=f(x_1)g(x_2)h(x_3)=\big(f*(g*h)\big)(x)= \mu \circ \big(f(x_1) \otimes g(x_{21})h(x_{22})\big)=\\
            &=f(x_1)g(x_2)h(x_3)
        \end{split}
        \label{eq:convolutionalgebraassociativityproof}
    \end{equation}
    for any $f,g,h \in Hom_{\mathbb{K}}(C,A)$ and $x \in C$.\\
    Finally, $\eta \circ \epsilon$ satisfies unitality:
    \begin{equation}
        \begin{split}
            &\big((\eta \circ \epsilon) * f\big)(x)= \mu \circ \big((\eta \circ \epsilon) \otimes f\big)(x_1 \otimes x_2) =\\
            &= \mu \circ \big((\eta \circ \epsilon)(x_1) \otimes f(x_2)\big)=\eta\big(\epsilon(x_1)\big)f(x_2)=\\
            &=\epsilon(x_1)f(x)_2=f(x).
        \end{split}
        \label{eq:convolutionalgebraunitproof}
    \end{equation}
    The proof is similar for $\big(f*(\eta \circ \epsilon)\big)=f$.
\end{proof}
\begin{definition}
    A $\mathbb{K}$-linear map $j \colon H \to A$ is said \textsl{convolution invertible} there exists a $\mathbb{K}$-linear map $i \colon H \to A$ such that
    \begin{equation}
        j(h_1)i(h_2)=(\eta \circ \epsilon)(h)=\epsilon(h)1_A=i(h_1)j(h_2),
        \label{eq:convolutioninvertibilitydef}
    \end{equation}
    for every element $h \in H$. In this case one can write $i=j^{-1}$ or $j = i^{-1}$.
    \label{def:convolutioninvertibility}
\end{definition}
\begin{proposition}
    Let $H$ be a Hopf algebra.
    \begin{enumerate}
        \item If $f \colon C \to H$ is a morphism of coalgebras then $S \circ f$ is the convolution inverse of $f$ in the convolution algebra $Hom_{\mathbb{K}}(C,H)$.
        \item If $f \colon H \to A$ is a morphism of algebras then $f \circ S$ is the convolution inverse of $f$ in the convolution algebra $Hom_{\mathbb{K}}(H,A)$.
    \end{enumerate}
    \label{pro:inveritbilityinconvolutionalgebra}
\end{proposition}
\begin{proof} In order:
    \begin{enumerate}
    \item Let $f \colon C \to H$ be a coalgebra map, i.e. $\Delta \circ f = (f \otimes f) \circ \Delta_C$ and $\epsilon \circ f= \epsilon_C$, where $\Delta_C$ and $\epsilon_C$ are the coproduct and counit map of $C$. We have to prove $(S \circ f)*f=\mu \circ \big((S \circ f)\otimes f\big) \circ \Delta_C = \eta \circ \epsilon_C = f*(S \circ f) = \mu \circ \big(f \otimes (S \circ f)\big) \circ \Delta_C$. For any $x \in C$ we have:
    \begin{equation}
        \begin{split}
            &\big((S \circ f) * f\big)(x) = \mu \circ \big(S(f(x_1))\otimes f(x_2)\big)=S\big(f(x_1)\big)f(x_2)=\\
            &=S\big(f(x)_1\big)f(x)_2=\eta \circ \epsilon\big(f(x)\big)=\\
            &=\eta \circ \epsilon_C(x)=\epsilon_C(x)1_H=\\
            &=\big(f*(S\circ f)\big)(x)=f(x)_1S\big(f(x)_2\big)=\eta \circ \epsilon\big(f(x)\big)=\\
            &=\eta \circ \epsilon_C(x)=\epsilon_C(x)1_H.
        \end{split}
        \label{eq:proofconvolutioninverse1}
    \end{equation}  
    \item Let $f \colon H \to A$ be an algebra map, i.e. $f \circ \mu = \mu_A \circ (f \otimes f)$ and $f \circ \eta= \eta_A$, where $\mu_A$ and $\eta_A$ are the product and unit map of $A$. We have to prove $(f \circ S)*f=\mu_A \circ \big((f \circ S)\otimes f\big) \circ \Delta = \eta_A \circ \epsilon(x) = f*(f \circ S) = \mu_A \circ \big(f \otimes (f \circ S)\big) \circ \Delta$. For any $h \in H$ we have:
    \begin{equation}
        \begin{split}
            &\big((f \circ S) * f\big)(h) = \mu_A \circ \big(f(S(h_1))\otimes f(h_2)\big)=f\big(S(h_1)\big)f(h_2)=\\
            &=f\big(S(h_1)h_2\big)=f\big(\eta \circ \epsilon(h)\big)= \eta_A \circ \epsilon(h) = \epsilon(h)1_A=\\
            &=\big(f*(f \circ S)\big)(h)=f(h_1)f\big(S(h_2)\big)=f\big(h_1S(h_2)\big)=\\
            &=f\big(\eta \circ \epsilon(h)\big)=\eta_A \circ \epsilon(h)=\epsilon(h)1_A.
        \end{split}
        \label{eq:proofconvolutioninverse2}
    \end{equation}
    \end{enumerate}
\end{proof}
Given the last two propositions, we can prove the following properties of the antipode and see it is an anti-bialgebra map, i.e. it flips the product and the coproduct.
\begin{proposition}
    Let $H$ be a Hopf algebra. Its antipode $S \colon H \to H$ satisfies:
    \begin{enumerate}
        \item $S(hh^{\prime})=S(h^{\prime})S(h)$ for every $h,h^{\prime} \in H$;
        \item $S(1_H)=1_H$;
        \item $\Delta \circ S = (S \otimes S) \circ flip \circ \Delta$;
        \item $\epsilon \circ S =\epsilon$.
    \end{enumerate}
    \label{pro:Hopfantipodeproperties}
\end{proposition}
\begin{proof} Proving point by point:
\begin{enumerate}
    \item We define $f := S \circ \mu \colon H \otimes H \to H$ and $g := \mu \circ flip \circ (S \otimes S) \colon H \otimes H \to H$. By Proposition \ref{pro:inveritbilityinconvolutionalgebra} we have that $f$ is the convolution inverse of $\mu$ in the convolution algebra $Hom_{\mathbb{K}}(H \otimes H, H)$, since $\mu$ is a coalgebra map by definition of bialgebra.
    Moreover, we notice that $g = \mu \circ flip \circ (S \otimes S)$ is a (left) convolution inverse of $\mu \colon H \otimes H \to H$ in $Hom_{\mathbb{K}}(H \otimes H, H)$, in fact:
    \begin{equation}
        \begin{split}
            &(g * \mu)(h \otimes h^{\prime})= \mu \circ (g \otimes \mu) \circ \Delta(h \otimes h^{\prime}) = g(h_1 \otimes h^{\prime}_1)\mu(h_2 \otimes h^{\prime}_2)=\\
            &=\Big(\big(\mu \circ flip \circ (S \otimes S)\big)(h_1 \otimes h^{\prime}_1)\Big)h_2h^{\prime}_2=\\
            &=S(h^{\prime}_1)S(h_1)h_2h^{\prime}_2=\epsilon(h)S(h^{\prime}_1)h^{\prime}_2=\epsilon(h)\epsilon(h^{\prime})1_H=\epsilon(hh^{\prime})1_H,
        \end{split}
        \label{eq:antipodeproperties1}
    \end{equation}
    where $\Delta \colon H \otimes H \to H \otimes H \otimes H \otimes H, h \otimes h^{\prime} \mapsto h_1 \otimes h_2 \otimes h^{\prime}_1 \otimes h^{\prime}_2$ is the coproduct for $H \otimes H$ and we used the fact that $\epsilon \circ \mu=\mu \circ (\epsilon \otimes \epsilon)$, i.e.  that $\epsilon$ is an algebra map. \\
    Therefore, since both $f$ and $g$ are convolution inverses of $\mu$ in the convolution algebra $Hom_{\mathbb{K}}(H \otimes H, H)$ an by uniqueness of the inverse, we have $f=g$, i.e. $S \circ \mu = \mu \circ flip \circ (S \otimes S)$ or $S(hh^{\prime}) = S(h^{\prime})S(h)$ for any $h,h^{\prime} \in H$: the antipode $S$ is an anti-algebra morphism.
    \item Since $\eta \colon \mathbb{K} \to H$ is a coalgebra map we know $S \circ \eta$ is its convolution inverse by Proposition \ref{pro:inveritbilityinconvolutionalgebra}. Moreover $\eta$ is the unit of the convolution algebra $Hom_{\mathbb{K}}(\mathbb{K}, H)$ and so the convolution inverse of itself. 
    \item Let $f := \Delta \circ S \colon H \to H \otimes H$ and $g := (S \otimes S) \circ flip \circ \Delta \colon H \to H \otimes H$. Since $\Delta \colon H \to H \otimes H$ is an algebra map $f = \Delta \circ S$ is the convolution inverse of $\Delta$ by Proposition \ref{pro:inveritbilityinconvolutionalgebra}; moreover, for every $h \in H$ we get
    \begin{equation}
        \begin{split}
            &(g*\Delta)(h)=\mu \circ (g \otimes \Delta)(h_1 \otimes h_2)=g(h_1)\Delta(h_2)\\
            &=\big(S(h_{12}) \otimes S(h_{11})\big)(h_{21} \otimes h_{22})=\big(S(h_2)\otimes S(h_1)\big)(h_3 \otimes h_4)=\\
            &=S(h_2)h_3 \otimes S(h_1)h_4 = \epsilon(h)1_H \otimes S(h_1)h_2=\\
            &=\epsilon(h)1_H \otimes \epsilon(h)1_H=\epsilon(h)\epsilon(h)(1_H \otimes 1_H)=\epsilon(h)(1_H \otimes 1_H),
        \end{split}
        \label{eq:antipodeproperties3}
    \end{equation}
    i.e. $g$ is also a convolution inverse of $\Delta$ in $Hom_{\mathbb{K}}(H, H \otimes H)$. Therefore, by uniqueness, $f=g$, i.e. $\Delta \circ S=(S\otimes S) \circ flip \circ \Delta$: $S$ is an anti-coalgebra morphism.
    \item Since $\epsilon \colon H \to \mathbb{K}$ is an algebra map, $\epsilon \circ S$ must be its convolution inverse by Proposition $\ref{pro:inveritbilityinconvolutionalgebra}$. On the other hand $\epsilon$ is the unit of the convolution algebra $Hom_{\mathbb{K}}(H,\mathbb{K})$ and so its the convolution inverse of itself. Therefore $\epsilon \circ S=\epsilon$.
    \end{enumerate}
\end{proof}

\subsection{Comodules and comodule algebras} \label{comodulealgebras}
\begin{definition}
    A \textsl{left} $H$\textsl{-module} is a vector space $M$ over the field $\mathbb{K}$ with a \textsl{left} $H$\textsl{-action} $H \times M \to M, (h,m) \mapsto h \triangleright m$ such that $(hg) \triangleright m = h \triangleright (g \triangleright m)$ for every $h,g \in H$ and $m \in M$ and $1_H \triangleright m = m$ for every $m \in M$.
    
    Similarly, a \textsl{right} $H$\textsl{-module} is a vector space $M$ over $\mathbb{K}$ with a \textsl{right} $H$\textsl{-action} $M \times H \to H,(m,h) \mapsto m \triangleleft h$ such that $m \triangleleft (hg)= (m \triangleleft h) \triangleleft g$ and $m \triangleleft 1_H = m$.

    An $H$\textsl{-bimodule} $B$ is both a left and a right $H$-module, with actions $\triangleright$ and $\triangleleft$ such that $h \triangleright (b \triangleleft h^{\prime}) = (h \triangleright b) \triangleleft h^{\prime}$ for every $h,h^{\prime} \in H$ and $b \in B$.
    \label{def:module}
\end{definition}
\begin{definition}
    A \textsl{left} $H$\textsl{-module algebra} $V$ is an associative algebra over $\mathbb{K}$ which is also a left $H$-module, with compatibility conditions
    \begin{equation}
        h \triangleright 1_V = \epsilon(h)1_V, \;\;\; h \triangleright (vw)=(h_1 \triangleright v)(h_2 \triangleright w)
        \label{eq:leftmodulealgebracompatibility}
    \end{equation}
    for any $v, w \in V$ and $h \in H$.
    
    Similarly, a \textsl{right} $H$\textsl{-module algebra} $V$ is an associative algebra over $\mathbb{K}$ which is also a right $H$-module, with compatibility conditions
    \begin{equation}
        1_V \triangleleft h = 1_V\epsilon(h), \;\;\; (vw) \triangleleft h=( v \triangleleft h_1)(w \triangleleft h_2)
        \label{eq:rightmodulealgebracompatibility}
    \end{equation}
    for any $v, w \in V$ and $h \in H$.
    \label{def:modulealgebra}
\end{definition}
\begin{definition}
    A \textsl{right} $H$\textsl{-comodule} is a vector space $V$ together with a map $\Delta_V \colon V \to V \otimes H$, called \textsl{right} $H$\textsl{-coaction}, such that the diagrams\\
    \[
    \begin{tikzcd}
    V \arrow{r}{\Delta_V} \arrow{d}{\Delta_V} & V \otimes H \arrow{d}{\Delta_V \otimes id} \\
    V \otimes H \arrow{r}{id \otimes \Delta} & V \otimes H \otimes H
    \end{tikzcd}\;\;\;
    \begin{tikzcd}
    V \arrow{r}{\Delta_V} \arrow{rd}{id} &V \otimes H \arrow{d}{id \otimes \epsilon} \\
    & V
    \end{tikzcd}
    \]
    commute, i.e.
    \begin{equation}
        \begin{split}
            &(\Delta_V \otimes id) \circ \Delta_V= (id \otimes \Delta) \circ \Delta_V;\\
            &(id \otimes \epsilon) \circ \Delta_V=id.
        \end{split}
        \label{eq:commutingdiagramsrightcomodule}
    \end{equation}
    
    A \textsl{left} $H$\textsl{-comodule} is a vector space $V$ together with a map $\leftindex_V \Delta \colon V \to h \otimes V$, called \textsl{left} $H$\textsl{-coaction}, such that the diagrams\\
    \[
    \begin{tikzcd}
    V \arrow{r}{\leftindex_V\Delta} \arrow{d}{\leftindex_V\Delta} & H \otimes V \arrow{d}{id \otimes \leftindex_V\Delta} \\
    H \otimes V \arrow{r}{\Delta \otimes id} & H \otimes H \otimes V
    \end{tikzcd}\;\;\;
    \begin{tikzcd}
    V \arrow{r}{\leftindex_V\Delta} \arrow{rd}{id} & H \otimes V \arrow{d}{\epsilon \otimes id} \\
    & V
    \end{tikzcd}
    \]
    commute, i.e.
    \begin{equation}
        \begin{split}
            &(id \otimes \leftindex_V\Delta) \circ \leftindex_V\Delta= (\Delta \otimes id) \circ \leftindex_V\Delta;\\
            &(\epsilon \otimes id) \circ \leftindex_V\Delta=id.
        \end{split}
        \label{eq:commutingdiagramsleftcomodule}
    \end{equation}
    \label{def:comodule}
\end{definition}
\begin{remark}
    The Sweedler notation we will be using for right and left $H$-coactions on $V$ is the following:
    \begin{equation}
        \begin{split}
            &\Delta_V(v) =: v_0 \otimes v_1;\\
            &\leftindex_V\Delta(v) =: v_{-1} \otimes v_0.
        \end{split}
        \label{eq:sveedlerleftandrightcoactionnotation}
    \end{equation}
    for all $v \in V$.
    \label{re:onSweedlernotationforcoactions}
\end{remark}
\begin{definition}  
    Let $(V, \Delta_V)$ and $(W, \Delta_W)$ be right $H$-comodules. A \textsl{morphism of right} $H$\textsl{-comodules}, or \textsl{right} $H$\textsl{-colinear map}, is a $\mathbb{K}$-linear map $\phi \colon V \to W$ such that the diagram\\
    \[
    \begin{tikzcd}
    V \arrow{r}{\phi} \arrow{d}{\Delta_V} & W \arrow{d}{\Delta_W} \\
    V \otimes H \arrow{r}{\phi \otimes id} & W \otimes H
    \end{tikzcd}
    \]
    commutes, i.e.
    \begin{equation}
        \Delta_W \circ \phi = (\phi \otimes id) \circ \Delta_V.
        \label{eq:commutingdiagramrightcomodulemorphism}
    \end{equation}

    Let $(V, \leftindex_V\Delta)$ and $(W, \leftindex_W\Delta)$ be left $H$-comodules. A \textsl{morphism of left} $H$\textsl{-comodules}, or a \textsl{left} $H$\textsl{-colinear map}, is a $\mathbb{K}$-linear map $\phi \colon V \to W$ such that the diagram\\
    \[
    \begin{tikzcd}
    V \arrow{r}{\phi} \arrow{d}{\leftindex_V\Delta} & W \arrow{d}{\leftindex_W\Delta} \\
    H \otimes V \arrow{r}{id \otimes \phi} & H \otimes W
    \end{tikzcd}
    \]
    commutes, i.e.
    \begin{equation}
        \leftindex_W\Delta \circ \phi = (id \otimes \phi) \circ \leftindex_V\Delta.
        \label{eq:commutingdiagramleftcomodulemorphism}
    \end{equation}
    \label{def:morphismofcomodules}
\end{definition}
\begin{definition}
    A \textsl{right} $H$\textsl{-comodule algebra} $A$ is an algebra which is also a right $H$-comodule with right $H$-coation $\Delta_A \colon A \to A \otimes H$, such that $\mu \colon A \otimes A \to A$ and $\eta \colon \mathbb{K} \to A$ are morphisms of right $H$-comodules: namely, for any $a,b \in A$,
    \begin{equation}
        \begin{split}
            &\Delta_A(ab)=\Delta_A(a)\Delta_A(b) \;\;\; \text{and}\\
            &\Delta_A(1_A)=1_A \otimes 1_H.
        \end{split}
        \label{eq:compadibiltityleftcomodulealgebra2}
    \end{equation}
    
    A \textsl{left} $H$\textsl{-comodule algebra} $A$ is an algebra which is also a left $H$-comodule with left $H$-coaction $\leftindex_A\Delta \colon A \to H \otimes A$, such that $\mu \colon A \otimes A \to A$ and $\eta \colon \mathbb{K} \to A$ are morphisms of left $H$-comodules: namely
    \begin{equation}
        \begin{split}
            &\leftindex_A\Delta(ab)=\leftindex_A\Delta(a)\leftindex_A\Delta(b) \;\;\; \text{and}\\
            &\leftindex_A\Delta(1_A)=1_H \otimes 1_A.
        \end{split}
        \label{eq:compadibiltityleftcomodulealgebra1}
    \end{equation}
    \label{def:comodulealgebra}
\end{definition}
\begin{definition}
    Let $(V, \Delta_V)$ and $(W, \Delta_W)$ be right $H$-comodule algebras. A \textsl{morphism of right} $H$\textsl{-comodule algebras} is a morphism of right $H$-comodules $\phi \colon V \to W$ that is also a morphism of algebras, i.e. 
    \begin{equation}
        \begin{split}
            &\phi(vv^{\prime})=\phi(v)\phi(v^{\prime}),\\
            &\phi(1_V)=1_W
        \end{split}
        \label{eq:leftcomodulealgebramorphism2}
    \end{equation}
    for every $v,v^{\prime} \in V$. 
    
    Let $(V, \leftindex_V\Delta)$ and $(W, \leftindex_W\Delta)$ be left $H$-comodule algebras. A \textsl{morphism of left} $H$\textsl{-comodule algebras} is a morphism of left $H$-comodules $\phi \colon V \to W$ that is also a morphism of algebras, i.e. 
    \begin{equation}
        \begin{split}
            &\phi(vv^{\prime})=\phi(v)\phi(v^{\prime}),\\
            &\phi(1_V)=1_W.
        \end{split}
        \label{eq:leftcomodulealgebramorphism1}
    \end{equation} 
    for every $v,v^{\prime} \in V$.
    \label{def:comodulealgebramorphisms}
\end{definition}
\begin{definition}
    Let $A$ be a right $H$-comodule algebra, we define the subalgebra $A^{coH}$ of \textsl{right-coinvariant} elements under the right $H$-coaction $\Delta_A \colon A \to A \otimes H$ as:
    \begin{equation}
        A^{coH} := \{ a \in A \;|\; \Delta_A(a) = a \otimes 1_H \}.
        \label{eq:rightcoinvariantsubalgebradef}
    \end{equation}
    
    Let $A$ be a left $H$-comodule algebra. We define the subalgebra $\leftindex^{coH}A$ of \textsl{left-coinvariant} elements under the left $H$-coaction $\leftindex_A\Delta \colon A \to H \otimes A$ as:
    \begin{equation}
        \leftindex^{coH}A := \{ a \in A \;|\; \leftindex_A\Delta(a) = 1_H \otimes a \}.
        \label{eq:leftcoinvariantsubalgebradef}
    \end{equation}
    \label{def:leftandrightcoinvariantsubalgebras}
\end{definition}
\begin{exmp}
    Let us define the \textsl{adjoint right} $H$\textsl{-coaction} as
    \begin{equation}
        \begin{split}
            ad_R \colon &H \to H \otimes H,\\
            &h \mapsto h_2 \otimes S(h_1)h_3.
        \end{split}
        \label{eq:adjointrightcoactiondef}
    \end{equation}
    We have that the Hopf algebra $H$ is also a right $H$-comodule with right $H$-coaction $ad_R$. To show this, we have to prove that 
    $(id \otimes \Delta) \circ ad_R = (ad_R \otimes id) \circ ad_R$. Using the fact that $\Delta$ is an algebra map and the third property of Proposition \ref{pro:Hopfantipodeproperties}, the left-hand side reads
    \begin{equation}
        \begin{split}
            &(id \otimes \Delta) \circ ad_R(h) = (id \otimes \Delta)\big(h_2 \otimes S(h_1)h_3\big) =\\
            &=h_2 \otimes \Delta\big(S(h_1)h_3\big)= h_2 \otimes \Delta\big(S(h_1)\big)\Delta(h_3)=\\
            &= h_2 \otimes \big(flip \circ (S \otimes S) \circ \Delta(h_1)\big)\Delta(h_3)=\\
            &=h_2 \otimes \big(S(h_{12}) \otimes S(h_{11})\big)(h_{31} \otimes h_{32})=\\
            &=h_2 \otimes \big(S(h_{12})h_{21} \otimes S(h_{11})h_{22}\big)=\\
            &=h_3 \otimes S(h_2)h_4 \otimes S(h_1)h_5,
        \end{split}
        \label{eq:adjointrightHcomoduleproof1}
    \end{equation}
    and the right-hand side is
    \begin{equation}
        \begin{split}
            &(ad_R \otimes id) \circ ad_R(h) = (ad_R \otimes id)\big(h_2 \otimes S(h_1)h_3\big) =\\
            &=ad_R(h_2) \otimes S(h_1)h_3= \big(h_{22} \otimes S(h_{21})h_{23}\big) \otimes S(h_1)h_3=\\
            &=h_3 \otimes S(h_2)h_4 \otimes S(h_1)h_5.
        \end{split}
        \label{eq:adjointrightHcomoduleproof2}
    \end{equation}
    Moreover, we have to check that $(id \otimes \epsilon) \circ ad_R = id$. For any $h \in H$, using that $\epsilon$ is an algebra map and the fourth property in Proposition \ref{pro:Hopfantipodeproperties}, we can write
    \begin{equation}
        \begin{split}
            &(id \otimes \epsilon) \circ ad_R(h) = \big(id \otimes \epsilon)(h_2 \otimes S(h_1)h_3\big) = h_2 \otimes \epsilon\big(S(h_1)h_3\big) = \\
            &= h_2 \otimes \epsilon\big(S(h_1)\big)\epsilon(h_3) = h_2 \otimes (\epsilon \circ S)(h_1)\epsilon(h_3) = h_2 \otimes \epsilon(h_1)\epsilon(h_3) = h
        \end{split}
        \label{eq:adjointrightHcomoduleproof3}
    \end{equation}
    which proves that $ad_R$ is a right $H$-coaction, endowing $H$ with a right $H$-comodule structure.
    \label{ex:rightadjointcoactiongivesrightHcomodule}
\end{exmp}
\begin{remark}
    Every Hopf algebra $H$ can be seen as a right $H$-comodule algebra by seeing the coproduct $\Delta \colon H \to H \otimes H$ as a right $H$-coaction. In fact,  $(\Delta \otimes id) \circ \Delta(h) = (id \otimes \Delta) \circ \Delta(h)$ is the coassociativity axiom for a Hopf algebra and $(id \otimes \epsilon) \circ \Delta = id$ is counitality. Finally $\Delta(hh^{\prime}) = \Delta(h)\Delta(h^{\prime})$ and $\Delta(1_H) = 1_H \otimes 1_H$ since $H$ is a bialgebra and thus $\Delta$ is a morphsim of algebras.
    \label{rem:Hopfalgebraasacomodulealgebrawithcoproductascoaction}
\end{remark}

\section{Hopf-Galois extensions} \label{HopfGaloisextensions}
Let $B := A^{coH}$ be the subalgebra of coinvariant elements of $A$ under $\Delta_A$. In the lemma below, we show some properties of a convolution invertible right $H$-colinear map $j \colon H \to A$.
\begin{lemma}
    Let $j \colon H \to A$ be a convolution invertible morphism of right $H$-comodules, where $H$ is a right $H$-comodule by the coproduct $\Delta$ as in Remark \ref{rem:Hopfalgebraasacomodulealgebrawithcoproductascoaction}. Denote by $j^{-1} \colon H \to A$ the convolution inverse of $j$, i.e. $j * j^{-1} = \eta \circ \epsilon$ or $j(h_1)j^{-1}(h_2)=\epsilon(h)1_A$. Then:
    \begin{enumerate}
        \item the convolution inverse satisfies $\Delta_A \circ j^{-1} = (j^{-1} \otimes S) \circ flip \circ \Delta$;
        \item $\mu \circ (id \otimes j^{-1}) \circ \Delta_A$ is a map $A \to B$, i.e. $a_0 j^{-1}(a_1)\in B$ for any $a \in A$;
        \item $j(1_H) \in A$ is invertible with inverse $j^{-1}(1_H) \in A$;
        \item there is a unital right $H$-colinear convolution invertible map $j^{\prime} \colon H \to A, h \mapsto j^{-1}(1_H)j(h)$, with convolution inverse $j^{\prime-1} \colon H \to A, h \mapsto j^{-1}(h)j(1_H)$;
        \item if $j$ is an algebra morphism then $j^{-1}$ is anti-algebra morphism, that is $j^{-1}(hh^{\prime}) = j^{-1}(h^{\prime})j^{-1}(h)$ and $j^{-1}(1_H) = 1_A$ for every $h,h^{\prime} \in H$.
    \end{enumerate}
    \label{lem:convolutioninvertiblecomodulemorphismproperties}
\end{lemma}
\begin{proof}
    \begin{enumerate}
        \item Using the fact that $j \colon H \to A$ is a right $H$-colinear map, i.e. $\Delta_A \circ j = (j \otimes id) \circ \Delta$, and that $A$ is a right $H$-comodule algebra, so that $\Delta_A(aa^{\prime}) = \Delta_A(a)\Delta_A(a^{\prime})$ for any $a,a^{\prime} \in A$ and $\Delta_A(1_A) = 1_A \otimes 1_H$, we can write
        \begin{equation}
            \begin{split}
                &\epsilon(h)(1_A \otimes 1_H)= \Delta_A\big(\epsilon(h)1_A\big) =  \Delta_A\big((j * j^{-1})(h)\big) = \\
                &=\Delta_A\big(j(h_1)j^{-1}(h_2)\big) = \Delta_A\big(j(h_1)\big)\Delta_A\big(j^{-1}(h_2)\big) = \\
                &= \big(j(h_{11}) \otimes h_{12}\big)\Delta_A\big(j^{-1}(h_2)\big)=\\
                &= \big(j(h_{11}) \otimes h_{12}\big)\big(j^{-1}(h_2)_0 \otimes j^{-1}(h_2)_1\big)=\\
                &=j(h_{11})j^{-1}(h_2)_0 \otimes h_{12}j^{-1}(h_2)_1 =\\
                &=j(h_1)j^{-1}(h_3)_0 \otimes h_2j^{-1}(h_3)_1
            \end{split}
            \label{eq:convolinvprop11}
        \end{equation}
        for every $h \in H$. Consider now $f := (j \otimes id) \circ \Delta \colon H \to A \otimes H$, i.e. $f = \Delta_A \circ j = (j \otimes id) \circ \Delta$ since $j$ is a morphism of right $H$-comodules. We have that:
        \begin{equation}
            \begin{split}
                &\big(f*(\Delta_A \circ j^{-1}) \big)(h)=\mu \circ \big(f(h_1) \otimes (\Delta_A \circ j^{-1})(h_2)\big)=\\
                &=\big(j(h_{11}) \otimes h_{12}\big)\Delta_A\big(j^{-1}(h_2)\big)=\\
                &=\big(j(h_{11}) \otimes h_{12}\big)\big(j^{-1}(h_2)_0 \otimes j^{-1}(h_2)_1\big)=\\
                &=j(h_{11})j^{-1}(h_2)_0 \otimes h_{12}j^{-1}(h_2)_1 =\\
                &=j(h_1)j^{-1}(h_3)_0 \otimes h_2j^{-1}(h_3)_1=\epsilon(h)(1_A \otimes 1_H),
            \end{split}
            \label{eq:convolinvprop12}
        \end{equation}
        where in the last equality we have used Equation (\ref{eq:convolinvprop11}). Thus, $\Delta_A \circ j^{-1}$ is the convolution inverse of $f = \Delta_A \circ j = (j \otimes id) \circ \Delta$ in the convolution algebra $(Hom_{\mathbb{K}}(H, A \otimes H),*)$. \\
        Consider now $g:=(j^{-1} \otimes S) \circ flip \circ \Delta$. We have
        \begin{equation}
            \begin{split}
                &(f*g)(h)=\mu \circ \big(f(h_1) \otimes g(h_2)\big)= f(h_1)g(h_2)=\\
                &=\big(j(h_{11}) \otimes h_{12}\big)\big(j^{-1}(h_{22}) \otimes S(h_{21})\big)=\\
                &=\big(j(h_1) \otimes h_2\big)\big(j^{-1}(h_4) \otimes S(h_3)\big)= \\
                &=j(h_1)j^{-1}(h_4) \otimes h_2S(h_3)= j(h_1)j^{-1}(h_2) \otimes \epsilon(h)1_H =\\
                &=\epsilon(h)1_A \otimes \epsilon(h)1_H=\epsilon(h)(1_A \otimes 1_H),
            \end{split}
            \label{eq:convolinvprop13}
        \end{equation}
        which shows also $g$ is a convolution inverse of $f$. Thus, by the uniqueness of the inverse, we have that $\Delta_A \circ j^{-1}= (j^{-1} \otimes S) \circ flip \circ \Delta$, proving the first statement.
        \item Using, the previous point, by a direct computation we find
        \begin{equation}
            \begin{split}
                &\Delta_A \circ \big(\mu \circ (id \otimes j^{-1}) \circ \Delta_A(a)\big)=\\
                &\Delta_A\big(a_0j^{-1}(a_1)\big)=\Delta_A(a_0)\Delta_A\big(j^{-1}(a_1)\big) =\\
                &=(a_{00}\otimes a_{01})\big(j^{-1}(a_{12}) \otimes S(a_{11})\big)=\\
                &=a_{00}j^{-1}(a_{12}) \otimes a_{01}S(a_{11})=a_0j^{-1}(a_3) \otimes a_{1}S(a_2) =\\
                &=a_0j^{-1}(a_1) \otimes 1_H
            \end{split}
            \label{eq:convolinvprop2}
        \end{equation}
        for every $a \in A$, showing us that $a_0j^{-1}(a_1) \in B=A^{coH}$.
        \item Since $j^{-1}$ is the convolution inverse of $j$ we can write
        \begin{equation}
            \begin{split}
                &(j * j^{-1})(1_H)= j(1_H)j^{-1}(1_H)= \epsilon(1_H)1_A = 1_A=\\
                &=(j^{-1}*j)(1_H)=j^{-1}(1_H)j(1_H),
            \end{split}
            \label{eq:convolinvprop3}
        \end{equation}
        from which we get that $j^{-1}(1_H)$ is the inverse of $j^{-1}(1_H)$ in $A$.
        \item From the previous point it follows that
        \begin{equation}
            \begin{split}
                &(j^{\prime}*j^{\prime-1})(h)=j^{\prime}(h_1)j^{\prime-1}(h_2)=\\
                &=j^{-1}(1_H)j(h_1)j^{-1}(h_2)j(1_H)=j^{-1}(1_H)\epsilon(h)1_Aj(1_H)=\\
                &=\epsilon(h)j^{-1}(1_H)j(1_H)=\epsilon(h)1_A.
            \end{split}
            \label{eq:convolinvprop4}
        \end{equation}
        Similarly, one gets $(j^{\prime-1}*j^{\prime})(h)=\epsilon(h)1_A$.
        \item If $j$ is an algebra map, we have by Proposition \ref{pro:inveritbilityinconvolutionalgebra} that its convolution inverse is $j^{-1}=j \circ S$. In particular, using the properties of the antipode and the fact that $j$ is an algebra morphism, we can write
        \begin{equation}
            \begin{split}
                &j^{-1}(hh^{\prime})= j\big(S(hh^{\prime})\big)= j\big(S(h^{\prime})S(h)\big) =\\
                &=j\big(S(h^{\prime})\big)j\big(S(h)\big)=j^{-1}(h^{\prime})j^{-1}(h)
            \end{split} 
            \label{eq:convolinvprop51}
        \end{equation}
        for every $h,h^{\prime} \in H$. Moreover, setting $h^{\prime}=1_H$ we find
        \begin{equation}
            j^{-1}(h1_H)=j^{-1}(h)=j^{-1}(1_H)j^{-1}(h),
            \label{eq:convolinvprop52}
        \end{equation}
        which implies $j^{-1}(1_H)=1_A$.
    \end{enumerate}
\end{proof}
We now define the notion of extension, which is crucial in defining quantum principal bundles.
\begin{definition}
    We call $B = A^{coH} \subseteq A$
    \begin{enumerate}
        \item a \textsl{trivial extension} if there is a convolution invertible morphism $j \colon H \to A$ of right $H$-comodule algebras;
        \item a \textsl{cleft extension} if there is a convolution invertible morphism $j \colon H \to A$ of right $H$-comodules, called \textsl{cleaving map};
        \item a \textsl{Hopf-Galois extension} if the \textsl{canonical map}
        \begin{equation}
            \begin{split}
                \chi \colon &A \otimes_B A \to A \otimes H,\\
                &a \otimes_B a^{\prime} \mapsto aa^{\prime}_0 \otimes a^{\prime}_1
            \end{split}
            \label{eq:HopfGaloisextensioncanonicalmap}
        \end{equation}
        is invertible.
    \end{enumerate}
    \label{def:extensions}
\end{definition}
\begin{remark}
    Elements of $A \otimes_B A$ satisfy $a \otimes_B ba^{\prime}=ab \otimes_B a^{\prime}$ for every $a,a^{\prime} \in A$ and $b \in B$. \\
    The canonical map $\chi(a \otimes_B a^{\prime}) := aa^{\prime}_0 \otimes a^{\prime}_1$ is well defined on $A \otimes_B A$. To show this, we check that it acts the same way on $a \otimes_B ba^{\prime}$ and on $ab \otimes_B a^{\prime}$. Notice that
    \begin{equation}
        \begin{split}
             &(ba^{\prime})_0 \otimes (ba^{\prime})_1 = \Delta_A(ba^{\prime})= \Delta_A(b)\Delta_A(a^{\prime})= \\
             &=(b \otimes 1_H)(a^{\prime}_0 \otimes a^{\prime}_1) = ba^{\prime}_0 \otimes a^{\prime}_1,
        \end{split}
        \label{eq:canonicalmapHGwelldefined1}
    \end{equation}
    from which:
    \begin{equation}
        \begin{split}
             &\chi(a \otimes_B ba^{\prime})= a(ba^{\prime})_0 \otimes (ba^{\prime})_1=a(ba^{\prime}_0) \otimes a^{\prime}_1=\\
             &=(ab)a^{\prime}_0 \otimes a^{\prime}_1= \chi(ab \otimes_B a^{\prime}).
        \end{split}
        \label{eq:canonicalmapHGwelldefined2}
    \end{equation}
    \label{rem:onthequotienttensorproductotimesB}
\end{remark}
\begin{proposition}
    Every trivial extension is also a cleft extension. Ever cleft extension is also a Hopf-Galois extension.
    \label{pro:trivialarecleftareHG}
\end{proposition}
\begin{proof}
    The first statement is obviously satisfied.\\
    Let now $B=A^{coH} \subseteq A$ be a cleft extension, $j \colon H \to A$ be the cleaving map with convolution inverse $j^{-1} \colon H \to A$. Define
    \begin{equation}
        \begin{split}
            \chi^{-1} \colon &A \otimes H \to A \otimes_B A,\\
            &a \otimes h \mapsto aj^{-1}(h_1) \otimes_B j(h_2).
        \end{split}
        \label{eq:cleftareHGproof1}
    \end{equation}
    Since $j$ is a right $H$-comodule morphism we have $\Delta_A(j(h))=j(h)_0 \otimes j(h)_1=j(h_1) \otimes h_2$ and we can write:
    \begin{equation}
        \begin{split}
            &\chi\big(\chi^{-1}(a \otimes h)\big)=\chi\big(aj^{-1}(h_1) \otimes_B j(h_2)\big)=\\
            &=aj^{-1}(h_1)j(h_2)_0 \otimes j(h_2)_1 = aj^{-1}(h_1)j(h_{21}) \otimes h_{22}=\\
            &=aj^{-1}(h_1)j(h_2) \otimes h_3= \epsilon(h)a \otimes h=a \otimes h,
        \end{split}
        \label{eq:cleftareHGproof2}
    \end{equation}
    where in the last equality we have used the counit property. On the other hand, by Lemma \ref{lem:convolutioninvertiblecomodulemorphismproperties} $a_0j^{-1}(a_1) \in B$ for any $a \in A$, hence
    \begin{equation}
        \begin{split}
            &\chi^{-1}(\chi(a \otimes_B a^{\prime}))=\chi^{-1}(aa^{\prime}_0 \otimes a^{\prime}_1) = \\
            &= (aa^{\prime}_0)j^{-1}(a^{\prime}_{11}) \otimes_B j(a^{\prime}_{12}) = a\big(a^{\prime}_0j^{-1}(a^{\prime}_1)\big) \otimes_B j(a^{\prime}_2)=\\
            &= a \otimes_B \big(a^{\prime}_0j^{-1}(a^{\prime}_1)\big)j(a^{\prime}_2) = a \otimes_B a^{\prime}.
        \end{split}
        \label{eq:cleftareHGproof3}
    \end{equation}
    We have thus shown $\chi$ is invertible with inverse $\chi^{-1}$ as defined in Equation (\ref{eq:cleftareHGproof1}), proving that every cleft extension is a Hopf-Galois extension.
\end{proof}

\subsection{Smashed product algebras} \label{smashedproductalgebras}
In this section we define the smashed product algebra and prove its correspondence with the case of a trivial extension. \\
In the following, let $B$ be an algebra.
\begin{definition}
    We say that $H$ \textsl{measures} $B$ if there is a $\mathbb{K}$-linear map
    \begin{equation}
        \begin{split}
            \triangleright \colon &H \otimes B \to B,\\
            &h \otimes b \to h \triangleright b,
        \end{split}
        \label{eq:cdotmeasures1}
    \end{equation}
    such that
    \begin{equation}
        \begin{split}
            &h \triangleright 1_B = \epsilon(h)1_B;\\
            &h \triangleright (bb^{\prime}) = (h_1\triangleright b)(h_2 \triangleright b^{\prime})
        \end{split}
        \label{eq:cdotmeasures2}
    \end{equation}
    for every $h \in H$ and $b,b^{\prime} \in B$.
    \label{def:measures}
\end{definition}
Note that $\triangleright \colon H \otimes B \to B$ is not assumed to be a left $H$-action, i.e., for $H$ to measure $B$, it is not required that $h \triangleright (h^{\prime} \triangleright b) = hh^{\prime} \triangleright b$.    
\begin{definition}
    Assume $H$ measures $B$. Consider a map $\sigma \colon H \otimes H \to B$.
    \begin{enumerate}
        \item We call $\sigma \colon H \otimes H \to B$ a $2$\textsl{-cocyle} with values in $B$ if it is a convolution invertible morphism such that
        \begin{equation}
            \begin{split}
                &\sigma(h \otimes 1_H) = \epsilon(h)1_B = \sigma(1_H \otimes h),\\
                &\big(h_1 \triangleright \sigma(h^{\prime}_1 \otimes h^{\prime\prime}_1)\big) \sigma(h_2 \otimes h^{\prime}_2 h^{\prime\prime}_2 ) = \sigma(h_1 \otimes h^{\prime}_1 )\sigma(h_2 h^{\prime}_2 \otimes h^{\prime\prime}),
            \end{split}
            \label{eq:2cocycle}
        \end{equation}
        for all $h,h^{\prime},h^{\prime\prime} \in H$ and $b \in B$.
    \item We call $B$ a $\sigma$\textsl{-twisted left} $H$\textsl{-module} if there is a $2$-cocycle $\sigma \colon H \otimes H \to B$ with values in $B$ such that
    \begin{equation}
            \begin{split}
                &1_H \triangleright b = b,\\
                &h \triangleright (h^{\prime} \triangleright b) = \sigma(h_1 \otimes h^{\prime}_1)\big((h_2 h^{\prime}_2) \triangleright b\big)\sigma^{-1}(h_3 \otimes h^{\prime}_3),
            \end{split}
            \label{eq:sigmatwistedleftHmodule2cocycleproperties}
        \end{equation}
        for all $h, h^{\prime} \in H$ and $b \in B$.
        \end{enumerate}
    \label{def:2cocycletwisted}
\end{definition}
As shown in \cite{del_donno_durdevic}, Lemma 2.2.3, one can define, given a $\sigma$-twisted left $H$-module $B$, an associative unital product
\begin{equation}
        \begin{split}
            \mu_{\#_{\sigma}} \colon &(B \otimes H) \otimes (B \otimes H) \to B \otimes H,\\
            &(b \otimes h) \otimes (b^{\prime} \otimes h^{\prime}) \mapsto b(h_1 \triangleright b^{\prime})\sigma(h_2 \otimes h^{\prime}_1) \otimes h_3 h^{\prime}_2,
        \end{split}
        \label{eq:crossedproduct}
\end{equation}
on $B \otimes H$, with unit $1_B \otimes 1_H$. The algebra $B \#_{\sigma} H := (B \otimes H, \mu_{\#_{\sigma}})$ is called the \textsl{crossed product algebra}.\\
The map
\begin{equation}
    \begin{split}
        \sigma \colon &H \otimes H \to B,\\
        &h \otimes h^{\prime} \mapsto \epsilon(hh^{\prime})1_B,
    \end{split}
    \label{eq:trivial2cocycle}
\end{equation}
is trivially a $2$-cocycle, and is such that the corresponding $\sigma$-twisted left $H$-module $B$, by Definition \ref{def:2cocycletwisted}, satisfies
\begin{equation}
    \begin{split}
        &1_H \triangleright b = b,\\
        &h \triangleright (h^{\prime} \triangleright b) = (hh^{\prime}) \triangleright b,
    \end{split}
    \label{eq:trivial2cocycleproperties}
\end{equation}
for all $b \in B$ and $h,h^{\prime} \in H$, so that $\triangleright$ is a left $H$-action and $B$ is a left $H$-module algebra. \\
The product obtained by using this trivial $2$-cocycle in Equation (\ref{eq:crossedproduct}) defines the smashed product algebra.
\begin{definition}
    Let $B$ be a left $H$-module algebra, with left $H$-action $\triangleright$. The product
    \begin{equation}
        \begin{split}
            \mu_{\#} \colon &(B \otimes H) \otimes (B \otimes H) \to B \otimes H,\\
            &(b \otimes h) \otimes (b^{\prime} \otimes h^{\prime}) \mapsto b(h_1 \triangleright b^{\prime}) \otimes h_2h^{\prime}
        \end{split}
        \label{eq:smashproductalgebraproduct}
    \end{equation}
    makes $B\#H := (B \otimes H, \mu_{\#})$ an associative unital algebra, with unit $1_B \otimes 1_H$, called the \textsl{smashed product algebra}. We denote an element of $B\#H$ by $b\#h$ and we write $(b\#h)(b^{\prime}\#h^{\prime}) = b(h_1 \triangleright b^{\prime}) \otimes h_2h^{\prime}$.
    \label{def:smashproductalgebra}
\end{definition}
We can give $B\#H$ a natural right $H$-comodule structure by the right $H$-coaction
\begin{equation}
    \Delta_{B\#H} := id_B \otimes \Delta \colon B\#H \to B\#H \otimes H,
    \label{eq:naturalrightHcoactiononsmashproduct}
\end{equation}
identifying the subalgebra of coinvariant elements $(B\#H)^{coH} = B \otimes 1_H \cong B$.\\
We now show that the smashed product algebra identifies a trivial extension and viceversa.
\begin{proposition}
    Any smashed product algebra $B\#H$ is a trivial extension with convolution invertible morphism of right $H$-comodule algebras
    \begin{equation}
        \begin{split}
            j \colon &H \to B\#H,\\
            &h \mapsto 1_B \# h.
        \end{split}
        \label{eq:smashproductalgebraimpliestrivialmap}
    \end{equation}
    Conversely, a trivial extension $B \subseteq A$ with $j \colon H \to A$ an algebra map, defines a $\sigma$-twisted left $H$-module action
    \begin{equation}
        \begin{split}
            \triangleright \colon &H \otimes B \to B,\\
            &h \otimes b \mapsto h \triangleright b := j(h_1)bj^{-1}(h_2)
        \end{split}
        \label{eq:cleavingimpliestwistedleftHmoduleaction}
    \end{equation}
    on $B$, and the $2$-cocycle
     \begin{equation}
        \begin{split}
            \sigma \colon &H \otimes H \to B,\\
            &h \otimes h^{\prime} \mapsto \sigma(h \otimes h^{\prime}) := j(h_1) j(h^{\prime}_1)j^{-1}(h_2 h^{\prime}_2)
        \end{split}
        \label{eq:cleavingimplies2cocycle}
    \end{equation}
    is the trivial one, such that $A \cong B\#H$ are isomorphic as right $H$-comodule algebras.
    \label{pro:correspondencebetweentrivialextandsmashproductalg}
\end{proposition}
\begin{proof}
    Given a smashed product algebra $B\#H$ we consider the morphism of right $H$-comodule algebras
    \begin{equation}
        \begin{split}
          j \colon &H \to B\#H,\\
          &h \mapsto 1_B\#h.
        \end{split}
        \label{eq:Hcolinearmap1}
    \end{equation}
    This map is convolution invertible, with convolution inverse
    \begin{equation}
        \begin{split}
          j^{-1} \colon &H \to B\#H,\\
          &h \mapsto 1_B\# S(h).
        \end{split}
        \label{eq:Hcolinearmap2}
    \end{equation}
    In fact, we have
    \begin{equation}
        \begin{split}
            &(j*j^{-1}) (h) = j(h_1)j^{-1}(h_2)= (1_B\#h_1)\big(1_B\#S(h_2)\big) = \\
            &=1_B(h_{11} \triangleright 1_B)\#h_{12}S(h_2) = \epsilon(h)\big(1_B\#h_1S(h_2)\big) = \epsilon(h)1_B\#1_H
        \end{split}
        \label{eq:BdiesisthenBtrivial1}
    \end{equation}
    and
    \begin{equation}
        \begin{split}
            &(j^{-1}*j) (h) = j^{-1}(h_1)j(h_2)= \big(1_B\#S(h_1)\big)(1_B\#h_2) = \\
            &=1_B\big((S(h_1))_1 \triangleright 1_B\big)\#(S(h_1))_2h_2 = \big(S(h_{12}) \triangleright 1_B\big)\#S(h_{11})h_2 =\\
            &=\epsilon(h)1_B\#S(h_1)h_2 = \epsilon(h)1_B\#1_H.
        \end{split}
        \label{eq:BdiesisthenBtrivial2}
    \end{equation}
    Conversely, given $j \colon H \to A$ a convolution invertible morphism of right $H$-comodule algebras with convolution inverse $j^{-1}$, the map $\triangleright \colon H \otimes B \to B$ defined as in Equation (\ref{eq:cleavingimpliestwistedleftHmoduleaction}) is a $H$-measure on $B$. In fact
    \begin{equation}
        h \triangleright 1_B = j(h_1)j^{-1}(h_2) = \epsilon(h)1_B
        \label{eq:jtrivialthensmash2}
    \end{equation}
    and
    \begin{equation}
        \begin{split}
            &(h_1 \triangleright b)(h_2 \triangleright b^{\prime}) = j(h_{11})bj^{-1}(h_{12})j(h_{21})b^{\prime}j^{-1}(h_{22}) = \\
            &=j(h_1)bj^{-1}(h_2)j(h_3)b^{\prime}j^{-1}(h_4) = j(h_1)bb^{\prime}j^{-1}(h_2) =\\
            &=h \triangleright (bb^{\prime}).
        \end{split}
        \label{eq:jtrivialthensmash3}
    \end{equation}
    Moreover, the map $\sigma \colon H \otimes H \to B$ defined as in Equation (\ref{eq:cleavingimplies2cocycle}) is the trivial $2$-cocycle:
    \begin{equation}
        \begin{split}
            &\sigma(h \otimes h^{\prime}) = j(h_1) j(h^{\prime}_1)j^{-1}(h_2 h^{\prime}_2) = j(h_1) j(h^{\prime}_1)j^{-1}(h^{\prime}_2)j^{-1}(h_2) = \\
            &=\epsilon(hh^{\prime})1_B,
        \end{split}
        \label{eq:jtrivialthensmash4}
    \end{equation}
    where we used that since $j \colon H \to A$ is an algebra map by assumption, then $j^{-1}$ is an anti-algebra map by Lemma \ref{lem:convolutioninvertiblecomodulemorphismproperties}. Thus, $\triangleright$ is also left $H$-action and $A \cong B\#H$.
\end{proof}
\begin{remark}
    One could prove that (see \cite{del_donno_durdevic} Theorem 2.2.5), more in general, any cleft extension $B \subseteq A$ defines an $H$-measure on $B$ as in Equation (\ref{eq:cleavingimpliestwistedleftHmoduleaction}) and a $2$-cocycle as in Equation (\ref{eq:cleavingimplies2cocycle}), with $j \colon H \to A$ the cleaving map, so that $A$ is isomorphic to the crossed product algebra $B\#_{\sigma}H$, with product as in Equation (\ref{eq:crossedproduct}). Moreover, given the crossed product algebra $B\#_{\sigma}H$, one can prove that $j \colon H \to B\#_{\sigma}H, h \mapsto 1_B\#_{\sigma}h$ is a cleaving map. 
    \label{rem:onthemoregeneraltheoremaboutcleftcrossed}
\end{remark}

\section{First order differential calculi} \label{differentialcalculi}

\begin{definition}
    A \textsl{first order differential calculus} $(\Gamma, \mathrm{d})$ over a $\mathbb{K}$-algebra $A$ is the datum of
    \begin{enumerate}
        \item an $A$-bimodule $\Gamma$;
        \item a linear map $\mathrm{d} \colon A \to \Gamma$ satisfying the Leibniz rule
        \begin{equation}
            d(ab)=(da)b + a(db)
            \label{eq:firstorderdiffcalcLeibnizrule}
        \end{equation}
        for every $a,b \in A$, called the \textsl{differential};
        \item a surjectivity condition
        \begin{equation}
            \Gamma=AdA,
            \label{eq:firstorderdiffcalcsurjectivitycondition}
        \end{equation}
        i.e. $\Gamma = \text{span}\{ adb | a,b \in A \}$.
    \end{enumerate}
    \label{def:firstorderdifferentialcalculus}
\end{definition}
In classical differential geometry $A = C^{\infty}(M)$, and $a\mathrm{d}b = \mathrm{d}ba$ for every $a, b \in A$, which is not true anymore in the noncommutative setting.
\begin{definition}
    Let $(\Gamma, \mathrm{d})$ on $A$ and $(\Gamma^{\prime}, \mathrm{d}^{\prime})$ on $A^{\prime}$ be first order differential calculi. A \textsl{morphism of differential calculi} is a pair $(\Phi, \phi)$, where $\Phi \colon \Gamma \to \Gamma^{\prime}$ is a $\mathbb{K}$-linear map, and $\phi \colon A \to A^{\prime}$ is an algebra morphism, such that
    \begin{equation}
        \Phi(a \cdot \omega \cdot b) = \phi(a) \cdot^{\prime} \Phi(\omega) \cdot^{\prime} \phi(b),
        \label{eq:conditionfrodifferentialcalculimorphism}
    \end{equation}
    for all $a,b \in A$ and $\omega \in \Gamma$, with $\cdot$ denoting the left and right $A$-actions on $\Gamma$ and $\cdot^{\prime}$ the $A^{\prime}$-actions on $\Gamma^{\prime}$, and such that the diagram\\
    \[
    \begin{tikzcd}
    A \arrow{r}{\phi} \arrow{d}{\mathrm{d}} &A^{\prime} \arrow{d}{\mathrm{d}^{\prime}} \\
    \Gamma \arrow{r}{\Phi} & \Gamma^{\prime}
    \end{tikzcd}
    \]
    commutes, that is
    \begin{equation}
        \begin{split}
            \Phi \circ \mathrm{d} = \mathrm{d}^{\prime} \circ \phi.
        \end{split} 
        \label{eq:commutingdiagramdiffcalcmorphism}
    \end{equation}
    \label{def:morphismofdifferentialcalculi1}
\end{definition}
The following proposition shows that there always exists a first order differential calculus $\Gamma_u$ on an algebra $A$, called the universal differential calculus.
\begin{proposition}
    For every algebra $A$ there is a first order differential calculus $(\Gamma_u, \mathrm{d}_u)$ defined by
    \begin{equation}
        \Gamma_u := \ker{\mu_A} = \{ \sum_i a^i \otimes b^i \in A \otimes A | \sum_i a^ib^i=0  \},
        \label{eq:universalfirstorderdiffcalculus}
    \end{equation}
    with differential
    \begin{equation}
        \mathrm{d}_u(a):= 1_A \otimes a - a \otimes 1_A
        \label{eq:differentialuniversalfirst}
    \end{equation}
    for every $a \in A$.
    \label{pro:universaldifferentialcalculus}
\end{proposition}
\begin{proof}
    $\Gamma_u = \ker{\mu_A}$ has $A$-bimodule structure given by the left product on the first factor and the right product on the second factor of $A \otimes A$. \\
    The map $\mathrm{d}_u$ as defined in (\ref{eq:differentialuniversalfirst}) has values in $\Gamma_u = \ker{\mu_A}$ since
    \begin{equation}
        \mu_A\big(\mathrm{d}_u(a)\big) = \mu_A(1_A \otimes a - a \otimes 1_A) = a - a = 0,
        \label{eq:fodcuniversal1}
    \end{equation}
    and it satisfies the Leibniz rule:
    \begin{equation}
        \begin{split}
            &\mathrm{d}_u(a)\cdot a^{\prime} + a \cdot \mathrm{d}_u(a^{\prime}) = (1_A \otimes a - a \otimes 1_A)\cdot a^{\prime} + a\cdot (1_A \otimes a^{\prime} - a^{\prime} \otimes 1_A) = \\
            &= 1_A \otimes aa^{\prime} - a \otimes a^{\prime} + a \otimes a^{\prime} - aa^{\prime} \otimes 1_A = 1_A \otimes aa^{\prime} - aa^{\prime} \otimes 1_A =\\
            &= \mathrm{d}_u(aa^{\prime}),
        \end{split} 
        \label{eq:fodcuniversal2}
    \end{equation}
    for all $a,a^{\prime} \in A$. \\
    Let $\sum_i a^i \otimes b^i$ a generic element of $\Gamma_u := \ker{\mu_A}$, that is $\sum_i a^ib^i=0$. Then
    \begin{equation}
        \begin{split}
            &\sum_i a^i \mathrm{d}_u(b^i) = \sum_i a^i (1_A \otimes b^i - b^i \otimes 1_A) = \sum_i a^i \otimes b^i - a^ib^i \otimes 1_A = \\
            &= \sum_i a^i \otimes b^i,
        \end{split}
        \label{eq:fodcuniversal3}
    \end{equation}
    proving the surjectivity condition.
\end{proof}
\begin{remark}
    It is worth noting that every first order differential calculus $(\Gamma, \mathrm{d})$ over a unital $\mathbb{K}$-algebra $A$ can be obtained as a quotient of the universal differential calculus $(\Gamma_u, \mathrm{d}_u)$, as proven by Woronowicz \cite{worono}. In particular, there exists an $A$-subbimodule $N \subseteq \Gamma_u$ such that $\Gamma \cong \Gamma_u/N$ as $A$-bimodules, with the quotient map $\pi \colon \Gamma_u \to \Gamma_u/N$ being the unique surjective morphism such that $\mathrm{d} = \pi \circ \mathrm{d}_u$.
    \label{rem:onthefactthatanyfodcisaquotientoftheuniversal}
\end{remark}
Given a calculus on an algebra $A$, we can always induce a calculus on any of its subalgebras or quotient algebras. One denotes by $[a] = a + I$ an element of the quotient $A/I$, for $a \in A$ and $I$ an ideal. As in \cite{latiniweber_projective}, we present the following result.
\begin{proposition}
    Let $(\Gamma, \mathrm{d})$ be a first order differential calculus on an algebra $A$.
    \begin{enumerate}
        \item A morphism of algebras $\iota \colon A^{\prime} \to A$ induces a first order differential calculus $(\Gamma_{\iota}, \mathrm{d}_{\iota})$ on $A^{\prime}$ with
        \begin{equation}
            \Gamma_{\iota} := \iota(A)\mathrm{d}\iota(A) \subseteq \Gamma
            \label{eq:pullbackcalculus}
        \end{equation}
        and
        \begin{equation}
            \begin{split}
                \mathrm{d}_{\iota} := \mathrm{d} \circ \iota \colon &A^{\prime} \to \Gamma_{\iota},\\
                &a^{\prime} \mapsto \mathrm{d}\iota(a^{\prime}). 
            \end{split}
            \label{eq:pullbackcalculusdifferential}
        \end{equation}
        \item A surjective algebra morphism $\phi \colon A \to A_0$, such that we can identify $A_0$ with the quotient $A/I$ by the ideal $I := \ker{\phi}$, induces a first order differential calculus $(\Gamma_{\phi}, \mathrm{d}_{\phi})$ on $A_0$, where the $A_0$-bimodule $\Gamma_{\phi}$ is the quotient
        \begin{equation}
            \Gamma_{\phi} := \Gamma/\Gamma_I,
            \label{eq:quotientcalculus}
        \end{equation}
        with
        \begin{equation}
            \Gamma_I := I\mathrm{d}A + A\mathrm{d}I,
            \label{eq:idealcalculus}
        \end{equation}
        and
        \begin{equation}
            \begin{split}
                \mathrm{d}_{\phi} \colon &A_0 \to \Gamma_{\phi},\\
                &\phi(a) \mapsto [\mathrm{d}a].
            \end{split}
            \label{eq:quotientcalculusdifferential}
        \end{equation}
    \end{enumerate}
    \label{pro:pullbackcalculusandquotientcalculus}
\end{proposition}
\begin{proof}
    In order:
    \begin{enumerate}
        \item $\Gamma_{\iota}$ in Equation (\ref{eq:pullbackcalculus}) has an $A^{\prime}$-bimodule structure, defined as
        \begin{equation}
            a^{\prime} \cdot \omega \cdot b^{\prime} := \iota(a^{\prime})\omega\iota(b^{\prime})
            \label{eq:pullbackcalculus1}
        \end{equation}
        for every $a^{\prime},b^{\prime} \in A^{\prime}$ and $\omega \in \Gamma$. We need to check that $\cdot$ maps close in $\Gamma_{\iota}$. A generic element $\omega \in \Gamma_{\iota}$ is written, by definition, as $\omega = \sum_i \iota(a^{\prime i})\mathrm{d}\iota(b^{\prime i})$ with $a^{\prime i}, b^{\prime i} \in A^{\prime}$. Then, for all $a^{\prime}, b^{\prime} \in A^{\prime}$, we have
        \begin{equation}
            \begin{split}
                &a^{\prime} \cdot \omega \cdot b^{\prime} = \sum_i \iota(a^{\prime}) \iota(a^{\prime i})\big(\mathrm{d}\iota(b^{\prime i})\big) \iota(b^{\prime}) = \\
                &= \sum_i \iota(a^{\prime}) \iota(a^{\prime i}) \mathrm{d}\big(\iota(b^{\prime i})  \iota(b^{\prime})\big) - \sum_i \iota(a^{\prime}) \iota(a^{\prime i}) \iota(b^{\prime i}) \mathrm{d}\iota(b^{\prime}) = \\
                &= \sum_i \iota(a^{\prime}a^{\prime i}) \mathrm{d}\big(\iota(b^{\prime i}b^{\prime})\big) - \sum_i \iota(a^{\prime}a^{\prime i}b^{\prime i})  \mathrm{d}\iota(b^{\prime}) \in \Gamma_{\iota},              
            \end{split}
            \label{eq:pullbackcalculus2}
        \end{equation}
        using the Leibniz rule for $\mathrm{d}$ and the fact that $\iota$ is an algebra morphism. The maps $\cdot$ in (\ref{eq:pullbackcalculus1}) are left and right $A^{\prime}$-actions since
        \begin{equation}
            a^{\prime} \cdot (b^{\prime} \cdot \omega) = \iota(a^{\prime})\iota(b^{\prime})\omega = \iota(a^{\prime}b^{\prime})\omega = a^{\prime}b^{\prime} \cdot \omega
            \label{eq:pullbackcalculus3}
        \end{equation}
        and
        \begin{equation}
            1_{A^{\prime}} \cdot \omega = \iota(1_{A^{\prime}})\omega = \omega,
            \label{eq:pullbackcalculus4}
        \end{equation}
        and similarly for the right $A^{\prime}$-action. Finally, we have $(a^{\prime} \cdot \omega) \cdot b^{\prime} = \big(\iota(a^{\prime})\omega\big)\iota(b^{\prime}) = \iota(a^{\prime})\big(\omega\iota(b^{\prime})\big) = a^{\prime} \cdot (\omega \cdot b^{\prime})$, so that $\Gamma_{\iota}$ is an $A^{\prime}$-bimodule.\\
        The map $\mathrm{d}_{\iota} := \mathrm{d} \circ \iota$ satisfies the Leibniz rule:
        \begin{equation}
            \begin{split}
                &\mathrm{d}_{\iota}(a^{\prime}b^{\prime}) = \mathrm{d}\iota(a^{\prime}b^{\prime}) = \mathrm{d}\big(\iota(a^{\prime})\iota(b^{\prime})\big) = \\    
                &= \big(\mathrm{d}\iota(a^{\prime})\big)\iota(b^{\prime}) + \iota(a^{\prime})\mathrm{d}\iota(b^{\prime}) = \mathrm{d}_{\iota}(a^{\prime}) \cdot b^{\prime} + a^{\prime} \cdot \mathrm{d}_{\iota}(b^{\prime}).
            \end{split}
            \label{eq:pullbackcalculus5}
        \end{equation}
        The surjectivity condition is satisfied since a generic element in $\Gamma_{\iota}$ can be written as $\omega = \sum_i \iota(a^{\prime i})\mathrm{d}\iota(b^{\prime i}) = \sum_i a^{\prime i} \cdot \mathrm{d}b^{\prime i}$.
        \item The subset $I = \ker{\phi} \subset A$ is an ideal since for any $x,y \in \ker{\phi}$ we have $\phi(x + y) = 0$ and for any $a,b \in A$ and $x \in \ker{\phi}$ we have $\phi(xa) = \phi(x)\phi(a) = 0 = \phi(a)\phi(x) = \phi(ax)$. Hence, $\Gamma_I := I\mathrm{d}A + A\mathrm{d}I$ is an $A$-subbimodule by the Leibniz rule of $\mathrm{d}$. Thus, $\Gamma_{\phi} := \Gamma/\Gamma_I$ is an $A$-bimodule since, denoting by $[\omega] = \omega + I$ any element of $\Gamma_{\phi}$, we can write for all $a \in A$
        \begin{equation}
            a \cdot [\omega] = a \cdot (\omega + I) = a\omega + I \in \Gamma_{\phi}
            \label{eq:quotientcalculus1}
        \end{equation}
        and similarly for the right $A$-action. Since $I \cdot \Gamma \subseteq \Gamma_I$, $\Gamma \cdot I \subseteq \Gamma_I$, the $A$-actions on $\Gamma$ descend to $A_0$-actions on $\Gamma_{\phi}$, with $A_0 = A/I$, as
        \begin{equation}
            \phi(a) \cdot [\omega] \cdot \phi(b) := [a\omega b]
            \label{eq:quotientcaculus2}
        \end{equation}
        for all $a,b \in A$ and $\omega \in \Gamma$. \\
        The map $\mathrm{d}_{\phi}\big(\phi(a)\big) := [\mathrm{d}a]$ satisfies the Leibniz rule. In fact, since $\phi$ is an algebra morphism:
        \begin{equation}
            \begin{split}
                &\mathrm{d}_{\phi}\big(\phi(a)\phi(b)\big) = \mathrm{d}_{\phi}\big(\phi(ab)\big) = [\mathrm{d}(ab)] = [(\mathrm{d}a)b] + [a\mathrm{d}b] =\\
                &= [\mathrm{d}a] \cdot b + a \cdot [\mathrm{d}b] = \mathrm{d}_{\phi}\big(\phi(a)\big) \cdot b + a \cdot \mathrm{d}_{\phi}\big(\phi(b)\big).
            \end{split}
            \label{eq:quotientcaculus3}
        \end{equation}
        for all $a, b \in A$.\\
        The surjectivity condition is satisfied since a generic element of $\Gamma_{\phi}$ is written as $[\omega] = [\sum_i a^{i}\mathrm{d}b^i] = \sum_i \phi(a^i) \cdot \mathrm{d}_{\phi}\big(\phi(b^i) \big) \in A_0\mathrm{d}_{\phi}A_0$.
    \end{enumerate}
\end{proof}
We call $(\Gamma_{\iota}, \mathrm{d}_{\iota})$ and $(\Gamma_{\phi},\mathrm{d}_{\phi})$ as in the above proposition the \textsl{pullback calculus} and \textsl{quotient calculus} respectively.

\subsection{First order $H$-covariant calculi} \label{firstordercovariantcalculi}
Let $H$ be a Hopf algebra and $A$ an $H$-comodule algebra.
\begin{definition}
    A first order differential calculus $(\Gamma, \mathrm{d})$ on a right $H$-comodule algebra $(A, \Delta_A)$ is called a \textsl{right} $H$\textsl{-covariant first order differential calculus} if $\Gamma$ is a \textsl{right} $H$\textsl{-covariant} $A$-bimodule, i.e., it has coaction
    \begin{equation}
        \begin{split}
            \Delta_{\Gamma} \colon &\Gamma \to \Gamma \otimes H,\\
            &a\mathrm{d}b \mapsto a_0\mathrm{d}b_0 \otimes a_1b_1
        \end{split}
        \label{eq:coactionforrightHcovatiantfirstorderdiffcalc}
    \end{equation}
    for all $a,b \in A$, such that the differential $d \colon A \to \Gamma$ is a morphsim of right $H$-comodules, that is
    \begin{equation}
        \Delta_{\Gamma} \circ \mathrm{d} = (\mathrm{d} \otimes id_H) \circ \Delta_A,
        \label{eq:rightHcovariantdiffcolinearity}
    \end{equation}
    Similarly, it is called a \textsl{left} $H$\textsl{-covariant first order differential calculus} if $\Gamma$ is a \textsl{left} $H$\textsl{-covariant} $A$-bimodule, i.e., it has coaction
    \begin{equation}
        \begin{split}
            \leftindex_{\Gamma}\Delta \colon &\Gamma \to H \otimes \Gamma,\\
            &a\mathrm{d}b \mapsto a_{-1}b_{-1} \otimes a_0\mathrm{d}b_0
        \end{split}
        \label{eq:coactionforleftHcovatiantfirstorderdiffcalc}
    \end{equation}
    for all $a,b \in A$, such that the differential $\mathrm{d} \colon A \to \Gamma$ is a morphsim of left $H$-comodules, that is 
    \begin{equation}
        \leftindex_{\Gamma}\Delta \circ \mathrm{d} = (id_H \otimes \mathrm{d}) \circ \leftindex_A\Delta,
        \label{eq:leftHcovariantdiffcolinearity}
    \end{equation}
    It is called an $H$\textsl{-bicovariant first order differential calculus} if it is both left and right $H$-covariant.
    \label{def:leftrightHcovariantdiffcalculi}
\end{definition}
The following proposition (see \cite{latiniweber_projective} Proposition 3.7) characterizes the pullback calculus and the quotient calculus on an $H$-comodule algebra as $H$-covariant first order differential calculi.
\begin{proposition}
    Let $(\Gamma, \mathrm{d})$ be a right $H$-covariant first order differential calculus on a right $H$-comodule algebra $A$.
    \begin{enumerate}
        \item If $\iota \colon A^{\prime} \to A$ is a right $H$-comodule algebra morphism then the pullback calculus $(\Gamma_{\iota}, \mathrm{d}_{\iota})$ is a right $H$-covariant first order differential calculus on  $A^{\prime}$.
        \item If $\phi \colon A \to A_0$ is a surjective right $H$-comodule algebra morphism then the quotient calculus $(\Gamma_{\phi}, \mathrm{d}_{\phi})$ is a right $H$-covariant first order differential calculus on $A_0$.
    \end{enumerate}
    \label{pro:pullbackandquotientcovariant}
\end{proposition}
\begin{proof}
    In order:
    \begin{enumerate}
        \item By Proposition \ref{pro:pullbackcalculusandquotientcalculus}, $(\Gamma_{\iota}, \mathrm{d}_{\iota})$ is a first order differential calculus on $A^{\prime}$. For all $a^{\prime},b^{\prime} \in A^{\prime}$ we have
        \begin{equation}
            \iota(a^{\prime})_0\mathrm{d}\iota(b^{\prime})_0 \otimes \iota(a^{\prime})_1\iota(b^{\prime})_1 = \iota(a^{\prime}_0)\mathrm{d}\iota(b^{\prime}_0) \otimes a^{\prime}_1b^{\prime}_1 \in \Gamma_{\iota} \otimes H,
            \label{eq:pullbackcovariant1}
        \end{equation}
        where we used that $\iota$ is a morphism of right $H$-comodules, so that the right $H$-coaction $\Delta_{\Gamma}$ on $\Gamma$ restricted to $\Gamma_{\iota} \subseteq \Gamma$ makes the pullback calculus a right $H$-covariant differential calculus on $A^{\prime}$. 
        \item By Proposition \ref{pro:pullbackcalculusandquotientcalculus}, $(\Gamma_{\phi}, \mathrm{d}_{\phi})$ is a first order differential calculus on $A_0$. Since $\phi$ is a morphism of right $H$-comodules, we have that $\Delta_A(I) \subseteq I \otimes H$, with $I = \ker{\phi}$. Then $\Gamma_I = I\mathrm{d}A + A\mathrm{d}I \subseteq \Gamma$ is such that $\Delta_{\Gamma}(\Gamma_I) \subseteq \Gamma_I \otimes H$, so that $\Delta_{\Gamma} \colon \Gamma \to \Gamma \otimes H$ induces a well defined right $H$-coaction on $\Gamma_{\phi}=\Gamma/\Gamma_I$ as
        \begin{equation}
            \Delta_{\Gamma}([a\mathrm{d}b]) = [a_0 \mathrm{d}b_0] \otimes a_1b_1 = \phi(a_0) \cdot \mathrm{d}_{\phi}\phi(b_0) \otimes a_1b_1
            \label{eq:quotientcovariant1}
        \end{equation}
        for all $a,b \in A$. Thus, $(\Gamma_{\phi}, \mathrm{d}_{\phi})$ is a right $H$-covariant first order differential calculus on $A_0$.
    \end{enumerate}
\end{proof}
Analogous results hold for left $H$-covariant calculi on left $H$-comodule algebras.\\

Let us now fix $(\Gamma_H, \mathrm{d}_H)$ a left covariant first order differential calculus over a Hopf algebra $H$. We denote the module of left coinvariant forms by
\begin{equation}
    \Lambda^1 := \{ \omega \in \Gamma_H \;|\; \leftindex_{\Gamma_H}\Delta (\omega) = 1 \otimes \omega \}.
    \label{eq:moduleofleftcoinvariantforms}
\end{equation}
\begin{definition}
    We define the \textsl{quantum Maurer-Cartan form} as the canonical $\mathbb{K}$-linear map from the kernel $H^+ = \ker{\epsilon}$ of the counit $\epsilon \colon H \to \mathbb{K}$ to the coinvariant forms as
    \begin{equation}
        \begin{split}
            \varpi \colon &H^+ \to \Lambda^1,\\
            &h \mapsto S(h_1)\mathrm{d}_Hh_2.
        \end{split}
        \label{eq:quantumMCform}
    \end{equation}
    \label{def:quantumMCdef}
\end{definition}
The following lemma is used to prove the classification theorem, due to Woronowicz, which provides a correspondence between right ideals $I \subseteq H$ that are in $\ker{\epsilon}$ and left covariant first order differential calculi $\Gamma_H$ over $H$.
\begin{lemma}
    The quantum Maurer-Cartan form $\varpi \colon H^+ \to \Lambda^1$ is a surjective morphism and moreover $I = \ker{\varpi} \subseteq H^+ \subseteq H$ is a right ideal, i.e. it is a $\mathbb{K}$-vector subspace of $H$ and it is closed under right multiplication by $H$.
    \label{lem:qMCsurjectiveandkerideal}
\end{lemma}
\begin{proof}
    Let $\cdot \colon H \otimes \Gamma_H \to \Gamma_H$ be the left module action of $H$ on $\Gamma_H$. Considering a left-coinvariant form $\omega= a^i db^i \in \Lambda^1$, with $a^i,b^i \in H$, we find
    \begin{equation}
        \cdot \circ (S \otimes id) \circ \leftindex_{\Gamma_H} \Delta(\omega) = \cdot \circ (S(1_H) \otimes \omega) = \omega,
        \label{eq:qMCproof1}
    \end{equation}
    since $S(1_H)=1_H$ and by left $H$-coinvariance. Morever, since $\Gamma_H$ is a left $H$-covariant first order differential calculus and the coproduct $\Delta$ on $H$ can be seen as a left $H$-comodule action, we notice that
    \begin{equation}
        \begin{split}
            &\omega = \triangleright \circ (S \otimes id) \circ \leftindex_{\Gamma_H} \Delta(\omega)=\triangleright \circ (S \otimes id) \circ \leftindex_{\Gamma_H} \Delta(a^i\mathrm{d}_Hb^i)=\\
            &=\cdot \circ (S \otimes id) (a^i_{-1}b^i_{-1} \otimes a^i_0\mathrm{d}_Hb^i_0)=\\
            &=\cdot \circ \big(S(a^i_{-1}b^i_{-1}) \otimes a^i_0\mathrm{d}_Hb^i_0\big)=\cdot \circ \big(S(b^i_{-1})S(a^i_{-1}) \otimes a^i_0\mathrm{d}_Hb^i_0\big) =\\
            &=S(b^i_{-1})S(a^i_{-1})a^i_0\mathrm{d}_Hb^i_0 = \epsilon(a^i)S(b^i_{-1})\mathrm{d}_Hb^i_{0} = \epsilon(a^i)\varpi(b^i)= \\
            &=\varpi\big(\epsilon(a^i)b^i\big)=\varpi\big(\epsilon(a^i)b^i - \epsilon(a^ib^i)1_H\big)
        \end{split}
        \label{eq:qMCproof2}
    \end{equation}
    where in the last equality we added the term $-\epsilon(a^ib^i)1_H$, since $\varpi(1_H)=0$. The argument of $\varpi$ is in $H^+=\ker{\epsilon}$. In fact, since $\epsilon$ is an algebra map, we have:
    \begin{equation}
        \epsilon\big(\epsilon(a^i)b^i - \epsilon(a^ib^i)1_H\big) = \epsilon(a^i)\epsilon(b^i) - \epsilon(a^ib^i)\epsilon(1_H)=\epsilon(a^ib^i)-\epsilon(a^ib^i)=0.
        \label{eq:qMCproof3}
    \end{equation}
    Moreover, any element in $H^+$ can be written in this form, since the map $\pi^+ \colon H \to H^+, h \mapsto h - \epsilon(h)$ is surjective, and any $h \in H$ can be written as $\epsilon(a)b$ for some $a,b \in H$.
    Thus, by (\ref{eq:qMCproof2}), we have shown that $\varpi$ is surjective.\\
    For completeness, let us show here that any $c=h - \epsilon(h)1_H \in H^{+}$ (with $h \in H$) is sent to an element in $\Lambda^1$:
    \begin{equation}
        \begin{split}
            &\leftindex_{\Gamma_{H}}\Delta\big(\varpi(c)\big) = \leftindex_{\Gamma_{H}}\Delta\big(S(h_1)\mathrm{d}_Hh_2\big)= \Delta\big(S(h_1)\big)\leftindex_{\Gamma_{H}}\Delta\big(\mathrm{d}_Hh_2\big) = \\
            &=\big(S(h_{12}) \otimes S(h_{11})\big)\big((id \otimes \mathrm{d}_H) \circ \Delta(h_2)\big)=\\
            &=\big(S(h_{12}) \otimes S(h_{11})\big)(h_{21} \otimes \mathrm{d}_Hh_{22})=\\
            &=S(h_{12})h_{21} \otimes S(h_{11})\mathrm{d}_Hh_{22} = S(h_2)h_{3} \otimes S(h_1)dh_{4} =\\
            &=\epsilon(h)1_H \otimes S(h_1)dh_2.
        \end{split}
        \label{eq:robainpiuforsemia}
    \end{equation}
    To show that $I = \ker{\varpi}$ is a right ideal we simply consider $\eta \in I$ and $h \in H$, to obtain
    \begin{equation}
        \begin{split}
            &\varpi(\eta h)=S(\eta_1 h_1)\mathrm{d}_H(\eta_2 h_2) =\\
            &=S(h_1)S(\eta_1)\big(\mathrm{d}_H(\eta_2)h_2-\eta_2\mathrm{d}_H(h_2)\big)=\\
            &=S(h_1)S(\eta_1)\mathrm{d}_H(\eta_2)h_2-S(h_1)S(\eta_1)\eta_2\mathrm{d}_H(h_2)=\\
            &=S(h_1)\varpi(\eta)h_2-S(h_1)\epsilon(\eta)\mathrm{d}_H(h_2)=0,
        \end{split}
        \label{eq:qMCproof4}
    \end{equation}
    since $\varpi(\eta)=0$ because $\eta \in I$ and $\epsilon(\eta)=0$ because $\eta \in I=ker(\varpi)\subseteq H^+$, which shows $\eta h \in I$ for any $h \in H$, thus $I$ is a right ideal.
\end{proof}
Denote the quotient map by $\pi \colon H^+ \to H^+/I, \;h \mapsto \pi(h) = [h]$. The following theorem (see \cite{del_donno_durdevic} Theorem 3.1.12 for a detailed proof) provides a characterization of any $H$-covariant first order differential calculus on a Hopf algebra $H$.
\begin{theorem}
    For any right ideal $I \subseteq H$ with $I \subseteq H^+$ we have a left covariant first order differential calculus $(\Gamma_H, \mathrm{d}_H)$ on $H$, where
    \begin{equation}
        \begin{split}
            &\Gamma_H = H \otimes (H^+/I),\\
            &\mathrm{d}_Hh = (id \otimes \pi)\big(\Delta(h) - h \otimes 1\big). 
        \end{split}
        \label{eq:classthm1}
    \end{equation}
    $\Gamma_H$ is an $H$-bimodule by the following left and right $H$-module actions
    \begin{equation}
        \begin{split}
            &h \cdot (h^{\prime} \otimes [g]) = hh^{\prime} \otimes [g],\\
            &(h \otimes [g]) \cdot h^{\prime} = hh^{\prime}_1 \otimes [gh^{\prime}_2],
        \end{split}
        \label{eq:GammaHbimodule}
    \end{equation}
    with $h, h^{\prime} \in H$ and $g \in H^+$. The left $H$-coaction on $\Gamma_H$ is 
    \begin{equation}
        \leftindex_{\Gamma_{H}}\Delta = \Delta \otimes id_{H/I}.
        \label{eq:classthm2}
    \end{equation}
    If $ad_R(I) \subseteq I \otimes H$ (see Example \ref{ex:rightadjointcoactiongivesrightHcomodule}), $(\Gamma_H, \mathrm{d}_H)$ is bicovariant with right $H$-coaction
    \begin{equation}
        \Delta_{\Gamma} (h \otimes [g]) = (h_1 \otimes [g_2]) \otimes h_2 S(g_1) g_3,
        \label{eq:classthm3}
    \end{equation}
    for $h \in H$ and $g \in H^+$. Moreover, every $H$-covariant first order differential calculus on $H$ is of this form.
    \label{thm:classificationtheoremIguess}
\end{theorem}

\subsection{The smashed product calculus} \label{smashedproductcalculus}

In this section we show the construction of a covariant differential calculus on the smashed product algebra $B\#H$ (\cite{latiniweber_projective}, \cite{pflaum}) from an $H$-module calculus on an $H$-module algebra $B$ and a bicovariant calculus on the Hopf algebra $H$. If the $H$-action on $B$ is trivial, we recover the tensor product differential calculus on $B \otimes H$ described in the following proposition.
\begin{proposition}
    Given a first order differential calculus $(\Gamma, \mathrm{d})$ on $A$ and a first order differential calculus $(\Gamma^{\prime}, \mathrm{d}^{\prime})$ on $A^{\prime}$, there is a FODC ($\Gamma_{A\otimes A^{\prime}}, \mathrm{d}_{A\otimes A^{\prime}})$ on $A \otimes A^{\prime}$, where 
    \begin{equation}
        \Gamma_{A\otimes A^{\prime}} = \Gamma \otimes A^{\prime} \oplus A \otimes \Gamma^{\prime}    
        \label{eq:tensorproductcalculus}
    \end{equation}
    and 
    \begin{equation}
        \begin{split}
            \mathrm{d}_{A\otimes A^{\prime}} \colon &A \otimes A^{\prime} \to \Gamma_{A\otimes A^{\prime}},\\
            &a \otimes a^{\prime} \mapsto \mathrm{d}a \otimes a^{\prime} + a \otimes \mathrm{d}^{\prime}a^{\prime}.
        \end{split}
        \label{e:difffortensorproductdiffcalc}
    \end{equation}
The $A \otimes A^{\prime}$-bimodule structure on $\Gamma_{A \otimes A^{\prime}}$ is 
\begin{equation}
    (a \otimes a^{\prime}) \cdot ( \omega \otimes b^{\prime} + b \otimes \omega^{\prime}) \cdot (c \otimes c^{\prime}) = a \omega c \otimes a^{\prime} b^{\prime} c^{\prime} + abc \otimes a^{\prime} \omega^{\prime} c^{\prime},    \label{eq:bimodulestructureofdifftensorproductcalculus}
\end{equation}
where $a, b, c \in A$, $a^{\prime}, b^{\prime}, c^{\prime} \in A^{\prime}$ , $\omega \in \Gamma$ and $\omega^{\prime} \in \Gamma^{\prime}$. This construction is associative, i.e. $(\Gamma_{(A \otimes A^{\prime}) \otimes A^{\prime\prime}}, \mathrm{d}_{(A \otimes A^{\prime}) \otimes A^{\prime\prime}})= (\Gamma_{A \otimes (A^{\prime} \otimes A^{\prime\prime})}, \mathrm{d}_{A \otimes (A^{\prime} \otimes A^{\prime\prime})})$ for another algebra $A^{\prime\prime}$.
    \label{pro:Diffcalculusoftensorproduct}
\end{proposition}

\begin{definition}
    Let $B$ be a left $H$-module algebra, with action $\triangleright \colon H \otimes B \to B$. Let $M$ be a $B$-bimodule, with actions $\cdot \colon B \otimes M \to M$, $\cdot \colon M \otimes B \to M$, and a left $H$-module with action $\mathrel{\scriptstyle \to} \colon H \otimes M \to M$. We say that $M$ is a \textsl{relative} $H$\textsl{-module} $B$\textsl{-bimodule} if the $H$ and $B$ actions have the compatibility
    \begin{equation}
        h \mathrel{\scriptstyle \to} (b \cdot m \cdot b^{\prime} ) = (h_1 \triangleright b) \cdot (h_2 \mathrel{\scriptstyle \to} m) \cdot (h_3 \triangleright b^{\prime} ), 
        \label{eq:relativecompatibility}
    \end{equation}
    for all $h \in H$, $b,b^{\prime} \in B$, $m \in M$.
    \label{def:relativeHmoduleBbimodule}
\end{definition}
Similar definitions are given if $M$ is just a left or a right $B$-module.
\begin{definition}
    Let $B$ be a left $H$-module algebra with action $\triangleright \colon H \otimes B \to B$. A FODC $(\Gamma_B, d_B)$ on $B$ is said to be an $H$\textsl{-module FODC} if for any $b^i, b_i \in B$, $i = 1, 2, \dots, n$, ($n \in \mathbb{N}$) and $h \in H$ we have
    \begin{equation}
        \sum_i b^id_Bb_i = 0 \implies \sum_i(h_1 \triangleright b^i)d_B(h_2 \triangleright b_i)=0.
        \label{eq:HmoduleFODC}
    \end{equation}
    \label{def:HmoduleFODC}
\end{definition}
\begin{proposition}
    $(\Gamma_B, d_B)$ is an $H$-module FODC if and only if $\Gamma_B$ is a relative $H$-module $B$-bimodule and $d_B \colon B \to \Gamma_B$ is an $H$-module map: for all $h \in H$, $b \in B$, we have
    \begin{equation}
        h \mathrel{\scriptstyle \to} \mathrm{d}_B b = \mathrm{d}_B (h \triangleright b).
        \label{eq:fordBtobeanHmodulemap}
    \end{equation}
    we have a well-defined $H$-action given by,
    \begin{equation}
        \begin{split}
            &H \otimes \Gamma_B \to \Gamma_B,\\
            &h \mathrel{\scriptstyle \to} \sum_i(b^i\mathrm{d}_Bb_i) := \sum_i(h_1 \triangleright b^i)\mathrm{d}_B(h_2 \triangleright b_i),
        \end{split}
        \label{eq:HmoduleFODCwillimply}
    \end{equation}
    where $\sum_i b^i d_B b_i$ is a generic element of $\Gamma_B = B\mathrm{d}_BB$.
    \label{pro:relationbetween HmoduleFODCandrelativestuff}
\end{proposition}
Given an $H$-module algebra $B$, a first order differential calculus $(\Gamma_B , \mathrm{d}_B)$ on $B$ and a left covariant first order differential calculus $(\Gamma_H, \mathrm{d}_H)$ on $H$, we consider the $\mathbb{K}$-module $\Gamma_{\#} := \Gamma_B \otimes H \oplus B \otimes \Gamma_H$ and study when there is a first order differential calculus $(\Gamma_{\#}, \mathrm{d}_{\#})$ on $B\#H$.
\begin{lemma}
    The $\mathbb{K}$-module
    \begin{equation}
    \Gamma_{\#} := \Gamma_B \otimes H \oplus B \otimes \Gamma_H
    \label{eq:smashproductcalculusmaybe}
    \end{equation}
    is a $B\#H$-bimodule. The left $B\#H$-action on $\Gamma_{\#}$ is
    \begin{equation}
        (b\#h)\cdot (\omega_B \otimes h^{\prime} + b^{\prime} \otimes \omega_H) := b(h_1 \mathrel{\scriptstyle \to} \omega_B) \otimes h_2h^{\prime} + b(h_1 \triangleright b^{\prime}) \otimes h_2\omega_H,
        \label{eq:leftBsmashHactiononGammasmash}
    \end{equation}
    and the right $B\#H$-action on $\Gamma_{\#}$ is
    \begin{equation}
        (\omega_B \otimes h^{\prime} + b^{\prime} \otimes \omega_H) \cdot (b \# h) := \omega_B (h^{\prime}_1 \triangleright b) \otimes h^{\prime}_2h + b^{\prime}\big((\omega_H)_{-1} \triangleright b\big) \otimes (\omega_H)_0h,
        \label{eq:rightBsmashHactiononGammasmash}
    \end{equation}
    for all $b,b^{\prime} \in B$, $h,h^{\prime} \in H$, $\omega_B \in \Gamma_B$, $\omega_H \in \Gamma_H$.
    \label{lem:GammasmahisaBsmashHbimodule}
\end{lemma}
\begin{proof}
    One can prove this by noting that the $\mathbb{K}$-module $\Gamma_{\#}$ in (\ref{eq:smashproductcalculusmaybe}) is a direct sum of tensor products of left $H$-modules, hence it carries a left $H$-action canonically induced from the $H$-actions on the $H$-modules $\Gamma_B$, $H$, $B$, $\Gamma_H$, that is 
    \begin{equation}
            h \cdot (\omega_B \otimes h^{\prime} + b^{\prime} \otimes \omega_H) = h_1 \mathrel{\scriptstyle \to} \omega_B \otimes h_2h^{\prime} + h_1 \triangleright b^{\prime} \otimes h_2 \omega_H.
            \label{eq:leftHactionofsmashproductcalculus}
     \end{equation}
     for all $h \in H$, $\omega_B \otimes h^{\prime} + b^{\prime} \otimes \omega_H \in \Gamma_{\#}$. Defining the left $B$-action on $\Gamma_{\#}$ as the $B$-action on the first factors in the tensor products $\Gamma_B \otimes H$ and $B \otimes \Gamma_H$ we obtain the left $B\#H$-action on $\Gamma_{\#}$ as in (\ref{eq:leftBsmashHactiononGammasmash}), and one can check that it is actually an action analogously to how one proves associativity of the multiplication in the smashed product algebra.\\
     Similarly, we define
     \begin{equation}
            (\omega_B \otimes h^{\prime} + b^{\prime} \otimes \omega_H) \cdot b := \omega_B (h^{\prime}_1 \triangleright b) \otimes h^{\prime}_2 + b^{\prime}((\omega_H)_{-1} \triangleright b) \otimes (\omega_H)_0
            \label{eq:rightBactiononGammasmash}
    \end{equation}
    and prove it is a right $B$-action on $\Gamma_{\#}$, and defining the right $H$-action on $\Gamma_{\#}$ as the right $H$-action on the second factors in the tensor products $\Gamma_B \otimes H$ and $B \otimes \Gamma_H$ we obtain the right $B\#H$-action on $\Gamma_{\#}$ as in (\ref{eq:rightBsmashHactiononGammasmash}).\\
    Finally, one proves commutativity of the left and right $B\#H$-actions on $\Gamma_B \otimes H$ and on $B \otimes \Gamma_H$.
\end{proof}
This shows that $\Gamma_{\#}$ is a $B\#H$-bimodule. We now prove the surjectivity condition and define a differential $\mathrm{d}_{\#}$ satisfying the Leibniz rule, so that $(\Gamma_{\#},\mathrm{d}_{\#})$ is a first order differential calculus on the smashed product algebra.
\begin{theorem}
    Let $H$ be a Hopf algebra and $B$ a left $H$-module algebra. Given an $H$-module first order differential calculus $(\Gamma_B, \mathrm{d}_B)$ on $B$ and a left covariant first order differential calculus $(\Gamma_H, \mathrm{d}_H)$ on $H$ there is a first order differential calculus $(\Gamma_{\#}, \mathrm{d}_{\#})$ on $B\#H$, where the $\mathbb{K}$-module $\Gamma_{\#} := \Gamma_B \otimes H \oplus B \otimes \Gamma_H$ is endowed with the $B\#H$-bimodule actions (\ref{eq:leftBsmashHactiononGammasmash}), (\ref{eq:rightBsmashHactiononGammasmash}) and the exterior derivative $\mathrm{d}_{\#} \colon B\#H \to \Gamma_{\#}$ is defined by
    \begin{equation}
        \mathrm{d}_{\#}(b \# h) := \mathrm{d}_B b \otimes h  + b \otimes \mathrm{d}_H h
        \label{eq:differentialsmash}
    \end{equation}
    for all $b \in B$ and $h \in H$.
    \label{thm:actualsmashproductcalculus}
\end{theorem}
\begin{proof}
    We show that $\mathrm{d}_{\#} \colon B\#H \to \Gamma_{\#}$ satisfies the Leibniz rule:
    \begin{equation}
        \begin{split}
            &\mathrm{d}_{\#}\big((b\#h)(b^{\prime}\#h^{\prime})\big) = \mathrm{d}_{\#}\big(b(h_1 \triangleright b^{\prime})\#h_2h^{\prime}\big)=\\
            &=\mathrm{d}_B\big(b(h_1\triangleright b^{\prime})\big)\otimes h_2h^{\prime} + b(h_1\triangleright b^{\prime}) \otimes \mathrm{d}_H(h_2h^{\prime})=\\
            &=\mathrm{d}_B(b)(h_1\triangleright b^{\prime})\otimes h_2h^{\prime} + b(h_1\triangleright b^{\prime}) \otimes \mathrm{d}_H(h_2)h^{\prime} +\\
            &+ b\mathrm{d}_B(h_1\triangleright b^{\prime})\otimes h_2h^{\prime} + b(h_1\triangleright b^{\prime}) \otimes h_2\mathrm{d}_H(h^{\prime}) = \\
            &=\mathrm{d}_B(b)(h_1\triangleright b^{\prime})\otimes h_2h^{\prime} + b(h_1\triangleright b^{\prime}) \otimes \mathrm{d}_H(h_2)h^{\prime} +\\
            &+ b\big(h_1\mathrel{\scriptstyle \to} \mathrm{d}_B(b^{\prime})\big)\otimes h_2h^{\prime} + b(h_1\triangleright b^{\prime}) \otimes h_2\mathrm{d}_H(h^{\prime})=\\
            &= \mathrm{d}_{\#}(b\#h) \cdot (b^{\prime}\#h^{\prime}) + (b\#h) \cdot \mathrm{d}_{\#}(b^{\prime}\#h^{\prime}),
        \end{split}
        \label{eq:Leibnizruleforsmashproductcalcdiff}
    \end{equation}
    for all $b, b^{\prime} \in B$ and $h, h^{\prime} \in H$ and where we used (\ref{eq:fordBtobeanHmodulemap}). The last equality is correct since by (\ref{eq:leftBsmashHactiononGammasmash}) we have
    \begin{equation}
        \begin{split}
            &(b\#h) \cdot \mathrm{d}_{\#}(b^{\prime}\#h^{\prime}) = (b\#h) \cdot \big(\mathrm{d}_B(b^{\prime})\otimes h^{\prime} + b^{\prime}\otimes \mathrm{d}_H(h^{\prime})\big) = \\
            &= b(h_1 \mathrel{\scriptstyle \to} \mathrm{d}_B(b^{\prime}))\otimes h_2h^{\prime} + b(h_1 \triangleright b^{\prime}) \otimes h_2\mathrm{d}_H(h^{\prime})
        \end{split}
        \label{eq:proofrightmoduleactionforLeibniz}
    \end{equation}
    and by (\ref{eq:rightBsmashHactiononGammasmash})
    \begin{equation}
        \begin{split}
            & \mathrm{d}_{\#}(b\#h) \cdot (b^{\prime}\#h^{\prime}) = \big(\mathrm{d}_B(b) \otimes h + b \otimes \mathrm{d}_H(h)\big) \cdot (b^{\prime}\# h^{\prime})=\\
            &=\mathrm{d}_B(b)(h_1 \triangleright b^{\prime}) \otimes h_2h^{\prime} + b\big( (\mathrm{d}_H(h))_{-1} \triangleright b^{\prime}\big) \otimes (\mathrm{d}_H(h))_0h^{\prime} = \\
            &=\mathrm{d}_B(b)(h_1 \triangleright b^{\prime}) \otimes h_2h^{\prime} + b(h_1 \triangleright b^{\prime}) \otimes \mathrm{d}_H(h_2)h^{\prime},
        \end{split}
        \label{eq:prooflrftmoduleactionforLeibniz}
    \end{equation}
    where in the last equality we have used the fact that $\Gamma_H$ is left $H$-covariant, i.e. that $\leftindex_{\Gamma_H}\Delta \circ \mathrm{d}_H(h) = (\mathrm{d}_H(h))_{-1} \otimes (\mathrm{d}_{H}(h))_{0} = (id \otimes \mathrm{d}_H) \circ \leftindex_H \Delta(h) = h_1 \otimes \mathrm{d}_H(h_2)$. \\
    We still have to prove that $\Gamma_{\#} = (B\#H) \cdot \mathrm{d}_{\#} (B\#H)$, i.e., the surjectivity condition. Let $b, b^{\prime} \in B$ and $h, h^{\prime} \in H$. Note that any element of $\Gamma_{\#}$ can be written as $\omega_B \otimes \bar{h} + \bar{b} \otimes \omega_H= b\mathrm{d}_B(b^{\prime})\otimes \bar{h} + \bar{b} \otimes h\mathrm{d}_H(h^{\prime})$, with $\bar{h} \in H$ and $\bar{b} \in B$. Then we have
    \begin{equation}
        \begin{split}
            &b \mathrm{d}_B b^{\prime} \otimes \bar{h} = b \mathrm{d}_B b^{\prime} \otimes \bar{h} + bb^{\prime} \otimes \mathrm{d}_H \bar{h} - bb^{\prime} \otimes \mathrm{d}_H \bar{h} = \\
            &=(b\#1) \cdot \mathrm{d}_{\#} (b^{\prime}\#\bar{h}) - (bb^{\prime}\#1) \cdot \mathrm{d}_{\#} (1\#\bar{h}),
        \end{split}
        \label{eq:proofofsurjectivityconditionforsmashcalc1}
    \end{equation}
    since
    \begin{equation}
        \begin{split}
            &(b\#1) \cdot \mathrm{d}_{\#} (b^{\prime}\#\bar{h}) - (bb^{\prime}\#1) \cdot \mathrm{d}_{\#} (1\#\bar{h}) = \\
            &=(b\#1) \cdot \big(\mathrm{d}_B(b^{\prime}) \otimes \bar{h} + b^{\prime} \otimes \mathrm{d}_H(\bar{h})\big) - (bb^{\prime}\#1) \cdot \big(1 \otimes \mathrm{d}_H(\bar{h})\big) =\\
            &=b\big(1 \mathrel{\scriptstyle \to} \mathrm{d}_B(b^{\prime})\big) \otimes \bar{h} + b(1 \triangleright b^{\prime}) \otimes \mathrm{d}_H(\bar{h}) - bb^{\prime}(1 \triangleright 1) \otimes \mathrm{d}_H(\bar{h}),
        \end{split}
        \label{eq:proofofsurjectivityconditionforsmashcalc11}
    \end{equation}
    and $b \otimes h\mathrm{d}_H h^{\prime} = (b\#h) \cdot \mathrm{d}_{\#} (1\#h^{\prime})$ since:
    \begin{equation}
        \begin{split}
            & (b\#h) \cdot \mathrm{d}_{\#} (1\#h^{\prime}) = (b \# h) \cdot (\mathrm{d}_B(1) \otimes h^{\prime} +  \otimes \mathrm{d}_H(h^{\prime}) = \\
            &= b\big(h_1 \mathrel{\scriptstyle \to} \mathrm{d}_B(1)\big) \otimes h_2h^{\prime} + b(h_1 \triangleright 1) \otimes h_2\mathrm{d}_H(h^{\prime})=\\
            &=\epsilon(h_1)b \otimes h_2\mathrm{d}_H(h^{\prime}) = b \otimes hd_H(h^{\prime})
        \end{split}
        \label{eq:proofofsurjectivityconditionforsmashcalc2}
    \end{equation}
    where we used $\mathrm{d}_B(1)=0$, $h \triangleright 1 = \epsilon(h)1$ and the counit property $(\epsilon \otimes id) \circ \Delta(h)=\epsilon(h_1) \otimes h_2 = h$. This shows that all elements in $\Gamma_{\#}$, which we know can be written in the form $\omega_B \otimes \bar{h} + \bar{b} \otimes \omega_H$, can also be written as a linear combination of elements $x \mathrm{d}_{\#}x^{\prime}$ for some $x, x^{\prime} \in B\#H$.
\end{proof}
The smash product construction of differential calculi is compatible with right $H$-coactions. In particular, when the calculus on $H$ is bicovariant, the smashed product calculus is right $H$-covariant.
\begin{corollary}
    Let $H$ be a Hopf algebra, $B$ a left $H$-module algebra, $(\Gamma_B, \mathrm{d}_B)$ an $H$-module first order differential calculus on $B$ and $(\Gamma_H, \mathrm{d}_H)$ a bicovariant first order differential calculus on $H$. The first order differential calculus $(\Gamma_{\#}, \mathrm{d}_{\#})$ of Theorem $\ref{thm:actualsmashproductcalculus}$ is then right $H$-covariant.
    \label{cor:smashproductcalculusisrightHcovariant}
\end{corollary}
\begin{proof}
    Define a right $H$-coaction on $\Gamma_{\#}$ via
    \begin{equation}
        \begin{split}
            \Delta_{\Gamma_{\#}} \colon &\Gamma_{\#} \to \Gamma_{\#} \otimes H,\\
            &\omega_B \otimes h + b \otimes \omega_H \mapsto \omega_B \otimes h_1 \otimes h_2 + b \otimes (\omega_H)_0 \otimes (\omega_H)_1
        \end{split}
        \label{eq:rightHcoactiononsmashcalctomakeitrightcovariant}
    \end{equation}
    We prove that the calculus is right $H$-covariant by showing right $H$-colinearity of the differential. For all $b\#h \in B\#H$,
    \begin{equation}
        \begin{split}
            &\Delta_{\Gamma_{\#}} \big(\mathrm{d}_{\#} (b\#h)\big) = \Delta_{\Gamma_{\#}} \big(\mathrm{d}_B(b) \otimes h + b \otimes \mathrm{d}_H(h)\big) =\\
            &= \mathrm{d}_B(b) \otimes h_1 \otimes h_2 + b \otimes \mathrm{d}_H(h_1) \otimes h_2 =\\
            &= \mathrm{d}_{\#} (b\#h_1) \otimes h-2 = (\mathrm{d}_{\#} \otimes id) \Delta_{\Gamma_{\#}} (b\#h).
        \end{split}
        \label{eq:rightHcovariantproodforsmashcalculus}
    \end{equation}
\end{proof}

\chapter{Quantum principal bundles} \label{quantumprincipalbundles}
In this chapter, we delve into the theory of quantum principal bundles, discussing the noncommutative geometry analogues of principal bundles, their associated bundles and reductions. References on this are \cite{majid_qriemanniangeo}, \cite{brz_majid_qgauge}, \cite{coqueraux_quantum}, \cite{del_donno_weber}  and \cite{latiniweber_projective}

In Section \ref{qpb}, we define quantum principal bundles, understood as Hopf-Galois extensions which are faithfully flat. The total space of (classical) principal fiber bundles is replaced by a comodule algebra $A$ and the base space manifold by the subalgbera of coinvariant elements of $A$ under its right $H$-coaction. The Hopf algebra $H$ takes the place of the structure group and the Hopf-Galois condition represents the algebraic analogue of the principality of the action of the Lie group on the total space. In \ref{baseformsandhorizontal}, a definition of base forms and horizontal forms in the non commutative framework is given, assuming a first order differential calculus on $A$. In Section \ref{associatedqvb} we provide two different definitions of associated quantum vector bundles to quantum principal bundles: one as the set of coinvariants $(A \otimes V)^{coH}$ and the other as the cotensor product $A\underset{H}{\Box}V$. Their equivalence is proven. Finally, in Section \ref{quantumreductions}, in light of the recent theory developed in \cite{chiara_qreductions}, a definition of quantum reductions of quantum principal bundles is given. 

\section{Quantum principal bundles} \label{qpb}
We will now introduce the notion of quantum principal bundle as a faithfully flat Hopf-Galois extension.
\begin{definition}
    Let $B$ be a ring and let $X, Y, Z$ be $B$-modules. A $B$-module $A$ si called \textsl{faithfully flat} if
    \begin{equation}
        0 \to X \xrightarrow[]{} Y \xrightarrow[]{} Z \to 0
        \label{eq:sequenceofringmodules}
    \end{equation}
    is a short exact sequence if and only if
    \begin{equation}
        0 \to X \otimes_B A \xrightarrow[]{} Y \otimes_B A \xrightarrow[]{} Z \otimes_B A \to 0
        \label{eq:sequenceiffshortexactthenfaith}
    \end{equation}
    is a short exact sequence.
    \label{def:faithfullyflatleftmodule}
\end{definition}
\begin{remark}
    We recall here that a sequence of objects (such as groups, modules, or vector spaces)
    \begin{equation}
        0 \to A \xrightarrow{f} B \xrightarrow{g} C \to 0
        \label{eq:sequence}
    \end{equation}
    is called a short exact sequence if:
    \begin{enumerate}
        \item the map $f \colon A \to B$ is injective, i.e. $\ker(f)=0$;
        \item $\text{im}(f)=\ker(g)$;
        \item the map $g \colon B \to C$ is surjective, i.e. $\text{im}(g)=C$.
    \end{enumerate}
    This is saying that the structure of $B$ encapsulates both $A$ and $C$.
    \label{rem:shortexactsequence}
\end{remark}
The following definition introduces the noncommutative analogue of a principal fiber bundle, extending the classical framework to noncommutative geometry.  This generalization preserves the fundamental structural features while extending them to a broader, algebraic framework. Let $H$ be a Hopf algebra and $A$ a right $H$-comodule algebra with right $H$-coaction $\Delta_A \colon A \to A \otimes H$.
\begin{definition}
    A \textsl{quantum principal bundle}, is a faithfully flat Hopf-Galois extension $B := A^{coH} \subseteq A$, where $A$ is a right $H$-comodule algebra and $B$ is the subalgebra of coinvariant elements in $A$ under the coaction of $H$. Alternatively, we call the quantum principal bundle $B \subseteq A$ a \textsl{principal} $H$\textsl{-comodule algebra}.
    \label{def:quantumprincipalbundledef}
\end{definition}
\begin{remark}
    Usually, in the classical case, the base space $M$ in a principal fiber bundle is defined to be a manifold from the start. Alternatively, one could require the right action of the group to be proper, i.e. for any $A$ and $B$ compact subsets of $P$ then $\{g \in G | Ag \cap B\ \neq \emptyset\}$ is also compact: this condition ensures that $M = P/G$ is a topological manifold (see \cite{sharpe_dg} Theorem 2.4) and makes the fibration locally trivial. \\
    In the non commutative setting, properness of the action is captured by faithful flatness: it guarantees that the total space algebra $A$ retains enough information about $B = A^{coH}$, i.e. that tensoring with $A$ does not lose information about $B$ and that the descent of data from $A$ to $B$ maintains the right structure. From faithful flatness, essentially, one construct the local triviality of the quantum principal bundle, much like properness of the action ensures local triviality and compatibility with transition maps in the classical case.
    \label{rem:onfaithfullyflatness}
\end{remark}
Classically, a principal bundle $\pi \colon P \to M$ is endowed with a free and transitive action of the structure group $G$ on the total space $P$. This yields the principality condition, i.e. that the canonical map $P \times G \to P \times_MP, (p,g) \mapsto (p,p g)$ into the fibered product $P \times_M P=\{ (p,p^{\prime}) \in P \times P | \pi(p)=\pi(p^{\prime}) \}$ is a diffeomorphism. \\
Let us give a diagrammatic comparison of the quantum vs classical principal bundles, where the classical principality condition $P \times G \cong P \times_M P$ is replaced by the algebraic "Hopf-Galois condition" $A \otimes_B A \cong A \otimes H$ (invertibility of the canonical map $\chi$ precisely means $A \otimes_B A$ and $A \otimes H$ are isomorphic as left $A$-modules and as right $H$-comodules):\\
\[
 \begin{tikzcd}
    P \arrow{d}{\pi} & P \times G \arrow{l}{R} \arrow{d}{\cong} \\
    M  &P \times_M P
\end{tikzcd}, \;\;\;
\begin{tikzcd}
    A \arrow{r}{\Delta_A} & A \otimes H  \\
    B \arrow{u}{\subseteq} &A \otimes_B A \arrow{u}{\cong}
\end{tikzcd}
\]
where the right action $R$ of the structure group $G$ on $P$ is replaced by the right $H$-comodule action $\Delta_A$ on the principal comodule algebra $A$. 

\subsection{Base forms and horizontal forms} \label{baseformsandhorizontal}

Let $B := A^{coH} \subseteq A$ a quantum principal bundle. Given a first order differential calculus $(\Gamma_A, \mathrm{d}_A)$ on $A$, by Proposition \ref{pro:pullbackcalculusandquotientcalculus}, we have a first order differential calculus on $B$ as the pullback calculus by the natural inclusion $B \to A$, which is trivially a morphism of algebras.
The next two definitions characterize the quantum analogues of differential forms $\Omega^1(M)$ on the base space and horizontal forms $\Omega^1_{hor}(P)$ on the total space of a principal fiber bundle $\pi \colon P \to M$. 
\begin{definition}
    Let $(\Gamma_A, d_A)$ be a first order differential calculus on a right $H$-comodule algebra $A$. Let $B=A^{coH}$ be the subalgebra of coinvariant elements. We call the pullback calculus $(\Gamma_B,d_B):=(Bd_A|_BB,d_A|_B)$ on $B$ the \textsl{first order differential calculus of base forms}.
    \label{def:baseformsFODC}
\end{definition}
\begin{definition}
    We call $\Gamma_A^{hor} = A\Gamma_B$ the $(A,B)$\textsl{-bimodule of horizontal forms}.
    \label{def:horizontalforms}
\end{definition}
When $(\Gamma_A,d_A)$ is right $H$-covariant we can give the following result (\cite{latiniweber_projective} Theorem 3.17).
\begin{theorem}
    Let $B \subseteq A$ be a quantum principal bundle and $(\Gamma_A, d_A)$ a right $H$-covariant first order differential calculus on $A$. The map 
    \begin{equation}
        \begin{split}
            &A \otimes_B \Gamma_B \to \Gamma_A,\\
            &a \otimes_B \omega \mapsto a\omega
        \end{split}
        \label{eq:naturalmap}
    \end{equation}
    gives the left $A$-module isomorphism
    \begin{equation}
        A \otimes_B \Gamma_B \cong A\Gamma_B = \Gamma_A^{hor}.
        \label{eq:thesplitting}
    \end{equation}
    Moreover
    \begin{equation}
        \Gamma_B = \Gamma_A^{hor} \cap \Gamma_A^{coH} \cong (\Gamma_A^{hor})^{coH}.
        \label{eq:baseforms}
    \end{equation}
    \label{thm:horizontalformsintersectionandsplitting}
\end{theorem}
The proof of the above theorem relies on the existence of an equivalence of categories between the category of left $B$-modules and that of right $H$-covariant left $A$-modules, where $B = A^{coH} \subseteq A$ a faithfully flat Hopf-Galois extension. This is established in \cite{schneider}, Theorem 1.

\section{Associated quantum vector bundles} \label{associatedqvb}

We will provide here two definitions of associate quantum vector bundle to the principal bundle $B=A^{coH} \subseteq A$: one, following \cite{coqueraux_quantum}, as the cotensor product $A \underset{H}{\Box} W$, with $W$ a left $H$-comodule algebra, which is more evidently linked to the classical definition of associated vector bundle; the other, as in \cite{brz_majid_qgauge} Appendix A, as the subalgebra of coinvariant elements $(A \otimes V)^{coH}$, with $V$ a right $H^{op}$-comodule algebra. We will prove the two definitions of quantum associated bundle are equivalent.\\

We recall that classically, given a principal bundle $\pi \colon P \to M$ with structure group $G$, the associated vector bundle $P \times_G V$ is given in terms of a vector space $V$ with a left $G$-action (for a suitable representation of the structure group). In particular, we have seen in \ref{associatedbundles} that the associated vector bundle $F$ with fiber $V$ is defined as the quotient of $P \times V$ by the equivalence relation,
\begin{equation}
    (p, gv) \sim (pg, v)
    \label{eq:recollassociatedvectorbundle}
\end{equation}
with $p \in P$, $v \in V$, $g \in G$.\\
Dually, we need a left $H$-comodule algebra $W$ playing the role of the algebra $C(V)$ of functions on $V$. Now, the associated vector bundle, which corresponds to the algebra of functions on the total space of the vector bundle $F$, is given by the following definition.
\begin{definition}
    Let $B=A^{coH} \subseteq A$ be a quantum principal bundle and let $W$ be a left $H$-comodule algebra with coaction $\leftindex_W\Delta \colon W \to H \otimes W$. We call the contensor product 
    \begin{equation}
    E = A \underset{H}{\Box} W := \big\{ a \otimes w \in A \otimes W \; \big| \; \Delta_A(a) \otimes w = a \otimes \leftindex_W\Delta(w) \big\}
    \label{eq:quantumassociatedvectorbundle}
    \end{equation}
    the \textsl{associated quantum vector bundle} to the principal bundle $A$.
    \label{def:quantumassociatedvectorbundleBox}
\end{definition}
This subspace of $A \otimes W$ is an algebra with multiplication given by that of the tensor product algebra, that is $(a \otimes w)(a^{\prime} \otimes w^{\prime}) = aa^{\prime} \otimes ww^{\prime}$. We give the following lemma.
\begin{lemma}
    The following two statements are true:
    \begin{enumerate}
        \item $E = A \underset{H}{\Box} W$ is a subalgebra of $A \otimes W$;
        \item $B$ is a subalgebra of $E = A \underset{H}{\Box} W$.
    \end{enumerate}
    \label{lem:toproveBoxquantumvectorbundleisok}
\end{lemma}
\begin{proof}
    To prove the first statement, we observe that $E$ is closed under tensor product multiplication. In fact, for $a \otimes w$ and $a^{\prime} \otimes w^{\prime}$ in $A \underset{H}{\Box} H$ we have: 
    \begin{equation}
        \begin{split}
            &\Delta_A(aa^{\prime}) \otimes ww^{\prime} = \Delta_A(a)\Delta_A(a^{\prime}) \otimes ww^{\prime} =  \big(\Delta_A(a) \otimes w\big)\big(\Delta_A(a^{\prime}) \otimes w^{\prime}\big) = \\
            &= \big( a \otimes \leftindex_W\Delta(w) \big)\big( a^{\prime} \otimes \leftindex_W\Delta(w^{\prime})\big) = aa^{\prime} \otimes \leftindex_W\Delta(w)\leftindex_W\Delta(w^{\prime}) = aa^{\prime} \otimes \leftindex_W\Delta(ww^{\prime}).
        \end{split}
        \label{eq:tensorproductistheproductofboxalgebraproof}
    \end{equation}
    The second statement is proven by considering a map $j_E \colon B \xhookrightarrow{} A \otimes W$ defined as $j_E(b) = b \otimes 1_W$. We have that, for any $b \in B$, $j_E(b) \in E$
    \begin{equation}
        \Delta_A(b) \otimes 1_W = b \otimes 1_H \otimes 1_W = b \otimes \leftindex_W\Delta(1_W). 
        \label{eq:vediamosevale}
    \end{equation}
\end{proof}
An $H^{op}$-comodule algebra $V$ is an algebra with a right $H^{op}$-coaction $\rho_V \colon V \to V \otimes H$, meaning that its algebra compatibiltity condition is given in terms of the  reversed Hopf algebra multiplication $\cdot^{op}$, defined as $h \cdot^{op} h^{\prime} = h^{\prime}h$ for all $h,h^{\prime} \in H$. This explicitly reads
\begin{equation}
    \rho_V(vv^{\prime}) = \rho_V(v)\rho(v^{\prime}) = v_0v^{\prime}_0 \otimes v_1 \cdot^{op} v^{\prime}_1 = v_0v^{\prime}_0 \otimes v^{\prime}_1v_1
    \label{eq:Hopcomodulealgebra}
\end{equation}
for all $v,v^{\prime} \in V$. \\
As mentioned, one can alternatively define an associated quantum vector bundle in terms of a right $H^{op}$-comodule algebra $V$ as below.
\begin{definition}
    Let $B=A^{coH} \subseteq A$ be a quantum principal bundle and let $V$ be a right $H^{op}$-comodule algebra with coaction $\rho_V \colon V \to V \otimes H$. The space $A \otimes V$ is naturally endowed with a right $H$-comodule structure $\Delta_{\mathcal{E}} \colon A \otimes V \to A \otimes V \otimes H$ given by
    \begin{equation}
        \begin{split}
            \Delta_{\mathcal{E}} \colon &A \otimes V \to A \otimes V \otimes H,\\
            &a \otimes v \mapsto \Delta_{\mathcal{E}}(a \otimes v) = a_0 \otimes v_0 \otimes a_1v_1,
        \end{split}
        \label{eq:coactionforPotimesV}
    \end{equation}
    for any $a \in A$ and $v \in V$.
    We say that the space (of coinvariant elements of $P \otimes V$ with respect to the right $H$-coaction $\Delta_{\mathcal{E}}$)
    \begin{equation}
        \mathcal{E} = (A \otimes V)^{coH} = \big\{ a \otimes v \in A \otimes V \; \big| \; \Delta_{\mathcal{E}}(a \otimes v) = a \otimes v \otimes 1_H \big\}
        \label{eq:associatedvectorquantumbundle}
    \end{equation}
    is an \textsl{associated quantum vector bundle} to $A$ over $B$. We denote it by $\mathcal{E}=\mathcal{E}(B, V, H)$.
    \label{def:associatedquantumvectorbundleHop}
\end{definition}
We now show the analogue of Lemma \ref{lem:toproveBoxquantumvectorbundleisok}.
\begin{lemma}
    The following two statements are true:
    \begin{enumerate}
        \item $\mathcal{E}=(A \otimes V)^{coH}$ is a subalgebra of $A \otimes V$;
        \item $B$ is a subalgebra of $\mathcal{E}=(A \otimes V)^{coH}$.
    \end{enumerate}
    \label{lem:toproveHopquantumvectorbundleisok}
\end{lemma}
\begin{proof}
        To prove the first statement let us take $a \otimes v$ and $a^{\prime} \otimes v^{\prime} \in \mathcal{E}=(A \otimes V)^{coH}$. We have
        \begin{equation}
            \begin{split}
                &\Delta_{\mathcal{E}}(aa^{\prime} \otimes vv^{\prime}) = a_0a^{\prime}_0 \otimes v_0v^{\prime}_0 \otimes a_1a^{\prime}_1v^{\prime}_1v_1 = \\
                &=(a_0 \otimes v_0 \otimes a_1)(a^{\prime}_0 \otimes v^{\prime}_0 \otimes a^{\prime}_1v^{\prime}_1)(1_A \otimes 1_V \otimes v_1)=\\
                &=(a_0 \otimes v_0 \otimes a_1)(a^{\prime} \otimes v^{\prime} \otimes 1_H)(1_A \otimes 1_V \otimes v_1)=\\
                &=(a_0 \otimes v_0 \otimes a_1)(1_A \otimes 1_V \otimes v_1)(a^{\prime} \otimes v^{\prime} \otimes 1_H)=\\
                &=(a_0 \otimes v_0 \otimes a_1v_1)(a^{\prime} \otimes v^{\prime} \otimes 1_H) = \\
                &=(a \otimes v \otimes 1_H)(a^{\prime} \otimes v^{\prime} \otimes 1_H) = (aa^{\prime} \otimes vv^{\prime} \otimes 1_H),
            \end{split}
            \label{eq:proveqvbisasubalgebraoftotalspacealgebra}
        \end{equation}
        where in the second equality we used the fact that $V$ is a right $H^{op}$-comodule algebra. This shows $aa^{\prime} \otimes vv^{\prime} = (a \otimes v)(a^{\prime} \otimes v^{\prime})$, i.e. the ($P \otimes V$) product of two elements in $\mathcal{E}$, is still in $\mathcal{E}$, making $\mathcal{E}$ closed under the multiplication of $P \otimes V$ and proving the first statement.\\
        For the second assertion, we observe that there is a map $j_{\mathcal{E}} \colon B \xhookrightarrow{} A \otimes V$ defined by $j_{\mathcal{E}}(b) = b \otimes 1_V$ for any $b \in M$ and $j_{\mathcal{E}}(b) \in \mathcal{E}$ since
        \begin{equation}
            \Delta_{\mathcal{E}}(j_{\mathcal{E}}(b)) = \Delta_{\mathcal{E}}(b \otimes 1_V) = b_0 \otimes 1_V \otimes 1_H1_H= b \otimes 1_V \otimes 1_H = j_{\mathcal{E}}(b) \otimes 1_H, 
            \label{eq:jEmakesitwork}
        \end{equation}
        since $\Delta_H(b) = b \otimes 1_H$ and $\rho_V(1_V) = 1_V \otimes 1_H$. This proves the Lemma.
\end{proof}
In the following proposition, we show that \ref{def:quantumassociatedvectorbundleBox} and \ref{def:associatedquantumvectorbundleHop} are equivalent definitions of associated quantum vector bundle. We assume that the Hopf algebra $H$ has invertible antipode.
\begin{proposition}
    Let $V$ be a left $H$-comodule algebra with coaction $\leftindex_V\Delta \colon V \to H \otimes V, v \mapsto v_{-1} \otimes v_0$. Then, $V$ is a right $H^{op}$-comodule algebra with coaction
    \begin{equation}
        \begin{split}
            \rho_V \colon &V \to V \otimes H,\\
            &v \mapsto v_0 \otimes S(v_{-1}),
        \end{split}
        \label{eq:rightHopcoactiononV}
    \end{equation}
    so that $E = A \underset{H}{\Box} V$ and $\mathcal{E} = (A \otimes V)^{coH}$ are equivalent.
    \label{pro:thetwodefinitionsofassqvbareequivalent}
\end{proposition}
\begin{proof}
    Let us first consider $a \otimes v \in A \underset{H}{\Box} V$ and show it is an invariant under $\Delta_{\mathcal{E}}$. We need first to manipulate the condition satisfied by any element of $A \underset{H}{\Box} V$, which is $\Delta_A(a) \otimes v = a \otimes \leftindex_V\Delta(v)$, i.e. $a_0 \otimes a_1 \otimes v = a \otimes v_{-1} \otimes v_0$. By taking $v = 1_V$ and flipping the last two legs we get
    \begin{equation}
        a_0 \otimes 1_V \otimes a_1 = a \otimes 1_V \otimes 1_H.
        \label{eq:firstrealtionneededforward}
    \end{equation}
    By considering $a = 1_A$ instead, and later flipping the last two legs and applying $( 1_A \otimes 1_V \otimes S)$ to both sides, we get
    \begin{equation}
        1_A \otimes v \otimes 1_H = 1_A \otimes v_0 \otimes S(v_{-1}).
        \label{eq:secondrealtionneededforward}
    \end{equation}
    Using these two relations, we can now write
    \begin{equation}
        \begin{split}
            &\Delta_{\mathcal{E}}(a \otimes v) = a_0 \otimes v_0 \otimes a_1S(v_{-1}) =\\
            &=(a_0 \otimes 1_V \otimes a_1)(1_A \otimes v_0 \otimes S(v_{-1}))=\\
            &=(a \otimes 1_V \otimes 1_H)(1_A \otimes v \otimes 1_H)=\\
            &=a \otimes v \otimes 1_H
        \end{split}
        \label{eq:provingforward}
    \end{equation}
    for any $a \otimes v \in A \underset{H}{\Box} V$, proving that $A \underset{H}{\Box} V \subseteq (A \otimes V)^{coH}$.\\
    To prove that any invariant $a \otimes v \in \mathcal{E} = (A \otimes V)^{coH}$ is in $A \underset{H}{\Box} V$, we have to manipulate the relation satisfied by any element in $(A \otimes V)^{coH}$, which is $a_0 \otimes v_0 \otimes a_1S(v_{-1}) = a \otimes v \otimes 1_H$. First we consider $a =1_A$, and by applying $(1_A \otimes 1_V \otimes S^{-1})$ to both sides and flipping the last two legs we get
    \begin{equation}
        1_A \otimes v_{-1} \otimes v_0 = 1_A \otimes 1_H \otimes v.
        \label{eq:firstrealtionneededbackwards}
    \end{equation}
    Taking instead $v = 1_V$, and flipping the last two legs we get
    \begin{equation}
        a_0 \otimes a_1 \otimes 1_V = a \otimes 1_H \otimes 1_V.
        \label{eq:secondrelationneededbackwards}
    \end{equation}
    Using these two relations we can write
    \begin{equation}
        \begin{split}
            &\Delta_A(a) \otimes v = a_0 \otimes a_1 \otimes v =\\
            &=(a_0 \otimes a_1 \otimes 1_V)(1_A \otimes 1_H \otimes v)=\\
            &=(a \otimes 1_H \otimes 1_V)(1_A \otimes v_{-1} \otimes v_0) =\\
            &= a \otimes v_{-1} \otimes v_0 = a \otimes \rho_V(v)
        \end{split}
        \label{eq:provingbackwards}
    \end{equation}
    for any $a \otimes v \in (A \otimes V)^{coH}$, proving that $(A \otimes V)^{coH} \subseteq A \underset{H}{\Box} V$ and thus showing the two definitions of associated quantum vector bundle are equivalent. 
\end{proof}
In the following, we will consider the quantum associated bundle given as the set coinvariants $\mathcal{E} = (A \otimes V)^{coH}$, as it the preferred choice in most cases.

\subsection{Cross sections} \label{crosssections}

Below, the quantum analogue of sections of the associated fiber bundle is defined.
\begin{definition}
    Let $\mathcal{E} = (A \otimes V)^{coH}$ be a quantum fiber bundle associated to a quantum principal bundle $A(B, H)$. A left $B$-module map
    \begin{equation}
        s \colon \mathcal{E} \to B
        \label{eq:crosssectionofaqvb}
    \end{equation} 
    such that 
    \begin{equation}
        s(1_{\mathcal{E}}) = 1_A
        \label{eq:crosssectionofaqvbdefiningproperty}
    \end{equation}
    is called a \textsl{cross section} of $\mathcal{E}(B, V, H)$.
    \label{def:crosssectionofaqvb}
\end{definition}
\begin{remark}
    The left $B$-module structure for $\mathcal{E}$ is naturally introduced: we can write $b \triangleright (a \otimes v) = ba \otimes v$ for any $b \in B$ and $a \otimes v \in \mathcal{E}$.
    \label{rem:theBmodulestructureonmathcalE}
\end{remark}
Recall that classical sections $\sigma$ satisfy the defining property $\pi \circ \sigma = id$. The following lemma shows that an analogous relation is given for quantum bundle cross sections. 
\begin{lemma}
    If $s \colon \mathcal{E} \to M$ is a cross section of a quantum fiber bundle $\mathcal{E}(M, V, H)$ then
    \begin{equation}
         s \circ j_{\mathcal{E}} = id_B
        \label{eq:nowitbecomesalemma}
    \end{equation}
    where $j_{\mathcal{E}} \colon B \to \mathcal{E}, b \mapsto b \otimes 1_V$ is a natural inclusion.
    \label{lem:nowitbecomesalemma}
\end{lemma}
\begin{proof}
    For any $b \in B$, we have $s \circ j_{\mathcal{E}}(b) = s(b \otimes 1_V) = b \triangleright s(1_{\mathcal{E}}) = b$.
\end{proof}
\begin{remark}
    In \cite{brz_majid_qgauge}, the above lemma is given as a definition of cross section, but, as noted in \cite{majid_remarksgauge}, the left $B$-module structure is needed, since classically we can multiply sections by functions on the base space.
    \label{rem:onwhyweneedittobealeftBmodulemapandwhy}
\end{remark}

\section{Quantum reductions} \label{quantumreductions}

Let us give some preliminary notions to understand the definition of quantum reduction. In particular we define a Hopf ideal $J$ and we show that the quotient algebra $H_0 := H/J$ is a Hopf algebra.
\begin{definition}
    A vector subspace $J \subseteq H$ of a Hopf algebra $H$ is called a \textsl{Hopf ideal} if it is a two-sided ideal of $H$ as an algebra, i.e. $\mu(H \otimes J) \subseteq J$ and $\mu(J \otimes H) \subseteq J$, and
    satisfies the following conditions
    \begin{enumerate}
        \item $\Delta(J) \subseteq J \otimes H + H \otimes J$;
        \item $\epsilon(J)=0$;
        \item $S(J) \subseteq J$.
    \end{enumerate}
    \label{def:Hopfideal}
\end{definition}
\begin{lemma}
    Let $H$ be a Hopf algebra and $J \subseteq H$ and Hopf ideal. Then $H_0 := H/J$ is a Hopf algebra.
    \label{lem:H0isaHopfalgebra}
\end{lemma}
\begin{proof}
    An element $[h] \in H_0 = H / J$ can be written as $[h] = h + J$.\\ Since $J$ is a two-sided ideal, $H_0$ is an algebra. In fact, the product in $H_0$ is given by $(h + J)(h^{\prime} + J) = hh^{\prime} + J \in H_0$, and associtativity and unitality are trivially satisftied. The unit in $H_0$ is $1_H + J$. Moreover, $H_0$ is also a coalgebra, with coproduct $\Delta_0(h + J) = \Delta(h) + (J \otimes H + H \otimes J)$. Coassociativity follows from that of $\Delta$. Counit is given by $\epsilon_0(h + J) = \epsilon(h)$ and counitality follows. Finally, the antipode for $H_0$ is given by $S_0(h + J) = S(h) + J$.
\end{proof}
Moreover, for any Hopf algebra $H$ and Hopf ideal $J$, the quotient map $\pi \colon H \to H_0 := H/J$ is a Hopf algebra morphism.
\begin{remark}
    In general, the quotient $H_0 := H/J$ is not isomorphic to a Hopf subalgebra of $H$. Given the natural projection $\pi \colon H \to H/J$, for $H_0$ to be a Hopf subalgebra of $H$, there must exist a Hopf algebra morphism $\sigma \colon H/J \to H$ such that $\pi \circ \sigma = id_{H_0}$, so that $H_0 = H/J$ is isomorphic to the Hopf subalgebra $\sigma(H/J)$, which is not always the case. \\
    \label{rem:onthefactthatH0isnotasubalgebraofHingeneral}
\end{remark}
As discussed in Remark \ref{rem:Hopfalgebraasacomodulealgebrawithcoproductascoaction}, any Hopf algebra $H$ can be seen as a right (or left) $H$-comodule algebra with coaction given by the coproduct $\Delta$. Thus, given the surjective projection $\pi \colon H \to H_0 := H/J$, $H$ can be seen as a right $H_0$-comodule algebra with right $H_0$-coaction $(id_H \otimes \pi) \circ \Delta$, or as a left $H_0$-comodule algebra with left $H_0$-coaction $(\pi \otimes id_H) \circ \Delta$.  \\
We can now define a quantum reduction, following \cite{chiara_qreductions}, Definition 3.5, and restricting it to the affine case, so that the sheaf approach is not strictly needed.
\begin{definition}
    Let $B := A^{coH} \subseteq A$ be a quantum principal bundle with right $H$-coaction $\Delta_A \colon A \to A \otimes H$, an $\pi \colon H \to H_0$ as surjective Hopf algebra morphism, with $H_0 := H/J$ for some Hopf ideal $J \subset H$.
    Let $A_0$ be a principal $H_0$-comodule algebra with $B_0 := A_0^{coH_0}$ the coinvariats under the right $H_0$-coaction $\Delta_{A_0} \colon A_0 \to A_0 \otimes H_0$.\\
    We say that $A_0$ is a \textsl{reduction} of $A$ if:
    \begin{enumerate}
        \item $B \cong B_0$ as algebras;
        \item there exists a surjective morphism of $H_0$-comodule algebras $\phi \colon A \to A_0$ such that $\phi(B)=B_0$, where $A$ carries the induced $H_0$-coaction $\Delta_A^0 := (id_A \otimes \pi) \circ \Delta_A \colon A \to A \otimes H_0$ by the Hopf algebra morphism $\pi \colon H \to H_0$.
    \end{enumerate}
    \label{def:quantumreduction}
\end{definition}
The morphism of right $H_0$-comodules $\phi \colon A \to A_0$ (note that the direction of the arrow is inverted with respect to the principal bundle $\iota$ morphism defining a principal bundle reduction as in Section \ref{reductions}) is a $\mathbb{K}$-linear map such that $\Delta_{A_0} \circ \phi = (\phi \otimes id) \circ \Delta_A^0$, where $\Delta_A^0 := (id_A \otimes \pi) \circ \Delta_A$ is the induced right $H_0$-coaction on $A$, given in terms of the natural projection $\pi \colon H \to H_0 :=H/J$. \\
$A$ is a right $H_0$-comodule algebra by the coaction $\Delta_A^0$. In fact, since $\pi$ is a Hopf algbera morphism, we have
\begin{equation}
    \begin{split}
        &\Delta_A^0(aa^{\prime}) = (id_A \otimes \pi) \circ \Delta_A(aa^{\prime}) =(id_A \otimes \pi)\Delta_A(a)\Delta_A(a^{\prime}) = \\
        &= (id_A \otimes \pi)(a_0a^{\prime}_0 \otimes a_1a^{\prime}_1) = a_0a^{\prime}_0 \otimes \pi(a_1a^{\prime}_1) = a_0a^{\prime}_0 \otimes \pi(a_1)\pi(a^{\prime}_1) = \\
        &= \big(a_0 \otimes \pi(a_1)\big)\big(a_1 \otimes \pi(a^{\prime}_1)\big) = \Delta_A^0(a)\Delta_A^0(a^{\prime})
    \end{split}
    \label{eq:ragionandosuDeltaA0suA1}
\end{equation}
for all $a,a^{\prime} \in A$, and
\begin{equation}
    \Delta_A^0(1_A) = (id_A \otimes \pi) \circ \Delta_A(1_A) = 1_A \otimes \pi(1_H) = 1_A \otimes 1_{H_0}.
    \label{eq:ragionandosuDeltaA0suA2}
\end{equation}
Moreover, coinvariants $B := A^{coH}$ under $\Delta_A$ are also coinvariants under $\Delta_A^0$. In fact, for all $b \in B$:
\begin{equation}
        \Delta_A^0(b) = (id_A \otimes \pi) \circ \Delta_A(b) = b \otimes \pi(1_H) = b \otimes 1_{H_0},
    \label{eq:ragionandosuDeltaA0suA3}
\end{equation}
so that $B = A^{coH} \subseteq A^{coH_0}$.
\begin{remark}
    As observed in \cite{chiara_qreductions}, in general, the requirement for the map $\phi \colon A \to A_0$ to be an algebra map is too restrictive. When $\phi$ is a morphism of algebra, on has $A_0 \cong A/I$ by some ideal $I$ of $A$ and, in the degenerate case of $J$ and $I$ trivial, a reduction corresponds to a bundle automorphism (i.e., a gauge transformation): in this case, the algebra morphism condition should be relaxed \cite{aschieri_chiara}, \cite{brzezinski_translation}. For this reason, in \cite{chiara_qreductions}, a quantum reduction is defined in terms of a morphism $\phi \colon A \to A_0$ of right $H_0$-comodules, without assuming it to be an algebra map. When $\phi$ is an algebra morphism too, $A_0$ is called an algebraic reduction.\\
    In this thesis, for simplicity, we require $\phi$ to be an algebra map in order to construct the quotient calculus on $A_0$, which is used to prove Theorem \ref{thm:aquantumGstructureisaquantumframeresolution}. We believe that, in exchange for additional assumptions on the calculus, one could take $\phi$ to be an $H_0$-comodule morphism only. 
    \label{rem:onalgebraicreductions}
\end{remark}

\chapter{The quantum frame bundle} \label{quantumframebundle}
In this chapter, a quantum-geometric generalization of the important example of the frame bundle is developed, with the aim of providing a coherent definition of quantum $G$-structures. Our main references for this final chapter are \cite{cap_slovak}, \cite{brzezinski_translation} and \cite{majid_qbraided}.

We begin in Section \ref{frameresolutions} with a review of the classical theory of frame bundle, focusing on their defining feature: the existence of the canonical form. The key argument in this discussion is based on the bijective correspondence between horizontal and $G$-equivariant forms on the total space $P$ with values in a suitable vector space $V$, and differential forms on the base space $M$ with values in the associated bundle $P \times_G V$. This correspondence allows us to characterize frame bundles through the concept of a "frame resolution", which manifests as an isomorphism between the tangent bundle $TM$ and the associated vector bundle $P \times_G V$. In Section \ref{quantumframeresolutions}, we extend these ideas to the quantum setting by introducing the non commutative counterpart of this correspondence. This is done, following \cite{brzezinski_translation}, by employing the translation map on a quantum principal bundle, whose definition and properties are dealt with in \ref{thetranslationmap}. We thus arrive at the notion of a quantum frame resolution, which serves as the closest analogue to a (classical) frame bundle in noncommutative geomtery. In \ref{exampleofqfr} we show an example of quantum frame resolution by viewing the smashed product algebra as a quantum homogeneous principal bundle, in analogy with the classical affine case discussed in \ref{cartangeometries}. Finally, in Section \ref{quantumGstructure}, we provide a definition of quantum $G$-structures, using the framework of quantum reductions established in \ref{quantumreductions}. Mirroring the classical theory, we define a quantum $G$-structure as a quantum reduction of a quantum frame resolution. The main result in this work asserts that a quantum $G$-structure is itself a quantum frame resolution, under the assumption of a covariant first order differential calculus on the starting quantum frame resolution, and consequently a quotient calculus on the quantum $G$-structure and pullback calculi on the respective base spaces. To conclude, in \ref{exampleofqGstruc}, an example is provided.

\section{Frame resolutions} \label{frameresolutions}

We start by recalling that for any principal bundle $\pi \colon P \to M$ with structure group $G$, there is an associated bundle $E = P \times_G F$ with standard fiber $F$, consisting of equivalence classes in $P \times F$ under the relation $(p g, v) \sim (p, g \circ v)$, or, equivalently, $(p,v) \sim (p g, g^{-1} \circ v)$, for all $p \in P$, $v \in F$ and $g \in G$, given a left action of the group $G$ (possibly under a suitable representation, which is here left understood) on $F$. Let $\pi_F \colon E \to M$ the projection map of the associated bundle, which satisfies $\pi_F(p,v) = \pi(p)$. The following important propositions are given \cite{cap_slovak}.
\begin{proposition}
    Let $\pi \colon P \to M$ be a principal $G$-bundle and $F$ a smooth manifold endowed with a left $G$-action. \\
    There is a natural bijective correspondence
    \begin{equation}
        \Gamma(P \times_G F) \iff C^{\infty}(P, F)^G,
        \label{eq:sectionssmooth}
    \end{equation}
    where $\Gamma(P \times_G F)$ is the set of all smooth sections $\sigma$ of the associated bundle $P \times_G F$, and $C^{\infty}(P, F)^G$ is the set of all smooth functions $\phi \colon P \to F$ which are $G$-equivariant, i.e. $\phi(pg)=g^{-1}\phi(p)$. \\
    Explicitly, this correspondence is given as
    \begin{equation}
        \sigma(\pi(p))= [\![ p, \phi(p)]\!].
        \label{eq:explcorr}
    \end{equation}
    \label{pro:corrspondencebetweensectionsandequivariantmapsCap}
\end{proposition}
\begin{proof}
    Given $\phi \colon P \to F$ a smooth $G$-equivariant function, i.e. satisfying $\phi(pg)=g^{-1}\phi(p)$, points belonging to the total space of the associated fiber bundle $P \times_G F$ can be written as
    \begin{equation}
        [\![ p, x ]\!] = [\![ p, \phi(p)]\!] = [\![ pg, \phi(pg)]\!] = [\![ pg, g^{-1}\phi(p) ]\!] = [\![ pg, g^{-1}x]\!].
        \label{eq:xdxdxd}
    \end{equation}
    This implies, for each $\phi$ as above, we can define a section $\sigma \in \Gamma(P \times_G F)$ as
    \begin{equation}
        \begin{split}
            \sigma \colon &M \to P \times_G F\\
            &m \mapsto \sigma(m) = [\![ s_{\alpha}(m), \phi(s_{\alpha}(m)) ]\!],
        \end{split}
        \label{eq:diodio}
    \end{equation}
    for a smooth local section $s_{\alpha}$ of $P$ defined on a neighborhood $U_{\alpha} \ni m$, which implies smoothness of $\sigma$. Clearly this definition depends only on $m = \pi(p)$, since $\sigma(m)=\sigma(\pi(p))=\sigma(\pi(pg))$.\\
    Conversely, let $\sigma \in \Gamma(P \times_G F)$ be a smooth section of the associated fiber bundle. Any element of the fiber over $\pi(p)=\pi_F(p, x)$ can be uniquely written as an equivalence class $[\![ p, y ]\!] = [\![ pg, g^{-1} y ]\!]$. We can use
    \begin{equation}
        \sigma(\pi(p))=[\![ p, \phi(p)]\!]
        \label{eq:lolol1}
    \end{equation}
    to define $\phi \colon P \to F$. Since $\sigma(\pi(p))= \sigma(\pi(pg))$, then we have 
    \begin{equation}
        \sigma(\pi(p))=[\![ p, \phi(p)]\!] =[\![ pg, \phi(pg)]\!],
        \label{eq:lolol2}
    \end{equation}
    which implies $G$-equivariance $\phi(pg)=g^{-1}\phi(p)$ since an element of the associated bundle is such that $[\![ p, \phi(p)]\!] = [\![ pg, g^{-1}\phi(p)]\!]$. Smoothness of $\phi$ comes from defining the smooth section $\sigma$ in terms of smooth local sections $s_{\alpha}$ of the principal bundle as $\sigma(\pi(s_{\alpha}(m))) = [\![ s_{\alpha}(m), \phi(s_{\alpha}(m))]\!]$.
\end{proof}
The above proposition extends to associated vector bundles $P \times_G V$ with standard fibre a vector space $V$.
\begin{proposition}
    Let $\pi \colon P \to M$ be a principal fiber bundle with structure group $G$ and let $\rho \colon G \to GL(V)$ be a representation of $G$ acting on a vector space $V$.\\
    Then, for each $k \in \mathbb{N}$, there is a bijective correspondence
    \begin{equation}
        \Omega^k( M; P \times_G V) \iff \Omega^k_{hor}(P;V)^G,
        \label{eq:kformshorequi}
    \end{equation}
    where $\Omega^k(M; P \times_G V)$ is the set of $k$-forms with values in the associated bundle and $\Omega^k_{hor}(P;V)^G$ is the set of horizontal and equivariant forms (alternatively called tensorial forms), i.e. differential forms $\theta$ which annihilate vertical vector fields and satisfy
    \begin{equation}
        R_g^{\ast} \theta_{pg} = \rho(g^{-1}) \circ \theta_p \text{ for every } g \in G. 
        \label{eq:eccoci}
    \end{equation}
    \label{pro:correspondencetensorialformsandassfbvaluedformsCap}
\end{proposition}
\begin{proof}
    Given $\alpha \in \Omega^k(M; P \times_G V)$, let $m = \pi(p)$ with $p \in P$ and let $\xi_1,\dots,\xi_k \in T_pP$ be $k$ tangent vectors. There is a unique element $\theta_p (\xi_1, \dots ,\xi_k) \in V$ such that
    \begin{equation}
        \alpha_{\pi(p)}\big((\pi_{\ast})_p \cdot \xi_1, \dots , (\pi_{\ast})_p \cdot \xi_k\big) = [\![ p, \theta_p(\xi_1, \dots , \xi_k) ]\!].
        \label{eq:cheridere}
    \end{equation}
    with $(\pi_{\ast})_p \colon T_pP \to T_{\pi(p)}M$. Since $\alpha_{\pi(p)}$ is a $k$-form, it is an alternating (i.e. antisymmetric in its arguments) linear map, which implies also $\theta_p \colon (T_pP)^k \to V$ is. In particular, if one of the entries in the argument of $\alpha_{\pi(p)}$ vanishes, then also the result does. Thus, if one of the $\xi_i \in T_pP$ is a vertical tangent vector, which means that its push forward by the projection map vanishes (i.e. $(\pi_{\ast})_p \xi_i=0$), then $\theta_p (\xi_1,\dots,\xi_k)$ vanishes: $\theta_p$ is horizontal for any $p \in P$.\\
    Let $(R_g)_{\ast} \colon T_pP \to T_{pg}P$. To prove equivariance, we notice that we can write
    \begin{equation}
        \begin{split}
            &[\![ pg, R_g^{\ast} \theta_{pg} (\xi_1,\dots,\xi_k) ) ]\!] = [\![ pg, \theta_{pg} \big((R_g)_{\ast}\xi_1, \dots ,(R_g)_{\ast}\xi_k\big) ]\!] =\\
            &= \alpha_{\pi(pg)}\Big((\pi_{\ast})_p \big((R_g)_{\ast} \xi_1\big), \dots , (\pi_{\ast})_p \big( (R_g)_{\ast} \xi_k\big)\Big) = \\
            &=\alpha_{\pi(p)}\big( (\pi_{\ast})_p \xi_1, \dots , (\pi_{\ast})_p \xi_k\big) = [\![ p, \theta_p(\xi_1, \dots , \xi_k) ]\!],
        \end{split}
        \label{eq:considerationsss}
    \end{equation}
    where we used that $\pi_{\ast} \circ (R_g)_{\ast} = (\pi \circ R_g)_{\ast} = \pi_{\ast}$. Now, since any element $[\![ p, \phi(p)(\xi_1, \dots , \xi_k) ]\!] \in P \times_G V$ in the associated bundle is, by definition, such that
    \begin{equation}
        [\![ p,  \theta_p (\xi_1, \dots , \xi_k) ]\!] = [\![ pg, \rho(g^{-1}) \circ \theta_p((\xi_1, \dots , \xi_k) )]\!],
        \label{eq:okstepdio}
    \end{equation}
    we have proven equivariance and thus the forward implication.\\
    Conversely, suppose we have given a horizontal, equivariant form $\theta \in \Omega^k_{hor}(P;V)^G$. Then, for each $m \in M$, we can choose $p \in P$ such that $\pi(p)=m$, and any tangent vector at $m$ can be written as $(\pi_{\ast})_p \xi$ with $\xi \in T_pP$. Fixing $p$, we can use equation (\ref{eq:cheridere}) to define $\alpha_{\pi(p)}$. This does not depend on the choice of the lifts of the tangent vectors since $\theta$ is horizontal. By the $G$-equivariance of $\theta$, this does not depend on the choice of $p$. Finally, smoothness of $\theta$ implies smoothness of $\alpha$. 
\end{proof}

As in \cite{hajac_strong}, Appendix B, we show below a very important theorem which captures the essential properties of a frame bundle. 
Recall that an isomorphism of two principal fiber bundles $P(M, G)$ and $P^{\prime}(M, G)$, with the same structure group $G$, is a diffeomorphism $f \colon P \to P^{\prime}$ satisfying $\pi^{\prime} \circ f = \pi$ and $f \circ R_g = R_g^{\prime} \circ f$ for any $g \in G$.
\begin{theorem}
    Let $M$ be a smooth manifold with $\text{dim}(M) = n$. A principal $GL(n,\mathbb{R})$-bundle $\pi \colon P \to M$ is isomorphic to the frame bundle $FM$ if and only if there exists a smooth horizontal and equivariant $\mathbb{R}^n$-valued $1$-form $\theta \in \Omega^1_{hor}(P; \mathbb{R}^n)^{GL(n;\mathbb{R})}$.
    \label{thm:equivalencebetweenframeandprincipalwithsoldering}
\end{theorem}
\begin{proof}
    Let us first consider the direct implication. If $P(M,GL(n,\mathbb{R}))$ is isomorphic (as a principal bundle) to the frame bundle $FM$ through the principal bundle isomorphism $f \colon P \to FM$, the (horizontal and equivariant) canonical one form $\theta \in \Omega^1_{hor}(FM;\mathbb{R}^n)^{GL(n,\mathbb{R})}$ on $FM$, which was defined as $\theta(X) = u^{-1}(\tilde{\pi}_{\ast}X)$, with $X \in \mathscr{X}(FM)$, $\tilde{\pi}$ the projection map for the frame bundle and $FM \ni u \colon \mathbb{R}^n \to T_mM$ a (linear) isomorphism (with $u$ in the fiber over $m$), fills the requirements for the $1$-form in the statement of the theorem: it is in fact enough to consider $f^{\ast} \theta \in \Omega^1_{hor}(P;\mathbb{R}^n)^{GL(n,\mathbb{R})}$.\\
    The converse implication is more subtle. Let's assume we have a principal $GL(n,\mathbb{R})$-bundle $\pi \colon P \to M$, equipped with a $\theta \in \Omega^1_{hor}(P;\mathbb{R}^n)^{GL(n,\mathbb{R})}$, i.e. satifying $ker(\theta)=ker(\pi_{\ast})$ and $R_g^{\ast} \theta = g^{-1} \circ \theta$, or equivalently $\theta_{pg} = g^{-1} \circ \theta_p \circ (R_{g^{-1}})_{\ast}$, for any $g \in GL(n,\mathbb{R})$, where it is left implicit that we're considering the fundamental representation of the group acting on $\mathbb{R}^n$. We know a that frame, i.e. an element $u \in FM$, over each orbit $m \in M$, is an isomorphism $u \colon \mathbb{R}^n \xrightarrow{\sim} T_mM$, which means that any $u$ in the fiber over $m$ is a basis for $T_mM$: an assignement of a particular tangent vector to any $n$-tuplet of components. Analogously, a coframe (i.e. an element of $F^{\ast}M$) is, over each point in the base space, an isomorphism $T_mM \xrightarrow{\sim} \mathbb{R}^n$, acting on a tangent vector at $m$ to give its components in that basis.\\    
    Having said this, in order to prove the converse implication, we have to define a principal bundle isomorphism between $P$ and $FM$. Consider the following map:
    \begin{equation}
        \begin{split}
            F \colon &P \xrightarrow{\tilde{F}} F^{\ast}M \to FM,\\
            &p \mapsto (\sigma^{(p)})^{\ast} \theta_p \mapsto \big((\sigma^{(p)})^{\ast} \theta_p\big)^{-1},
        \end{split}
        \label{eq:hajacappendixthm}
    \end{equation}
    where $\sigma^{(p)}$ is any section for which $\sigma^{(p)}(\pi(p)) = p$. Let us check that this map is well defined. Since $\pi_{\ast}$ defines an isomorphism between the horizontal subspace $\mathcal{H}_pP \cong T_pP/\ker\pi_{\ast}$ and $T_{\pi(p)}M$ at each point $p \in P$, we have that $\theta_p$ defines an isomorphism between $\mathcal{H}_pP$ and $\mathbb{R}^n$. In fact, $\theta_p$ is injective on the horizontal subspace (otherwise it would have a non trivial kernel on it and $\ker\theta \neq \ker\pi_{\ast}$ is a contradiction) and $\text{dim}(\mathcal{H}_pP) = \text{dim}(T_{\pi(p)}M) = \text{dim}(\mathbb{R}^n)$. Thus, the pullback $(\sigma^{(p)})^{\ast} \theta_p = \theta_p \circ \sigma^{(p)}_{\ast}$ defines an isomorphism between $T_{\pi(p)}M$ and $\mathbb{R}^n$, i.e. a coframe in $F^{\ast}M$, independently of the chosen $\sigma^{(p)}$. This is because we assumed $\sigma^{(p)}(\pi(p)) = p$, hence, for any such section, $(\sigma^{(p)}_{\ast})_{\pi(p)} \circ (\pi_{\ast}|_{\mathcal{H}_pP})_p = id$, and $\sigma^{(p)}_{\ast} = (\pi_{\ast}|_{\mathcal{H}_pP})^{-1}$ is an isomorphism from $T_{\pi(p)}M$ to $\mathcal{H}_pP$. This proves that $\tilde{F}$ is well (and uniquely) defined, and consequently also $F$ is, since clearly $(\sigma^{\ast} \theta_p)^{-1}$ is a frame.\\
    Let $\tilde{\pi} \colon FM \to M$ be the projection map for the frame bundle. Then we immediately have that $\tilde{\pi} \circ F = \pi$, because $\big((\sigma^{(p)})^{\ast} \theta_p\big)^{-1}$ is a basis for $T_{\pi(p)}M$ for any $p$. Moreover, let $R^{\prime} \colon FM \times GL(n,\mathbb{R}) \to FM$ be the right $GL(n,\mathbb{R})$-action for the frame bundle and $R \colon P \times GL(n,\mathbb{R}) \to P$ the one for $P$. We have that $R^{\prime}_g \circ F = F \circ R_g$ for any $g \in GL(n,\mathbb{R})$, since by $GL(n,\mathbb{R})$-equivariance of $\theta$:
    \begin{equation}
        \begin{split}
            &F \circ R_g(p) = F(pg) = \big(\theta_{pg} \circ (\sigma^{(pg)}_{\ast})_{\pi(pg)}\big) = \big(g^{-1} \circ \theta_p \circ ((R_{g^{-1}})_{\ast})_{pg} \circ  (\sigma^{(pg)}_{\ast})_{\pi(p)}\big)^{-1} =\\
            &= \big(\theta_p \circ (R_{g^{-1}} \circ \sigma^{(pg)})_{\ast}\big)^{-1} \circ g = (\theta_p \circ \sigma^{(p)}_{\ast})^{-1} \circ g = R^{\prime}_g \circ F(p).
        \end{split}
        \label{eq:toproveeverythinggoeswell}
    \end{equation}
    This shows that $F$ satisfies both conditions for it to be a principal bundle isomorphism, i.e. $\tilde{\pi} \circ F = \pi$ and $R^{\prime}_g \circ F = F \circ R_g$, and clearly $F$ is a bijection, since $\tilde{F}$ uniquely identifies one coframe $g^{-1} \circ (\sigma^{(p)})^{\ast}\theta_p$, "shifted" by $g^{-1}$, for each point $pg$ in the fiber over $\pi(p)$.\\
    We are left with the task of proving $F$ is smooth and has a smooth inverse. There always exists a non empty overlap $U$, open neighborhood of the point $\pi(p) \in M$, over which we can pick smooth sections $s \colon U \to P$ and $\tilde{s} \colon U \to FM$. Let $\Psi_s \colon \pi^{-1}(U) \to U \times GL(n,\mathbb{R})$ and $\Psi_{\tilde{s}} \colon \tilde{\pi}^{-1}(U) \to U \times GL(n,\mathbb{R})$ be the local trivializations associated to these local sections. Since $p$ is arbitrary, to prove that $F$ is smooth, it is enough to show that the map 
    \begin{equation}
        \Psi_{\tilde{s}} \circ F \circ \Psi_s^{-1} \colon U \times GL(n, \mathbb{R}) \to U \times GL(n,\mathbb{R})
        \label{eq:considertoproveFsmooth1}
    \end{equation}
    is smooth. Let $m \in U$. Consider the following chain of equalities:
    \begin{equation}
        \begin{split}
            &F\big(s(m)g\big) = \tilde{s}\Big( \tilde{\pi}\big( F(s(m)g) \big) \Big) \circ g_{\tilde{s}}\Big( F\big(s(m)g\big)  \Big); \\
            &\big(\theta_{s(m)g} \circ \sigma^{(s(m)g)}_{\ast}\big)^{-1} = \tilde{s}\Big( \pi\big(s(m)g \big)\Big) \circ g_{\tilde{s}}\Big( \big(\theta_{s(m)g} \circ \sigma^{(s(m)g)}_{\ast}\big)^{-1}  \Big);\\
            &\big(\theta_{s(m)} \circ \sigma^{(s(m))}_{\ast}\big)^{-1} \circ g = \tilde{s}(m) \circ g_{\tilde{s}}\Big( \big(\theta_{s(m)} \circ \sigma^{(s(m))}_{\ast}\big)^{-1} \Big) \circ g ;\\
            &\big(\theta_{s(m)} \circ \sigma^{(s(m))}_{\ast}\big)^{-1}= \tilde{s}(m) \circ g_{\tilde{s}}\Big( \big(\theta_{s(m)} \circ \sigma^{(s(m))}_{\ast}\big)^{-1} \Big); \\
            &g_{\tilde{s}}\Big( \big(\theta_{s(m)} \circ \sigma^{(s(m))}_{\ast}\big)^{-1} \Big)^{-1} = \big(\theta_{s(m)} \circ \sigma^{(s(m))}_{\ast}\big) \circ \tilde{s}(m);\\
            &g_{\tilde{s}}\Big( \big(\theta_{s(m)} \circ \sigma^{(s(m))}_{\ast}\big)^{-1} \Big) = \Big( \theta_{s(m)} \circ \sigma^{(s(m))}_{\ast} \circ \tilde{s}(m) \Big)^{-1}.
        \end{split}
        \label{eq:considertoproveFsmooth2}
    \end{equation}
    Therefore, we have
    \begin{equation}
        \begin{split}
            &\big(\Psi_{\tilde{s}} \circ F \circ \Psi_s^{-1}\big)(m, g) = \big(\Psi_{\tilde{s}} \circ F\big)(s(m)g) = \Psi_{\tilde{s}}\Big(F\big(s(m)g\big)\Big) = \\
            &= \Big( \tilde{\pi}\Big( F\big(s(m)g\big)\Big) , g_{\tilde{s}}\Big( F\big(s(m)g\big)  \Big) \Big) = \Big( \pi\big(s(m)g\big) , g_{\tilde{s}}\Big( \big(\theta_{s(m)g} \circ \sigma^{(s(m)g}_{\ast}\big)^{-1}  \Big) \Big) =\\
            &= \Big( m , g_{\tilde{s}}\Big( \big(\theta_{s(m)} \circ \sigma^{(s(m))}_{\ast}\big)^{-1} \circ g   \Big) \Big) = \Big( m , g_{\tilde{s}}\Big( \big(\theta_{s(m)} \circ \sigma^{(s(m))}_{\ast}\big)^{-1} \Big)  \circ g   \Big) = \\
            &= \Big( m , \Big(\theta_{s(m)} \circ \sigma^{(s(m))}_{\ast} \circ \tilde{s}(m) \Big)^{-1}  \circ g   \Big)
        \end{split}
        \label{eq:considertoproveFsmooth3}
    \end{equation}
    We can now see that the smoothness of $F$ follows from the smoothness of the map
    \begin{equation}
        \begin{split}
            f \colon &U \to GL(n, \mathbb{R}),\\
            &m \mapsto \Big(\theta_{s(m)} \circ \sigma^{(s(m))}_{\ast} \circ \tilde{s}(m) \Big)^{-1}  \circ g,
        \end{split}
        \label{eq:considertoproveFsmooth4}
    \end{equation}
    which is smooth since it's a composition of smooth functions. Furthermore, $F^{-1}$ is also smooth because $\big(\Psi_s \circ F^{-1} \circ \Psi_{\tilde{s}}^{-1}\big)(m, g) = \big(m , \theta_{s(m)} \circ \sigma^{(s(m))}_{\ast} \circ \tilde{s}(m) \circ g)$. Hence  $F$ is a bundle isomorphism and $P( M, GL(n, \mathbb{R}))$ and $FM$ are isomorphic, as claimed.
\end{proof}
This theorem acts as an axiomatic definition of the frame bundle, which will be translated into the language of quantum principal bundles.\\
The important thing to notice is that in the case of the frame bundle $\tilde{\pi} \colon FM \to M$, by the bijective correspondence in Proposition \ref{pro:correspondencetensorialformsandassfbvaluedformsCap} for $k = 1$, the canonical form $\tilde{\theta} \in \Omega^1_{hor}(FM;\mathbb{R}^n)^{GL(n,\mathbb{R})}$ induces a vector bundle isomorphism $TM \cong FM \times_{GL(n,\mathbb{R})} \mathbb{R}^n$. As discussed, the canonical form is an assignment of an $n$-tuple $v = \tilde{\theta}_u(\xi_u) \in \mathbb{R}^n$ of components to any vector $\pi_{\ast}\xi_u \in T_{\tilde{\pi}(u)}M$. Thus, we define an inverse for the map in Equation (\ref{eq:cheridere}) by assigning to each $[\![ u, v ]\!] \in FM \times_{GL(n,\mathbb{R})} \mathbb{R}^n$ the tangent vector $u(v) \in T_{\tilde{\pi}(u)}M$, i.e., the tangent vector with coordinates $v$ in the frame $u$. This mapping shows that, by construction, one can naturally identify the tangent bundle $TM$ with the associated bundle $FM \times_{GL(n,\mathbb{R})} \mathbb{R}^n$ for the standard representation of $GL(n,\mathbb{R})$.\\
Having observed this, we now highlight the key idea in the following lemma.
\begin{lemma}
    Any principal $GL(n,\mathbb{R})$-bundle $\pi \colon P \to M$ equipped with tensorial $1$-form $\theta \in \Omega^1_{hor}(P;\mathbb{R}^n)^{GL(n;\mathbb{R})}$ induces an isomorphism $TM \cong P \times_{GL(n,\mathbb{R})} \mathbb{R}^n$ via Proposition \ref{pro:correspondencetensorialformsandassfbvaluedformsCap}.
    \label{lem:mypropositiontogettoframeresolutions}
\end{lemma}
\begin{proof}
    This follows from the fact that, as we pointed out, the canonical form $\tilde{\theta} \in \Omega^1_{hor}(FM;\mathbb{R}^n)^{GL(n;\mathbb{R})}$ induces the identification $TM \cong FM \times_{GL(n,\mathbb{R})} \mathbb{R}^n$. Then, by Theorem $\ref{thm:equivalencebetweenframeandprincipalwithsoldering}$, $P \cong FM$, and one has $P \times_{GL(n,\mathbb{R})} \mathbb{R}^n \cong FM \times_{GL(n,\mathbb{R})} \mathbb{R}^n \cong TM$. Thus, the bundle map $TM \to P\times_{GL(n,\mathbb{R})} \mathbb{R}^n$, corresponding to $\theta \in \Omega^1_{hor}(P;\mathbb{R}^n)^{GL(n;\mathbb{R})}$ via Proposition \ref{pro:correspondencetensorialformsandassfbvaluedformsCap}, is an isomorphism.
\end{proof}
\begin{remark}
    Alternatively, one could have started from noticing that the frame bundle is a principal $GL(n,\mathbb{R})$-bundle inducing an isomorphism $TM \cong FM \times_{GL(n,\mathbb{R})} \mathbb{R}^n$ through Proposition \ref{pro:correspondencetensorialformsandassfbvaluedformsCap} for $k=1$. Then, without proving Theorem \ref{thm:equivalencebetweenframeandprincipalwithsoldering}, as in \cite{majid_qbraided} Corollary 2.2, one could establish that any two principal $GL(n,\mathbb{R})$-bundles $P$ and $P^{\prime}$ endowed with tensorial forms $\theta \in \Omega^{1}_{hor}(P;\mathbb{R}^n)^{GL(n,\mathbb{R})}$ and $\theta^{\prime} \in \Omega^1_{hor}(P^{\prime};\mathbb{R}^n)^{GL(n,\mathbb{R})}$ which induce isomorphisms $P \times_{GL(n,\mathbb{R})} \mathbb{R}^n \cong TM \cong P^{\prime} \times_{GL(n,\mathbb{R})} \mathbb{R}^n$ are necessarily isomorphic as principal bundles. This reasoning directly underlies the quantum definition of a frame resolution, which will be given in the next section.
    \label{rem:onthereasoningbehindgettingtothequantumdefinitionlater}
\end{remark}
The above arguments make it clear that we can capture the essential properties of a frame bundle in the following definition.
\begin{definition}
    We call any $(P,M,V,\theta)$ with $\theta \in \Omega^1_{hor}(P;V)^G$ inducing an isomorphism $TM \cong P \times_G V$ like in Lemma \ref{lem:mypropositiontogettoframeresolutions} a \textsl{frame resolution} of the tangent bundle.
    \label{def:frameresolution}
\end{definition}
\begin{remark}
    Notice that in the definition above, we did not specify the structure group $G$ to be exactly $GL(n,\mathbb{R})$. Clearly, if we take $P = FM$ and $\theta$ the canonical form, any other $P^{\prime}$ with a tensorial form $\theta^{\prime}$ is isomorphic to it, and there will be an isomorphism between the tangent bundle and the associated bundle $P^{\prime} \times_{GL(n,\mathbb{R})} \mathbb{R}^n$. But one could have different examples of principal bundles with different structure groups inducing the same isomorphism: for example, one could consider the bundle of orthonormal frames with $G = O(n)$, or the bundle of affine frames with structure group $\mathbb{R}^n \rtimes GL(n, \mathbb{R})$. In particular, any $G$-structure can be viewed as a principal $G$-bundle $P \to M$ with a $\Theta \in \Omega^1_{hor}(P;\mathbb{R}^n)^{G}$ as shown in Proposition \ref{pro:thepowerofGstructures}, and one can prove that $(P,M,\mathbb{R}^n,\Theta)$ is a frame resolution.   
    \label{rem:onkeepingGandVgeneral}
\end{remark}

\section{Quantum frame resolutions} \label{quantumframeresolutions}

In the classical case, when defining a frame resolution, the obvious choice is to take $V = \mathbb{R}^{\text{dim}(M)}$. The quantum case, however, is more intricate, as one no longer has a well-defined $n$-dimensional manifold $M$. Instead, as discussed, the base space is replaced by the subalgebra of coinvariants $B = A^{coH}$ of a generic $H$-comodule algebra $A$.\\
With the aim of defining a quantum frame resolution, we begin by introducing the translation map and showing its properties \cite{brzezinski_translation}. This map will later be used to prove the algebraic analogues of Propositions \ref{pro:corrspondencebetweensectionsandequivariantmapsCap}
and \ref{pro:correspondencetensorialformsandassfbvaluedformsCap}, which justify the definition of quantum frame bundle as discussed in Remark \ref{rem:onthereasoningbehindgettingtothequantumdefinitionlater}.

\subsection{The translation map} \label{thetranslationmap}

Consider the diagonal right $H$-coaction on the tensor product $A \otimes H$ defined as
\begin{equation}
    \begin{split}
        \Delta^{ad}_{A \otimes H} \colon &A \otimes H \to A \otimes H \otimes H,\\
        &a \otimes h \mapsto a_0 \otimes h_2 \otimes a_1S(h_1)h_3,
    \end{split}
    \label{eq:diagonalrightHcoactionontensorproduct}
\end{equation}
and the right $H$-coaction on $A \otimes A$
\begin{equation}
    \begin{split}
        \Delta_{A \otimes A} \colon &A \otimes A \to A \otimes A \otimes H,\\
        &a \otimes a^{\prime} \mapsto (a_0 \otimes a^{\prime}_0) \otimes a_1a^{\prime}_1.
    \end{split}
    \label{eq:diagonalrightHcoactionontensorproductcoactionontensorprod}
\end{equation}
Define the following maps:
\begin{equation}
    \begin{split}
        \chi^{\prime} \colon &A \otimes A \to A \otimes H,\\
        &a \otimes a^{\prime} \mapsto aa^{\prime}_0 \otimes a^{\prime}_1,
    \end{split}
    \label{eq:diagonalrightHcoactionontensorproductauxmap1}
\end{equation}
and
\begin{equation}
    \begin{split}
        \nu \colon &A \otimes H \to H \otimes A \otimes H,\\
        &a \otimes h \mapsto a_1S(h_1) \otimes a_0 \otimes h_2.
    \end{split}
    \label{eq:diagonalrightHcoactionontensorproductauxmap2}
\end{equation}
We now prove a lemma showing properties of $\chi^{\prime}$ as defined in Equation (\ref{eq:diagonalrightHcoactionontensorproductauxmap1}).
\begin{lemma}
    Let $A$ be a right $H$-comodule algebra. Then
    \begin{enumerate}
        \item $(id \otimes \chi^{\prime}) \circ \big((flip_{A,H} \circ \Delta_A) \otimes id\big) = \nu \circ \chi^{\prime}$;
        \item $(\chi^{\prime} \otimes id) \circ (id \otimes \Delta_A) = (id \otimes \Delta) \circ \chi^{\prime}$;
        \item $(\chi^{\prime} \otimes id) \circ \Delta_{A \otimes A} = \Delta^{ad}_{A \otimes H} \circ \chi^{\prime}$.
    \end{enumerate}
    \label{lem:propertiesforchinotonthequotient}
\end{lemma}
\begin{proof}
    We consider left-hand side and right-hand side of each point and prove by direct calculation. In order:
    \begin{enumerate}
        \item \begin{equation}
            \begin{split}
                &(id \otimes \chi^{\prime}) \circ \big((flip_{A,H} \circ \Delta_A) \otimes id\big)(a \otimes a^{\prime})=\\
                &=(id \otimes \chi^{\prime}) \circ (a_1 \otimes a_0 \otimes a^{\prime})=\\
                &=a_1 \otimes \chi^{\prime}(a_0 \otimes a^{\prime})= a_1 \otimes a_0a^{\prime}_0 \otimes a^{\prime}_1.
            \end{split}
            \label{eq:prooflemma11}
        \end{equation}
        \begin{equation}
            \begin{split}
                &\nu \circ \chi^{\prime}(a \otimes a^{\prime})= \nu(aa^{\prime}_0 \otimes a^{\prime}_1) = (aa^{\prime}_0)_1S(a^{\prime}_{11}) \otimes (aa^{\prime}_0)_0 \otimes a^{\prime}_{12}=\\
                &=a_1a^{\prime}_1S(a^{\prime}_2) \otimes a_0a^{\prime}_0 \otimes a^{\prime}_3=a_1 \otimes a_0a^{\prime}_0 \otimes a^{\prime}_1.
            \end{split}
            \label{eq:prooflemma12}
        \end{equation}
        \item \begin{equation}
            \begin{split}
                &(\chi^{\prime} \otimes id) \circ (id \otimes \Delta_A) \circ (a \otimes a^{\prime})=(\chi^{\prime} \otimes id) \circ (a \otimes a^{\prime}_0 \otimes a^{\prime}_1)=\\
                &=\chi^{\prime}(a \otimes a^{\prime}_0) \otimes a^{\prime}_1= (aa^{\prime}_{00} \otimes a^{\prime}_{01})\otimes a^{\prime}_1 = aa^{\prime}_0 \otimes a^{\prime}_1 \otimes a^{\prime}_2.
                \end{split}
            \label{eq:prooflemma21}
        \end{equation}
        \begin{equation}
            \begin{split}
                &(id \otimes \Delta) \circ \chi^{\prime}(a \otimes a^{\prime})= (id \otimes \Delta) \circ (aa^{\prime}_0 \otimes a^{\prime}_1)=\\
                &= aa^{\prime}_0 \otimes a^{\prime}_{11} \otimes a^{\prime}_{12}= aa^{\prime}_0 \otimes a^{\prime}_1 \otimes a^{\prime}_2.
                \end{split}
            \label{eq:prooflemma22}
        \end{equation}
        \item \begin{equation}
            \begin{split}
                &(\chi^{\prime} \otimes id) \circ \Delta_{A \otimes A}(a \otimes a^{\prime})= (\chi^{\prime} \otimes id)\circ   \big((a_0 \otimes a^{\prime}_0) \otimes a_1a^{\prime}_1\big) =\\
                &=\chi^{\prime}(a_0 \otimes a^{\prime}_0) \otimes a_1a^{\prime}_1= a_0a^{\prime}_{00} \otimes a^{\prime}_{01} \otimes a_1a^{\prime}_1 =\\
                &=a_0a^{\prime}_0 \otimes a^{\prime}_1 \otimes a_1a^{\prime}_2.
                \end{split}
            \label{eq:prooflemma31}
        \end{equation}
        \begin{equation}
        \begin{split}
                &\Delta^{ad}_{A \otimes H} \circ \chi^{\prime}(a \otimes a^{\prime})=  \Delta^{ad}_{A \otimes H}(aa^{\prime}_0 \otimes a^{\prime}_1)= \\
                &=(aa^{\prime}_0)_0 \otimes  a^{\prime}_{12} \otimes  (aa^{\prime}_0)_1S(a^{\prime}_{11})a^{\prime}_{13}=a_0a^{\prime}_0 \otimes a^{\prime}_3 \otimes a_1a^{\prime}_1S(a^{\prime}_2)a^{\prime}_4=\\
                &= a_0a^{\prime}_0 \otimes a^{\prime}_1 \otimes a_1a^{\prime}_2.
                \end{split}
            \label{eq:prooflemma32}
        \end{equation}
    \end{enumerate}
\end{proof}
Let now $B \subseteq A$ be a quantum principal bundle. Relations derived in the Lemma \ref{lem:propertiesforchinotonthequotient} above for $\chi^{\prime}$ are also satisfied by the canonical map 
\begin{equation}
    \begin{split}
        \chi \colon &A \otimes_B A \to A \otimes H,\\
        &a \otimes_B a^{\prime} \to aa^{\prime}_0 \otimes a^{\prime}_1
    \end{split}
    \label{eq:recallcanonicalmap}
\end{equation}
on the quotient $A \otimes_B A$, which has a well defined inverse $\chi^{-1} \colon A \otimes H \to A \otimes_B A$.
\begin{definition}
    The \textsl{translation map} on a quantum principal bundle $B \subseteq A$ is defined as
    \begin{equation}
        \begin{split}
            \tau := \chi^{-1} \big|_{1_A \otimes H} \colon &H \to A \otimes_B A,\\
            &h \to \tau(h) = h^{\langle 1 \rangle} \otimes_B h^{\langle 2 \rangle}.
        \end{split}
        \label{eq:translmap}
    \end{equation}
    \label{def:translationmap}
\end{definition}
In the next proposition we show some properties of the translation map, which will be used later to prove important results.
\begin{proposition}
    The translation map $\tau \colon H \to A \otimes_B A$ in a quantum principal bundle $A(B, H)$, i.e. for a Hopf-Galois extension $B \subseteq A$, has the following properties:
    \begin{enumerate}
        \item $\big( (flip_{AH} \circ \Delta_A) \otimes_B id_A\big) \circ \tau = (S \otimes \tau) \circ \Delta$;
        \item $(id_A \otimes_B \Delta_A) \circ \tau = (\tau \otimes id_H) \circ \Delta$;
        \item $\Delta_{A \otimes A} \circ \tau = (\tau \otimes id_H) \circ Ad_R$;
        \item $\mu \circ \tau = 1_A \epsilon$;
    \end{enumerate}
    where $H$ has a right $H$-comodule structure by the coaction $Ad_R \colon H \to H \otimes H, h \mapsto h_2 \otimes S(h_1)h_3$.
    \label{pro:propertiesoftheBrztranslationmap}
\end{proposition}
\begin{proof}
Using the definition of the translation map and relations in Lemma \ref{lem:propertiesforchinotonthequotient} we prove the above in order.
 \begin{enumerate}
         \item  By the first property in Lemma \ref{lem:propertiesforchinotonthequotient} we have that
        \begin{equation}
            \begin{split}
                &(id_H \otimes \chi) \circ \big( (flip_{AH} \circ \Delta_A) \otimes id_A\big) \circ \tau(h) = (\nu \circ \chi \circ \tau)(h) = \\
                &= \nu (1_A \otimes h) = S(h_1) \otimes 1_A \otimes h_2,
            \end{split}
            \label{eq:propertiesofBrztranslmap1}
        \end{equation}
        since $\tau = \chi^{-1}\big|_{1_A \otimes H}$. Applying $(id_H \otimes \chi^{-1})$ to both sides we get
        \begin{equation}
            \begin{split}
                &(id_H \otimes \chi^{-1}) \circ (id_H \otimes \chi) \circ \big( (flip_{AH} \circ \Delta_A) \otimes id_A\big) \circ \tau (h) =\\
                &=\big( (flip_{AH} \circ \Delta_A) \otimes id_A\big) \circ \tau (h)= (id_H \otimes \chi^{-1}) \circ  (\nu \circ \chi \circ \tau)(h) = \\
                &= (id_H \otimes \chi^{-1}) \circ (S(h_1) \otimes 1_A \otimes h_2) = \big(S(h_1) \otimes \tau(h_2)\big) =\\
                &= S(h_1) \otimes (h_2)^{\langle 1 \rangle} \otimes_B (h_2)^{\langle 2 \rangle} = (S \otimes \tau) \circ \Delta(h),
            \end{split}
            \label{eq:propertiesofBrztranslmap2}
        \end{equation}
        where we have used that $\chi^{-1}(1_A \otimes h) = \tau(h)$ by definition of the translation map.
        \item Applying $(\chi^{-1} \otimes id_H)$ to both sides of the second property in \ref{lem:propertiesforchinotonthequotient}:
        \begin{equation}
            \begin{split}
                &(id_A \otimes_B \Delta_A) \circ \tau(h) = (\chi^{-1} \otimes id_H) \circ (id_A \otimes \Delta) \circ \chi \circ \tau(h) =\\
                &=(\chi^{-1} \otimes id_H) \circ (1_A \otimes h_1 \otimes h_2) = \tau(h_1) \otimes h_2 =\\
                &= (\tau \otimes id_H) \circ \Delta(h).
            \end{split}
            \label{eq:propertiesofBrztranslmap3}
        \end{equation}
        \item Consider this time the third property in Lemma \ref{lem:propertiesforchinotonthequotient} and again apply $(\chi^{-1} \otimes id_H)$ to both sides:
        \begin{equation}
            \begin{split}
                &\Delta_{A \otimes A} \circ \tau(h) = (\chi^{-1} \otimes id_H) \circ \Delta^{ad}_{A \otimes H} \circ \chi \circ \tau(h) =\\
                &= (\chi^{-1} \otimes id_H) \circ \Delta^{ad}_{A \otimes H}(1_A \otimes h) = (\chi^{-1} \otimes id_H) \circ (1_A \otimes h_2 \otimes S(h_1)h_3) = \\
                &= (\tau \otimes id_H) \circ Ad_R(h).
            \end{split}
            \label{eq:propertiesofBrztranslmap4}
        \end{equation}
        \item Using the definition of the translation map, for any $h \in H$ we have:
        \begin{equation}
            1_A \otimes h = \chi \circ \tau(h) = \chi(h^{\langle 1 \rangle} \otimes_B h^{\langle 2 \rangle}) = h^{\langle 1 \rangle} (h^{\langle 2 \rangle})_0 \otimes (h^{\langle 2 \rangle})_1. 
            \label{eq:propertiesofBrztranslmap5}
        \end{equation}
        Applying $(id_A \otimes \epsilon)$ to both sides and using the counit property $(id_A \otimes \epsilon) \circ \Delta_A = id_A$ we get:
        \begin{equation}
                1_A \otimes \epsilon(h) = h^{\langle 1 \rangle}(h^{\langle 2 \rangle})_0 \otimes \epsilon\big((h^{\langle 2 \rangle})_1\big) = h^{\langle 1 \rangle}h^{\langle 2 \rangle},
            \label{eq:propertiesofBrztranslmap6}
        \end{equation}
        from which we get the desired result.
    \end{enumerate}
\end{proof}

\subsection{Definition of quantum frame resolution} \label{defofquantumframeresolution}

We define as follows the quantum analogues of equivariant maps $C^{\infty}(P,F)^G$ on the total space of a principal $G$-bundle $P$ with values on the standard fiber $F$ of the associated bundle $P \times_G F$.
\begin{definition}
    Let $A$ be a right $H$-comodule algebra with coaction $\Delta_A$ and $V$ a right $H$-comodule with coaction $\rho_V$. We call a map $\phi \colon V \to A$ \textsl{pseudotensorial} if it is a morphism of right $H$-comodules such that $\phi(1_V) = 1_A$.
    \label{def:pseudotensorialmap}
\end{definition}
We are now ready to prove the quantum counterpart of Proposition \ref{pro:corrspondencebetweensectionsandequivariantmapsCap}, as in \cite{brzezinski_translation}, Theorem 4.3. This result is significant, and its proof method, which uses the translation map, is essential for grasping the definition of quantum frame resolution.
\begin{theorem}
    Let $H$ be a Hopf algebra with a bijective antipode. Cross sections of a quantum fiber bundle $\mathcal{E}(B, V, H)$ associated to a quantum principal bundle $B:=A^{coH} \subseteq A$, i.e. left $B$-module maps $s \colon \mathcal{E} \to B$ such that $s(1)=1$, are in bijective correspondence with pseudotensorial maps $\phi \colon V \to A$, i.e. $\Delta_A \circ \phi = (\phi \otimes id_H) \circ \rho_V$ such that $\phi(1_V) = 1_A$.
    \label{thm:thecorrespondence}
\end{theorem}
\begin{proof}
    Each equivariant map $\phi \colon V \to A$, i.e. $\Delta_A \circ \phi = (\phi \otimes id) \circ \rho_V$, induces a cross section $s \colon \mathcal{E} \to B$ by:
    \begin{equation}
        s = \mu \circ (id_A \otimes \phi).
        \label{eq:crosssectionfromphi}
    \end{equation}
    Let us show that for any $a \otimes v \in \mathcal{E}=(A \otimes V)^{coH}$, $s = \mu \circ (id_A \otimes \phi)$ has values in $B$. We have
    \begin{equation}
        \begin{split}
            &\Delta_A\big( s(a \otimes v) \big)=\Delta_A\big(\mu \circ (id_A \otimes \phi)(a \otimes v)\big)=\Delta_A(a\phi(v)) = \Delta_A(p)\Delta_A(\phi(v)) =\\
            &= (a_0 \otimes a_1) \big((\phi \otimes id_H) \circ \rho_V(v) \big) = (a_0 \otimes a_1)(\phi(v_0) \otimes v_1) =\\
            &= a_0\phi(v_0) \otimes a_1v_1 = \big(\mu \circ  (id_A \otimes \phi) \otimes id_H\big)(a_0 \otimes v_0 \otimes a_1v_1) =\\
            &= \big(\mu \circ (id_A \otimes \phi) \otimes id_H\big)\Delta_{\mathcal{E}}(a \otimes v) = \big(\mu \circ (id_A \otimes \phi) \otimes id_H\big)(a \otimes v \otimes 1_H)= \\
            &= a\phi(v) \otimes 1_H = s(a \otimes v) \otimes 1_H
        \end{split}
        \label{eq:suchanshasvaluesinM}
    \end{equation}
    where we used that $A$ is a right $H$-comodule algebra. Clearly $s(1_{\mathcal{E}}) = \mu(id_A \otimes \phi)(1_A \otimes 1_V) = \phi(1_V) = 1_A$ since $\phi(1_V) = 1_A$ by assumption. Moreover, the fact that $s = \mu(id_A \otimes \phi)$ is a left $B$-module map is apparent, since 
    \begin{equation}
        \begin{split}
            &s\big(b \triangleright (a \otimes v)\big)=\mu \circ (id_A \otimes \phi)(b \triangleright (a \otimes v))= \mu \circ (id_A \otimes \phi)(ba \otimes v)= \\
            &= ba\phi(v) = bs(a \otimes v)    
        \end{split}
        \label{eq:sfromphiisaleftBmodulemap}
    \end{equation}
    for any $b \in B$ and $a \otimes v \in \mathcal{E}$.\\
    Now, conversely, for any cross section $s \colon \mathcal{E} \to B$ of $\mathcal{E}$, we define a map $\phi \colon V \to A$ using the translation map $\tau \colon H \to A \otimes_B A, h \mapsto h^{\langle 1 \rangle} \otimes_B h^{\langle 2 \rangle}$ as:
    \begin{equation}
        \begin{split}
            \phi \colon &V \to A,\\
            &v \mapsto (S^{-1}(v_1))^{\langle 1 \rangle} s\big((S^{-1}(v_1))^{\langle 2 \rangle} \otimes v_0\big),
        \end{split}
        \label{eq:equivariantusingsection}
    \end{equation}
    which means $\phi := \mu \circ (id_A \otimes_B s) \circ (\tau \otimes id_V) \circ flip_{VH} \circ (id_V \otimes S^{-1}) \circ \rho_V$, $P(M, H)$. Notice that since $\tau(S^{-1}(v_1)) = (S^{-1}(v_1))^{\langle 1 \rangle} \otimes_B (S^{-1}(v_1))^{\langle 2 \rangle}$ is such that $(S^{-1}(v_1))^{\langle 1 \rangle}b \otimes_B (S^{-1}(v_1))^{\langle 2 \rangle} = (S^{-1}(v_1))^{\langle 1 \rangle} \otimes_B b(S^{-1}(v_1))^{\langle 2 \rangle}$ for any $b \in B$, the above definition for $\phi$ makes sense since the cross section $s$ is a left $B$-module map. \\
    In order for $\phi$ to be well-defined, the argument of $s$, i.e. $(S^{-1}(v_1))^{\langle 2 \rangle} \otimes v_0$, must be an element of $\mathcal{E}$. Alternatively, we show that $\tau(S^{-1}(v_1)) \otimes v_0$ is in $A \otimes_B \mathcal{E} = A \otimes_B (A \otimes V)^{coH}$. Using the second property in Proposition \ref{pro:propertiesoftheBrztranslationmap}, we have 
    \begin{equation}
        \begin{split}
            &(id_A \otimes \Delta_A) \circ \tau(S^{-1}(v_1)) = (S^{-1}(v_1))^{\langle 1 \rangle} \otimes_B \big( (S^{-1}(v_1))^{\langle 2 \rangle} \big)_0 \otimes \big( (S^{-1}(v_1))^{\langle 2 \rangle} \big)_1 =\\
            &=(\tau \otimes id_H) \circ \Delta(S^{-1}(v_1)) = (\tau \otimes id_H) \circ \big(S^{-1}(v_{12}) \otimes S^{-1}(v_{11}) \big) = \\
            &= ( S^{-1}(v_{12}) )^{\langle 1 \rangle} \otimes_B ( S^{-1}(v_{12}) )^{\langle 2 \rangle} \otimes S^{-1}(v_{11}),
        \end{split}
        \label{eq:secondproperty}
    \end{equation}    
    where we used that $\Delta \circ S^{-1} = (S^{-1} \otimes S^{-1}) \circ flip \circ \Delta$ by Proposition \ref{pro:Hopfantipodeproperties}. Thus, for any $v \in V$, we have
    \begin{equation}
        \begin{split}
            &(id_A \otimes_B \Delta_{\mathcal{E}})(\tau(S^{-1}(v_1)) \otimes v_0) =\\
            &= (id_A \otimes_B \Delta_{\mathcal{E}})\big((S^{-1}(v_1))^{\langle 1 \rangle} \otimes_B (S^{-1}(v_1))^{\langle 2 \rangle} \otimes v_0\big) = \\
            &= (S^{-1}(v_1))^{\langle 1 \rangle} \otimes_B \big((S^{-1}(v_1))^{\langle 2 \rangle}\big)_0 \otimes v_{00} \otimes \big((S^{-1}(v_1))^{\langle 2 \rangle}\big)_1v_{01} =\\
            &=( S^{-1}(v_{12}) )^{\langle 1 \rangle} \otimes_B  ( S^{-1}(v_{12}) )^{\langle 2 \rangle} \otimes v_{00} \otimes S^{-1}(v_{11}) v_{01} =\\
            &= ( S^{-1}(v_1) )^{\langle 1 \rangle} \otimes_B  ( S^{-1}(v_1) )^{\langle 2 \rangle} \otimes v_0 \otimes S^{-1}(v_{01}) v_{11} = \\
            &= \tau(S^{-1}(v_1)) \otimes v_0 \otimes 1_H,
        \end{split}
        \label{eq:phiiswelldefined}
    \end{equation}
    where we used that $(\rho_V \otimes id_H) \circ \rho_V = v_{00} \otimes v_{01} \otimes v_1 = (id \otimes \Delta) \circ \rho_V = v_0 \otimes v_{11} \otimes v_{12}$, which shows that $\tau(S^{-1}(v_1)) \otimes v_0 \in A \otimes_B \mathcal{E}$.
    Moreover, $\phi(1_V)=1_A$, since $s(1_{\mathcal{E}})=1_A$ and also $\tau(1_H) = 1_A \otimes_B 1_A$.\\
    We have to now show that $\phi \colon V \to A$ is pseudotensorial. By the first property in Proposition \ref{pro:propertiesoftheBrztranslationmap} we have
    \begin{equation}
        \begin{split}
            &\big( (flip_{AH} \circ \Delta_A) \otimes_B id_A \big) \circ \tau(S^{-1}(v_1)) =\\
            &=\big((S^{-1}(v_1))^{\langle 1 \rangle}\big)_1 \otimes \big( (S^{-1}(v_1))^{\langle 1 \rangle} \big)_0 \otimes_B (S^{-1}(v_1))^{\langle 2 \rangle} = \\
            &= (S \otimes \tau) \circ \Delta(S^{-1}(v_1)) = (S \otimes \tau) \circ \big( S^{-1}(v_{12}) \otimes S^{-1}(v_{11}) \big) = \\
            &=S(S^{-1}(v_{12})) \otimes (S^{-1}(v_{11}))^{\langle 1 \rangle} \otimes_B (S^{-1}(v_{11}))^{\langle 2 \rangle},
        \end{split}
        \label{eq:firstproperty}
    \end{equation}
    hence we can write
    \begin{equation}
        \begin{split}
            &\Delta_A \circ \phi(v) = \Delta_A\Big((S^{-1}(v_1))^{\langle 1 \rangle} s\big((S^{-1}(v_1))^{\langle 2 \rangle} \otimes v_0\big)\Big) =\\
            &= \Delta_A\Big((S^{-1}(v_1))^{\langle 1 \rangle}\Big)\Delta_A\Big(s\big((S^{-1}(v_1))^{\langle 2 \rangle}\otimes v_0\big)\Big) = \\
            &= \Big(\big((S^{-1}(v_1))^{\langle 1 \rangle}\big)_0 \otimes \big((S^{-1}(v_1))^{\langle 1 \rangle}\big)_1\Big) \Big(s\big((S^{-1}(v_1))^{\langle 2 \rangle}\otimes v_0\big) \otimes 1_H\Big) =\\
            &=\big((S^{-1}(v_1))^{\langle 1 \rangle}\big)_0s\big((S^{-1}(v_1))^{\langle 2 \rangle}\otimes v_0\big) \otimes \big((S^{-1}(v_1))^{\langle 1 \rangle}\big)_1 = \\
            &=(S^{-1}(v_{11}))^{\langle 1 \rangle} s\big( (S^{-1}(v_{11}))^{\langle 2 \rangle} \otimes v_0 \big) \otimes S(S^{-1}(v_{12}))=\\
            &=(S^{-1}(v_{01}))^{\langle 1 \rangle} s\big( (S^{-1}(v_{01}))^{\langle 2 \rangle} \otimes v_{00} \big) \otimes v_1=\\
            &= (\phi \otimes id_H) \circ \rho_V(v),
        \end{split}
        \label{eq:phipseudotensorial}
    \end{equation}
    for any $v \in V$. Hence we have shown that $\phi$ is an equivariant (or pseudotensorial) map.\\
    We now want to show that the maps $\theta \colon \phi \to s=\mu \circ (id_A \otimes \phi)$ and $\tilde{\theta} \colon s \to \phi = \mu \circ (id_A \otimes_B s) \circ (\tau \otimes id_V) \circ flip_{VH} \circ (id_V \otimes S^{-1}) \circ \rho_V$ are inverses of each other. For any $v \in V$ and $\phi \colon V \to A$ equivariant we have:
    \begin{equation}
        \begin{split}
            &(\tilde{\theta} \circ \theta)(\phi)(v) = (S^{-1}(v_1))^{\langle 1 \rangle} \theta(\phi)\big((S^{-1}(v_1))^{\langle 2 \rangle} \otimes v_0\big) =\\
            &=(S^{-1}(v_1))^{\langle 1 \rangle} \big(\mu \circ (id_A \otimes \phi) \big)\big((S^{-1}(v_1))^{\langle 2 \rangle} \otimes v_0\big)=\\
            &=(S^{-1}(v_1))^{\langle 1 \rangle}(S^{-1}(v_1))^{\langle 2 \rangle}\phi(v_0) = \phi(v),
        \end{split}
        \label{eq:thetatildetheta}
    \end{equation}
    where in the last equality we have used the last property in Proposition \ref{pro:propertiesoftheBrztranslationmap}.\\ Conversely, for any $a \otimes v \in \mathcal{E}$ and any cross section $s \colon \mathcal{E} \to B$, we have
    \begin{equation}
        (\theta \circ \tilde{\theta})(s)(a \otimes v) = a(\tilde{\theta}(s))(v) = a(S^{-1}(v_1))^{\langle 1 \rangle} s\big((S^{-1}(v_1))^{\langle 2 \rangle} \otimes v_0\big).
        \label{eq:thetathetatilde1}
    \end{equation}
    Using again the first property of Proposition \ref{pro:propertiesoftheBrztranslationmap} we find
    \begin{equation}
        \begin{split}
            &\big( (flip_{PH} \circ \Delta_A) \otimes_B id_A \otimes id_V \big) \big(a(S^{-1}(v_1))^{\langle 1 \rangle} \otimes_B (S^{-1}(v_1))^{\langle 2 \rangle} \otimes v_0 \big)=\\
            &= a_1\big((S^{-1}(v_1))^{\langle 1 \rangle}\big)_1 \otimes a_0\big((S^{-1}(v_1))^{\langle 1 \rangle}\big)_0 \otimes_B (S^{-1}(v_1))^{\langle 2 \rangle} \otimes v_0 =\\
            &= a_1S(S^{-1}(v_{12})) \otimes a_0(S^{-1}(v_{11}))^{\langle 1 \rangle} \otimes_B (S^{-1}(v_{11}))^{\langle 2 \rangle} \otimes v_0 =\\
            &= a_1v_{12} \otimes a_0(S^{-1}(v_{11}))^{\langle 1 \rangle} \otimes_B (S^{-1}(v_{11}))^{\langle 2 \rangle} \otimes v_0 =\\
            &= id_H \otimes \big(( \mu \otimes_B id_A \otimes id_V) \circ (id_A \otimes \tau \circ S^{-1} \otimes id_V) \circ \\
            &\;\;\;\;\;\;\;\;\;\;\;\;\;\;\;\;\;\;\;\;\;\;\;\;\;\;\;\circ (id_A \otimes flip_{VH} \circ \rho_V)\big)(a_1v_1 \otimes a_0 \otimes v_0) =\\
            &= a_1v_1 \otimes a_0(S^{-1}(v_{01}))^{\langle 1 \rangle} \otimes_B (S^{-1}(v_{01}))^{\langle 2 \rangle} \otimes v_{00} =\\
            &= id_H \otimes \big(( \mu \otimes_B id_A \otimes id_V) \circ (id_A \otimes \tau \circ S^{-1} \otimes id_V) \circ \\
            &\;\;\;\;\;\;\;\;\;\;\;\;\;\;\;\;\;\;\;\;\;\;\;\;\;\;\;\circ (id_A \otimes flip_{VH} \circ \rho_V)\big)(1_H \otimes a \otimes v) =\\
            &= 1_H \otimes a(S^{-1}(v_1))^{\langle 1 \rangle} \otimes_B (S^{-1}(v_1))^{\langle 2 \rangle} \otimes v_0,
        \end{split}
        \label{eq:thetathetatilde2}
    \end{equation}
    where in the last two equalities we have used that $a \otimes v \in \mathcal{E}$. From the above we see that $a(S^{-1}(v_1))^{\langle 1 \rangle} \otimes_B (S^{-1}(v_1))^{\langle 2 \rangle} \otimes v_0  \in B \otimes_B \mathcal{E}$, since we have already shown that $(S(v_1))^{\langle 2 \rangle} \otimes v_0 \in \mathcal{E}$. Thus, since $a(S^{-1}(v_1))^{\langle 1 \rangle} \in B$, we can use in \ref{eq:thetathetatilde1} that $s$ is a left $B$-module map and finally write
    \begin{equation}
        \begin{split}
            &(\theta \circ \tilde{\theta})(s)(a \otimes v)  = a(S^{-1}(v_1))^{\langle 1 \rangle} s\big((S^{-1}(v_1))^{\langle 2 \rangle} \otimes v_0\big) = \\
            &=s\big(a(S^{-1}(v_1))^{\langle 1 \rangle} (S^{-1}(v_1))^{\langle 2 \rangle} \otimes v_0 \big) = s (a \otimes v),
        \end{split}
        \label{eq:thetathetatilde4}
    \end{equation}
    where we have used again the fourth property of Proposition \ref{pro:propertiesoftheBrztranslationmap}, completing the proof.
\end{proof}
\begin{remark}
    Analogously, one could take $V$ a left $H$-comodule algebra with coaction $\leftindex_V\Delta$ and view it as a right $H^{op}$-comodule algebra with coaction $\rho_V$ as in Proposition \ref{pro:thetwodefinitionsofassqvbareequivalent}, without needing to require that $H$ has bijective antipode. Then, to each cross section $s \colon \mathcal{E} \to B$ corresponds a pseudotensorial map $\phi(v) := (v_{-1})^{\langle 1 \rangle}s\big((v_{-1})^{\langle 2 \rangle} \otimes v_0\big)$, and one can proceed as in the above proof to show the same results.
    \label{rem:ondoingthesamebutwiththeotherdefinitionofassociatedqvb}
\end{remark}
Let $B := A^{coH} \subseteq A$ again be a principal $H$-comodule algebra. In the following, we assume $(\Gamma_A, \mathrm{d}_A)$ is a right $H$-covariant first order differential calculus on $A$, with right $H$-coaction $\Delta_{\Gamma}$. Moreover, as in Definition \ref{def:baseformsFODC}, we have the pullback calculus $(\Gamma_B,\mathbb{d}_B)$ on $B$, and $\Gamma_A^{hor} := A\Gamma_B$ is the $A$-bimodule of horizontal forms.\\
The next definition gives the quantum analogue of horizontal and equivariant $V$-valued $1$-forms $\Omega^1_{hor}(P;V)^G$ on the total space of a principal fiber bundle.
\begin{definition}
    A form $\theta \colon V \to \Gamma_A$ is called \textsl{right strongly tensorial} if its image lies in $A\Gamma_B = \Gamma_A^{hor}$ and if it satisfies
    \begin{equation}
        \Delta_{\Gamma} \circ \theta = (\theta \otimes id_H) \circ \rho_V.
        \label{eq:rightstronglytensorialforms}
    \end{equation}
    \label{def:rightstronglytensorialforms}
\end{definition}
Proceeding like in Section \ref{frameresolutions}, we state, as a corollary to Theorem \ref{thm:thecorrespondence}, the quantum analogue of Proposition \ref{pro:correspondencetensorialformsandassfbvaluedformsCap}.
\begin{corollary}
    Right strongly tensorial $\theta \colon V \to A\Gamma_B$ are in bijective correspondence with left $B$-module maps $\sigma \colon \mathcal{E} \to \Gamma_B$.
    \label{cor:thecorrespondence2}
\end{corollary}
\begin{proof}
    The right $H$-covariant first order differential calculus $(\Gamma_A, \mathrm{d}_A)$ has right $H$-coaction
    \begin{equation}
        \begin{split}
            \Delta_{\Gamma} \colon &\Gamma_A \to \Gamma_A \otimes H,\\
            &a\mathrm{d}_Aa^{\prime} \mapsto a_0\mathrm{d}_Aa^{\prime}_0 \otimes a_1a^{\prime}_1 = \Delta_A(a)\Delta_{\Gamma}\big(\mathrm{d}_A(a^{\prime})\big),
        \end{split}
        \label{eq:coactionforrightHcovatiantfirstorderdiffcalcGammaA}
    \end{equation}
    for all $a,a^{\prime} \in A$, with $\Delta_{\Gamma} \circ \mathrm{d}_A =(\mathrm{d}_A \otimes id_H) \circ \Delta_A$.
    Given a strongly tensorial $\theta \colon V \to A\Gamma_B$, we can define a left $B$-module map $\sigma = \mu \circ (id_A \otimes \theta) \colon \mathcal{E} \to \Gamma_B$. This map has indeed values in $\Gamma_B$, since for any $a \otimes v \in \mathcal{E}$ we have
    \begin{equation}
        \begin{split}
            &\Delta_{\Gamma}\big(\sigma(a \otimes v) \big) = \Delta_{\Gamma}\big(a\theta(v)\big) = \Delta_A(a)\Delta_{\Gamma}\big(\theta(v)\big) = \\
            &= (a_o \otimes a_1)\big(\theta(v_0) \otimes v_1) = a_0\theta(v_0) \otimes a_1v_1 = a\theta(v) \otimes 1_H \in A\Gamma_B \otimes 1_H,
        \end{split}
        \label{eq:sigmagiventheta}
    \end{equation}
    and $\Gamma_B \cong \big(A\Gamma_B)^{coH}$ by Theorem \ref{thm:horizontalformsintersectionandsplitting}.\\
    Conversely, given a left $B$-module map $\sigma \colon \mathcal{E} \to \Gamma_B$, we can define a right strongly tensorial $\theta$ as in Theorem \ref{thm:thecorrespondence}, using the translation map, as $\theta(v) = (S^{-1}(v_1))^{\langle 1 \rangle}\sigma\big( (S^{-1}(v_1))^{\langle 2 \rangle} \otimes v_0\big)$ for every $v \in V$. This map clearly has values in $A\Gamma_B$, and is equivariant since
    \begin{equation}
        \begin{split}
            &\Delta_{\Gamma} \circ \theta(v) = \Delta_{\Gamma}\Big( (S^{-1}(v_1))^{\langle 1 \rangle}\sigma\big( (S^{-1}(v_1))^{\langle 2 \rangle} \otimes v_0\big) \Big) = \\
            &= \Delta_A\Big((S^{-1}(v_1))^{\langle 1 \rangle} \Big)\Delta_{\Gamma}\Big(\sigma\big( (S^{-1}(v_1))^{\langle 2 \rangle} \otimes v_0\big) \Big) = \\
            &= \Big( (S^{-1}(v_1))^{\langle 1 \rangle}\big)_0 \otimes (S^{-1}(v_1))^{\langle 1 \rangle}\big)_1 \Big)\Big(\sigma\big((S^{-1}(v_1))^{\langle 2 \rangle}\otimes v_0\big) \otimes 1_H\Big) =\\
            &=\big((S^{-1}(v_1))^{\langle 1 \rangle}\big)_0\sigma\big((S^{-1}(v_1))^{\langle 2 \rangle}\otimes v_0\big) \otimes \big((S^{-1}(v_1))^{\langle 1 \rangle}\big)_1 = \\
            &=(S^{-1}(v_{11}))^{\langle 1 \rangle} \sigma\big( (S^{-1}(v_{11}))^{\langle 2 \rangle} \otimes v_0 \big) \otimes S(S^{-1}(v_{12}))=\\
            &= (\theta \otimes id_H) \circ \rho_V(v),
        \end{split}
        \label{eq:equivarianceofthetaquantum}
    \end{equation}
    by using the first property of Proposition \ref{pro:propertiesoftheBrztranslationmap} exactly as in the proof of Theorem \ref{thm:thecorrespondence}.
\end{proof}
We can finally define a quantum frame resolution of $\Gamma_B$, i.e., the quantum analogue of a frame resolution of the tangent bundle, as follows.
\begin{definition}
    A \textsl{quantum frame resolution} of $(B, \Gamma_B)$ is a quantum $H$-principal bundle $B = A^{coH} \subseteq A$, a 
    right $H$-comodule $V$ and a right strongly tensorial form $\theta \colon V \to A\Gamma_B$ such that the map 
    \begin{equation}
        s_{\theta} = \cdot \circ (id_A \otimes \theta) \colon \mathcal{E} \to \Gamma_B
        \label{eq:toevidentiatethesthetamap}
    \end{equation}
    is an isomorphism of $B$-bimodules. We denote such a quantum frame resolution by the tuple $(A,H,V,\theta)$.
    \label{def:quanutumframeresolution}
\end{definition}
\begin{remark}
    In the above definition, we did not require $V$ to be an algebra: its right $H^{op}$-comodule algebra structure is necessary to show that it forms a subalgebra of the tensor product $A \otimes V$, as established in Lemma $\ref{lem:toproveHopquantumvectorbundleisok}$, but it is not needed for proving Theorem \ref{thm:thecorrespondence} and Corollary \ref{cor:thecorrespondence2}. Therefore, we may consider $V$ as simply a right $H$-comodule. This will be the case in the example of quantum frame resolution we will show in \ref{exampleofqfr}.
    \label{rem:onwhyonedoesnotneedVtobeanalgebra}
\end{remark}

\subsection{An example: the quantum affine case} \label{exampleofqfr}
We consider the smashed product algebra $B\#H$, with $B = \mathbb{C}^2_q$ the quantum plane, i.e., the algebra generated by elements $x,y$ such that
\begin{equation}
    xy = qyx,
    \label{eq:quantumplane}
\end{equation}
and $H = GL_q(2)$ the Hopf algebra generated by elements $a,b,c,d,D^{-1}$ as in Example \ref{ex:GLq2}, satisfying the quantum commutation relations
\begin{equation}
    \begin{split}
        & ab = qba, \;\; ac=qca, \;\; bd = qdb, \;\; cd=qdc, \\
        &bc=cb, \;\; ad-da=(q-q^{-1})bc,
    \end{split}
    \label{eq:GLq2commutationrelationsrecall}
\end{equation}
and
\begin{equation}
    \begin{split}
        &D^{-1}a = aD^{-1}, \;\; D^{-1}b=bD^{-1}, \;\; D^{-1}c = cD^{-1}, \;\; D^{-1}d= dD^{-1}.
    \end{split}
    \label{eq:GLq2commutationrelationsrecall1}
\end{equation}
We recall that, for any $b'b^{\prime} \in B$ and $h,h^{\prime} \in H$, the smashed product multiplication is
\begin{equation}
    (b\#h)(b^{\prime}\#h^{\prime}) = b(h_1 \triangleright b^{\prime})\#h_2h^{\prime},
    \label{eq:recallsmashedproductdefinition}
\end{equation}
where the left $H$-module action $\triangleright$ on $B$ is given by the following lemma (\cite{hajac_frame}, Lemma 3.1), which one proves by directly verifying that such an action respects the ideals defining $B$ and $H$.
\begin{lemma}
    There exists a left $H = GL_q(2)$-module action on $B = \mathbb{C}^2_q$ given by
    \begin{equation}
        \begin{split}
            &a\triangleright x = q^{-2}x, \;\; b\triangleright x = 0, \;\; c\triangleright x = (q^{-2} -1)y, \;\; d \triangleright x = q^{-1}x, \;\; D^{-1} \triangleright x = q^3x,\\
            &a \triangleright y = q^{-1}y, \;\; b \triangleright y= 0, \;\; c\triangleright y =0,\;\; d \triangleright y = q^{-2}y,\;\; D^{-1}\triangleright y = q^3y.
        \end{split}
        \label{eq:leftHmoduleactiononBgenerators}
    \end{equation}
    \label{lem:leftHmoduleanctiononB}
\end{lemma}
We now show that the smash product algebra $B\#H$ can be seen as an Hopf algebra generated by elements $x\#1_H, y\#1_H, 1_B\#a, 1_B\#b, 1_B\#c, 1_B\#d, 1_B\#D^{-1}$. In doing this, we construct the quantum group of affine transformations of the quantum plain. This coincides with Majid's version, which uses a braded Hopf coproduct on $B$ \cite{aziz_majid}.\\
The generators satisfy the following cross relations, which one can verify by direct calculation using (\ref{eq:recallsmashedproductdefinition}) and the left $H$-module action defined in (\ref{eq:leftHmoduleactiononBgenerators}):
\begin{equation}
    \begin{split}
        &(x\#1_H)(y\#1_H) = q(y\#1_H)(x\#1_H)
    \end{split}
    \label{eq:commutationrelations1}
\end{equation}
\begin{equation}
    \begin{split}
        &(1_B\#a)(1_B\#b) = q(1_B\#b)(1_B\#a), \;\; (1_B\#a)(1_B\#c) = q(1_B\#c)(1_B\#a), \\
        &(1_B\#b)(1_B\#d) = q(1_B\#d)(1_B\#b), \;\; (1_B\#c)(1_B\#d) = q(1_B\#d)(1_B\#c),\\
        &(1_B\#b)(1_B\#c)=(1_B\#c)(1_B\#b),\\
        &(1_B\#a)(1_B\#d) = (1_B\#d)(1_B\#a) + (q - q^{-1})(1_B\#b)(1_B\#c)
    \end{split}
    \label{eq:commutationrelations11}
\end{equation}
\begin{equation}
    \begin{split}
        &(1_B\#D^{-1})(1_B\#a) = (1_B\#a)(1_B\#D^{-1}),\\
        &(1_B\#D^{-1})(1_B\#c) = (1_B\#c)(1_B\#D^{-1}),\\
        &(1_B\#D^{-1})(1_B\#c) = (1_B\#c)(1_B\#D^{-1}),\\
        &(1_B\#D^{-1})(1_B\#d) = (1_B\#d)(1_B\#D^{-1}),
    \end{split}
    \label{eq:commutationrelations111}
\end{equation}
in complete analogy with (\ref{eq:quantumplane}), (\ref{eq:GLq2commutationrelationsrecall}) and (\ref{eq:GLq2commutationrelationsrecall1}), and, in addition
\begin{equation}
    \begin{split}
        &(1_B\#a)(x\#1_H) = q^{-2}(x\#1_H)(1_B\#a), \;\; (1_B\#b)(x\#1_H) = q^{-2}(x\#1_H)(1_B\#b),\\
        &(1_B\#c)(x\#1_H) = (q^{-2}-1)(y\#1_H)(1_B\#a) + q^{-1}(x\#1_H)(1_B\#c),\\
        &(1_B\#d)(x\#1_H) = (q^{-2}-1)(y\#1_H)(1_B\#b) + q^{-1}(x\#1_H)(1_B\#d),\\
        &(1_B\#a)(y\#1_H) = q^{-1}(y\#1_H)(1_B\#a), \;\; (1_B\#b)(y\#1_H) = q^{-1}(y\#1_H)(1_B\#b),\\
        &(1_B\#c)(y\#1_H) = q^{-2}(y\#1_H)(1_B\#c), \;\; (1_B\#d)(y\#1_H) = q^{-2}(y\#1_H)(1_B\#d),
    \end{split}
    \label{eq:qcommutationrelationsBsmashH1}
\end{equation}
\begin{equation}
    \begin{split}
        &(x\#1_H)(1_B\#D^{-1}) = q^{-3}(1_B\#D^{-1})(x\#1_H),\\
        &(y\#1_H)(1_B\#D^{-1}) = q^{-3}(1_B\#D^{-1})(y\#1_H).
    \end{split}
    \label{eq:qcommutationrelationsBsmashH2}
\end{equation}
We endow $B\#H$ with a coalgebra structure by defining a coproduct $\Delta_{\#} \colon B\#H \to B\#H \otimes B\#H$, a counit $\epsilon_{\#} \colon B\#H \to \mathbb{C}$, and an antipode $S_{\#}$ defined as follows.
The coproduct is
\begin{equation}
    \begin{split}
    &\Delta_{\#} \begin{pmatrix}
        (1_B\#1_H) &0 &0\\
        (x\#1_H) &(1_B\#a) &(1_B\#b)\\
        (y\#1_H) &(1_B\#c) &(1_B\#d)
    \end{pmatrix} = \\
    &=\begin{pmatrix}
        (1_B\#1_H) &0 &0\\
        (x\#1_H) &(1_B\#a) &(1_B\#b)\\
        (y\#1_H) &(1_B\#c) &(1_B\#d)
    \end{pmatrix} \dot{\otimes} \begin{pmatrix}
        (1_B\#1_H) &0 &0\\
        (x\#1_H) &(1_B\#a) &(1_B\#b)\\
        (y\#1_H) &(1_B\#c) &(1_B\#d)
    \end{pmatrix},
    \end{split}
    \label{eq:coproductmatrixBsmashH}
\end{equation}
that is
\begin{equation}
    \begin{split}
        &\Delta_{\#}(1_B\#a) = (1_B\#a) \otimes (1_B\#a) + (1_B\#b) \otimes (1_B\#c),\\
        &\Delta_{\#}(1_B\#b) = (1_B\#a) \otimes (1_B\#b) + (1_B\#b) \otimes (1_B\#d),\\
        &\Delta_{\#}(1_B\#c)=(1_B\#c)\otimes(1_B\#a) + (1_B\#d) \otimes (1_B\#c),\\
        &\Delta_{\#}(1_B\#d)=(1_B\#c) \otimes (1_B\#b) + (1_B\#d) \otimes (1_B\#d),\\
        &\Delta_{\#}(1_B\#D^{-1})= (1_B\#D^{-1}) \otimes (1_B\#D^{-1}),\\
        &\Delta_{\#}(x\#1_H) = (x\#1_H) \otimes (1_B\#1_H) + (1_B \# a) \otimes (x \# 1_H) + (1_B\# b) \otimes (y \# 1_H),\\
        &\Delta_{\#}(y\#1_H) = (y\#1_H) \otimes (1_B\#1_H) + (1_B \# c) \otimes (x \# 1_H) + (1_B\#d) \otimes (y \# 1_H),
    \end{split}
    \label{eq:coproductBsmashH}
\end{equation}
the counit is defined as
\begin{equation}
    \epsilon_{\#}\begin{pmatrix}
        (1_B\#1_H) &0 &0\\
        (x\#1_H) &(1_B\#a) &(1_B\#b)\\
        (y\#1_H) &(1_B\#c) &(1_B\#d)
    \end{pmatrix} =\begin{pmatrix}
        1 &0 &0\\ 0 &1 &0 \\0 &0 &1
    \end{pmatrix}
    \label{eq:counitmatrixBsmashH}
\end{equation}
and finally the antipode
\begin{equation}
    \begin{split}
        &S_{\#}(1_B\#a) = (1_B\#D^{-1})(1_B\#d),\\
        &S_{\#}(1_B\#b) = -q^{-1}(1_B\#D^{-1})(1_B\#b),\\
        &S_{\#}(1_B\#c) = -q (1_B\#D^{-1})(1_B\#c),\\
        &S_{\#}(1_B\#d) =(1_B\#D^{-1})(1_B\#a),\\
        &S_{\#}(1_B\#D^{-1})=(1_B\#D),\\
        &S_{\#}(x\#1_H) = (1_B\#D^{-1})\Big(q^{-1}(1_B\#b)(y\#1_H) -(1_B\#d)(x\#1_H) \Big),\\
        &S_{\#}(y\#1_H) = (1_B\#D^{-1})\Big( q(1_B\#c)(x\#1_H) - (1_B\#a)(y\# 1_H) \Big).
    \end{split}
    \label{eq:antipodeBsmashH}
\end{equation}
For brevity, we show coassociativity, counitality and antipode property are satisfied on $(x\#1_H)$, but one can analogously check the Hopf algebra properties on all generators.\\
Coassociativity reads
\begin{equation}
    \begin{split}
        &(\Delta_{\#} \otimes id) \circ \Delta_{\#}(x\#1) = \Delta_{\#}(x\#1) \otimes (1\#1) +\Delta_{\#}(1\#a) \otimes (x\#1) + \\
        &+ \Delta_{\#}(1\#b) \otimes (y\#1)=\\
        &=(x\#1) \otimes (1\#1) \otimes (1\#1) + (1 \# a) \otimes (x \# 1) \otimes (1\#1) + \\
        &+(1\# b) \otimes (y \# 1) \otimes (1\#1) + (1\#a) \otimes (1\#a) \otimes (x\#1)+\\
        &+(1\#b) \otimes (1\#c) \otimes (x\#1) + (1\#a) \otimes (1\#b) \otimes(y\# 1) +\\
        &+ (1\#b) \otimes (1\#d) \otimes (y\#1);\\
        &(id \otimes \Delta_{\#}) \circ \Delta_{\#}(x\#1) = (x\#1) \otimes \Delta_{\#}(1\#1) + (1\#a) \otimes \Delta_{\#}(x\#1) +\\
        &+ (1\#b) \otimes \Delta_{\#}(y\#1) = \\
        &= (x\#1) \otimes (1\#1) \otimes (1\#1) + (1\#a) \otimes (x\#1) \otimes (1\#1) +\\
        &+(1\#a) \otimes (1\#a) \otimes (x\#1) + (1\#a) \otimes (1\#b) \otimes (y\#1) +\\
        &+(1\#b) \otimes (y\#1) \otimes (1\#1) + (1\#b) \otimes (1\#c) \otimes (x \#1) +\\
        &+(1\#b) \otimes (1\#d) \otimes (y\#1).
     \end{split}
    \label{eq:coassociativityBsmashH}
\end{equation}
Counitality reads
\begin{equation}
    \begin{split}
        &(\epsilon_{\#} \otimes id) \circ \Delta_{\#}(x\#1) = \epsilon_{\#}(x\#1) \otimes (1\#1) + \epsilon_{\#}(1\#a) \otimes (x\#1) + \epsilon_{\#}(1\#b) \otimes (y\#1) =\\
        &= \epsilon_{\#}(x\#1)(1\#1) + \epsilon_{\#}(1\#a)(x\#1) + \epsilon_{\#}(1\#b)(y\#1) = 0 + (x\#1) + 0 = (x\#1);\\
        &(id \otimes \epsilon_{\#}) \circ \Delta_{\#}(x\#1) =  (x\#1)\epsilon_{\#}(1\#1) + (1\#a)\epsilon_{\#}(x\#1) + (1\#b)\epsilon_{\#}(y\#1) =\\
        &= (x\#1) + 0 + 0 = (x\#1).
    \end{split}
    \label{eq:counitalityBsmashH}
\end{equation}
Finally, the antipode property reads:
\begin{equation}
    \begin{split}
        &\mu_{\#} \circ (S_{\#} \otimes id) \circ \Delta_{\#}(x\#1) = S_{\#}(x\#1)(1\#1) + S_{\#}(1\#a)(x\#1) +\\
        &+ S_{\#}(1\#b)(y\#1) = \\
        &= q^{-1}(1\#D^{-1})(1\#b)(y\#1) - (1\#D^{-1})(1\#d)(x\#1) + (1\#D^{-1})(1\#d)(x\#1) +\\
        &-q^{-1}(1\#D^{-1})(1\#b)(y\#1) = 0 =\eta_{\#} \circ \epsilon_{\#}(x\#1);
    \end{split} 
    \label{eq:antipodepropertyBsmashH1}
\end{equation}
on the other hand, using $q$-commutation relations (\ref{eq:qcommutationrelationsBsmashH1}), (\ref{eq:qcommutationrelationsBsmashH2}) and associativity of the smashed product algebra, by a direct calculation we obtain
\begin{equation}
    \begin{split}
        &\mu_{\#} \circ (id \otimes S_{\#}) \circ \Delta_{\#}(x\#1) = (x\#1)S_{\#}(1\#1) + (1\#a)S_{\#}(x\#1) +\\
        &+ (1\#b)S_{\#}(y\#1) = \\
        &= (x\#1)(1\#1) + q^{-1}(1\#a)(1\#D^{-1})(1\#b)(y\#1)+\\
        &- (1\#a)(1\#D^{-1})(1\#d)(x\#1) + q(1\#b)(1\#D^{-1})(1\#c)(x\#1)+ \\
        &- (1\#b)(1\#D^{-1})(1\#a)(y\#1) =\\
        &= (x\#1) + q^{-2}(1\#a)(1\#D^{-1})(y\#1)(1\#b)  +\\
        &- (q^{-2} -1)(1\#a)(1\#D^{-1})(y\#1)(1\#b) - q^{-1}(1\#a)(1\#D^{-1})(x\#1)(1\#d)+\\
        &+  q(q^{-2}-1)(1\#b)(1\#D^{-1})(y\#1)(1\#a) +qq^{-1}(1\#b)(1\#D^{-1})(x\#1)(1\#c) +\\
        &- q^{-1}(1\#b)(1\#D^{-1})(y\#1)(1\#a) =\\
        &= (x\#1) + (1\#a)(1\#D^{-1})(y\#1)(1\#b) +\\
        &- q^{-1}(1\#a)(1\#D^{-1})(x\#1)(1\#d)-q(1\#b)(1\#D^{-1})(y\#1)(1\#a)+\\
        &+(1\#b)(1\#D^{-1})(x\#1)(1\#c) +\\
        &=(x\#1) + (1\#aD^{-1})(y\#b) - q^{-1}(1\#aD^{-1})(x\#d)-q(1\#bD^{-1})(y\#a)+\\
        &+(1\#bD^{-1})(x\#c) =\\
        &= (x\#1) + (aD^{-1} \triangleright y)\#aD^{-1}b - q^{-1}(aD^{-1} \triangleright x)\#aD^{-1}d +\\
        &- q(aD^{-1} \triangleright y)\#bD^{-1}a+(aD^{-1} \triangleright x)\#bD^{-1}c =\\
        &= (x\#1) + q^2(y\#aD^{-1}b) - (x\#aD^{-1}d) - q^3(y\#bD^{-1}a)+\\
        &+q(x\#bD^{-1}c) =\\
        &= (x\#1) + q^3(y\#D^{-1}ba) - (x\#D^{-1}ad) - q^3(y\#D^{-1}ba)+\\
        &+q(x\#D^{-1}bc) =  (x\#1) - \big(x\#D^{-1}(ad-qbc)\big) = \\
        &= (x\#1)- (x\#1) = 0
    \end{split}
    \label{eq:antipodepropertyBsmashH2secondtry}
\end{equation}
We are now able to view $B\#H$ as an Hopf algebra. We observe that this construction coincides with the one given by Majid in \cite{aziz_majid}, within the context of braided Hopf algebras.
\begin{definition}
    A \textsl{quantum homogeneous principal bundle} is a Hopf algebra surjection $\bar{\pi} \colon A \to H$ such that the right $H$-coaction $\Delta_A := (id \otimes \bar{\pi}) \circ \Delta$, with $\Delta$ the coproduct on $A$, makes $A$ a quantum principal bundle over $B = A^{coH}$.
    \label{def:quantumhomogeneousprincipalbundle}
\end{definition}
In our case, $B\#1_H \subseteq B\#H$ is quantum homogeneous principal bundle via the right $H$-coaction on $A = B\#H$ given by
\begin{equation}
    \Delta_{B\#H} = (id \otimes \bar{\pi}) \circ \Delta_{\#},
    \label{eq:rightHcoactiononBsmashH}
\end{equation}
where $\bar{\pi} \colon B\#H \to H$ is the Hopf algebra surjection sending any element of the form $(b\#h)$ to $0$, and any element of the form $(1_B\#h)$ to $h$. One can check that the subalgebra of coinvariants under this coaction is $B\#1_H$, as expected. This implies that $B\#1_H := (B\#H)^{coH} \subseteq A = B\#H$ is a quantum homogeneous principal bundle.\\
We define the right $H$-comodule
\begin{equation}
    V := \ker({\epsilon_{\#}}) \cap (B\#1_H),
    \label{eq:Vinourexample}
\end{equation}
with right $H$-coaction
\begin{equation}
    \rho_V(v) = v_2 \otimes \bar{\pi}\big(S_{\#}(v_1)\big),
    \label{eq:rightHcoactiononVinourexample}
\end{equation}
for every $v \in V$, as in \cite{majid_qbraided} Proposition 4.3.\\ 
In order to define $\theta$, we need a right $H$-covariant first order differential calculus on $B\#H$. As discussed in Section \ref{smashedproductcalculus}, we assume a right $H$-covariant first order differential calculus $(\Gamma_B,\mathrm{d}_B)$ on $B$ and an $H$-bicovariant first order differential calculus $(\Gamma_H,\mathrm{d}_H)$ on $H$, so that the smashed product calculus $(\Gamma_{\#}, \mathrm{d}_{\#})$, with
\begin{equation}
    \Gamma_{\#} := \Gamma_B \otimes H + B \otimes \Gamma_H,
    \label{eq:recallGammasmash}
\end{equation}
and
\begin{equation}
    \mathrm{d}_{\#}(b\#h) := \mathrm{d}_Bb \otimes h + b \otimes \mathrm{d}_Hh
    \label{eq:recalldsmash}
\end{equation}
for all $b \in B$ and $h \in H$, is a right $H$-covariant first order differential calculus on $B\#H$ by Theorem \ref{thm:actualsmashproductcalculus} and Corollary \ref{cor:smashproductcalculusisrightHcovariant}.
As in Equation (\ref{eq:rightHcoactiononsmashcalctomakeitrightcovariant}), we recall the right $H$-coaction on $\Gamma_{\#}$:
\begin{equation}
    \begin{split}
        \Delta_{\Gamma_{\#}} \colon &\Gamma_{\#} \to \Gamma_{\#} \otimes H,\\
        &\omega_B \otimes h + b \otimes \omega_H \mapsto \omega_B \otimes h_1 \otimes h_2 + b \otimes (\omega_H)_0 \otimes (\omega_H)_1
    \end{split}
    \label{eq:recallrightHcoactiononsmashcalc}
\end{equation}
for all $\omega_B \in \Gamma_B$, $\omega_H \in \Gamma_H$, $b \in B$ and $h \in H$.\\
Inspired by the classical case, we define $\theta$ as the restriction of the quantum Maurer-Cartan form \ref{def:quantumMCdef} to $B\#1_H$:
\begin{equation}
    \theta(v) = S_{\#}(v_1) \cdot \mathrm{d}_{\#}(v_2),
    \label{eq:thetainourexample}
\end{equation}
for all $v \in V$, where $\cdot$ is the $B\#H$ action on $\Gamma_{\#}$ like in (\ref{eq:leftBsmashHactiononGammasmash}).\\
Let us check that it is a right strongly $H$-tensorial form.
First, we need to verify it has values in the horizontal forms $(B\#H)\Gamma_{B\#1_H} = (B\#H)(\Gamma_B \otimes 1_H)$. We check this on elements $(x\#1)$ and $(y\#1)$ generating $V$:
\begin{equation}
    \begin{split}
        &\theta(x\#1) = \sum_iS_{\#}(x\#1)_{1_i}\cdot\mathrm{d}_{\#} (x\#1)_{2_i} =\\
        &=S_{\#}(x\#1)\cdot\mathrm{d}_{\#}(1\#1) + S_{\#}(1\#a)\cdot\mathrm{d}_{\#}(x\#1) + S_{\#}(1\#b)\cdot\mathrm{d}_{\#}(y\#1) =\\
        &=(1\#D^{-1}d) \cdot (\mathrm{d}_Bx \otimes 1_H) -q^{-1}(1\#D^{-1}b)\cdot (\mathrm{d}_By \otimes 1_H),
    \end{split}
    \label{eq:thetaonxsmash1simple}
\end{equation}
and
\begin{equation}
    \begin{split}
        &\theta(y\#1) = \sum_iS_{\#}(y\#1)_{1_i}\cdot\mathrm{d}_{\#} (y\#1)_{2_i} =\\
        &=S_{\#}(y\#1)\cdot\mathrm{d}_{\#}(1\#1) + S_{\#}(1\#c)\cdot\mathrm{d}_{\#}(x\#1) + S_{\#}(1\#d)\cdot\mathrm{d}_{\#}(y\#1) =\\
        &=-q(1\#D^{-1}c) \cdot (\mathrm{d}_Bx \otimes 1_H) + (1\#D^{-1}a)\cdot (\mathrm{d}_By \otimes 1_H).
    \end{split}
    \label{eq:thetaonysmash1simple}
\end{equation}
Let us now check equivariance, that is
\begin{equation}
    \Delta_{\Gamma_{\#}} \circ \theta = (\theta \otimes id) \circ \rho_V.
    \label{eq:equivarianceofthetainourexample}
\end{equation}
We show calculations on $(x\#1)$. In order to act with $\Delta_{\Gamma_{\#}}$ as in (\ref{eq:recallrightHcoactiononsmashcalc}), we need to compute, using (\ref{eq:leftBsmashHactiononGammasmash}),(\ref{eq:HmoduleFODCwillimply}) and (\ref{eq:leftHmoduleactiononBgenerators}):
\begin{equation}
    \begin{split}
        &\theta(x\#1) = (1\#D^{-1}d) \cdot (\mathrm{d}_Bx \otimes 1_H) -q^{-1}(1\#D^{-1}b)\cdot (\mathrm{d}_By \otimes 1_H) = \\
        &=(1\#D^{-1}c \mathrel{\scriptstyle \to} \mathrm{d}_Bx) \otimes D^{-1}b + (1\#D^{-1}d \mathrel{\scriptstyle \to} \mathrm{d}_Bx) \otimes D^{-1}d \\
        &-q^{-1}(D^{-1}a \mathrel{\scriptstyle \to} \mathrm{d}_By) \otimes D^{-1}b - q^{-1}(D^{-1}b \mathrel{\scriptstyle \to} \mathrm{d}_By) \otimes D^{-1}d =\\
        &=(D^{-1}d \triangleright 1)\mathrm{d}_B(D^{-1}c \triangleright x) \otimes D^{-1}b + (D^{-1}d \triangleright 1_B)\mathrm{d}_B(D^{-1}d \triangleright x) \otimes D^{-1}d +\\
        &-q^{-1}(D^{-1}a \triangleright 1)\mathrm{d}_B(D^{-1}a \triangleright y) =\\
        &=(q^{-2} - 1)q^3 \mathrm{d}_By \otimes D^{-1}b + q^2\mathrm{d}_Bx \otimes D^{-1}d -q^{-1}q^2\mathrm{d}_By \otimes D^{-1}b =\\
        &=q^2\mathrm{d}_Bx \otimes D^{-1}d - q^3\mathrm{d}_By \otimes D^{-1}b,
    \end{split}
    \label{eq:thetaonxsmash1hard}
\end{equation}
so that
\begin{equation}
    \begin{split}
        &\Delta_{\Gamma_{\#}} \circ \theta(x\#1) = \Delta_{\Gamma_{\#}}(q^2\mathrm{d}_Bx \otimes D^{-1}d - q^3\mathrm{d}_By \otimes D^{-1}b) =\\
        &= q^2 \mathrm{d}_Bx \otimes D^{-1}c \otimes D^{-1}b + q^2\mathrm{d}_Bx \otimes D^{-1}d \otimes D^{-1}d +\\
        &-q^3\mathrm{d}_By \otimes D^{-1}a \otimes D^{-1}b - q^3\mathrm{d}_By \otimes D^{-1}b \otimes D^{-1}d.
    \end{split}
    \label{eq:lefthandsideequivarianceexamplex}
\end{equation}
The right $H$-coaction on $(x\#1)$ reads
\begin{equation}
    \begin{split}
        &\rho_V(x\#1) = \sum_i(x\#1)_{2_i} \otimes \bar{\pi}\big(S_{\#}(x\#1)_{1_i}\big) =\\
        &= (1\#1) \otimes \bar{\pi}\big(S_{\#}(x\#1)\big) + (x\#1) \otimes \bar{\pi}\big(S_{\#}(1\#a)\big) + (y\#1)\otimes \bar{\pi}\big(S_{\#}(1\#b)\big) =\\
        &=(1\#1) \otimes \bar{\pi}\big( q^{-1}(1\#D^{-1}b)(y\#1) - (1\#D^{-1}d)(x\#1)\big) + (x\#1) \otimes \bar{\pi}(1\#D^{-1}d) +\\
        &-q^{-1}(y \# 1) \otimes \bar{\pi}(1\#D^{-1}b) = \\
        &=(1\#1) \otimes \bar{\pi}\big(q^{-1}(D^{-1}a \triangleright y)\#D^{-1}b - (D^{-1}c \triangleright x)\#D^{-1}b - (D^{-1}d \triangleright x)\#D^{-1}d\big) +\\
        &+ (x\#1) \otimes D^{-1}d - q^{-1}(y \#1) \otimes D^{-1}b =\\
        &=(x\#1) \otimes D^{-1}d - q^{-1}(y \#1) \otimes D^{-1}b.
    \end{split}
    \label{eq:rhoVonxsmash1}
\end{equation}
Similarly to what we did in (\ref{eq:thetaonxsmash1hard}), one computes
\begin{equation}
    \theta(y\#1) = q^4\mathrm{d}_By \otimes D^{-1}a - q^3\mathrm{d}_Bx \otimes D^{-1}c,
    \label{eq:thetaonysmash1hard}
\end{equation}
so that we can write
\begin{equation}
    \begin{split}
        &(\theta \otimes id) \circ \rho_V(x\#1) =\theta(x\#1) \otimes D^{-1}d - q^{-1}\theta(y\#1) \otimes D^{-1}b =\\
        &= q^2\mathrm{d}_Bx \otimes D^{-1}d \otimes D^{-1}d - q^3\mathrm{d}_By \otimes D^{-1}b \otimes D^{-1}d +\\
        &-q^3\mathrm{d}_By \otimes D^{-1}a \otimes D^{-1}b + q^2\mathrm{d}_Bx \otimes D^{-1}c \otimes D^{-1}b,
    \end{split}
    \label{eq:righthandsideequivarianceexamplex}
\end{equation}
which is exactly what we found in (\ref{eq:lefthandsideequivarianceexamplex}), showing equivariance of $\theta$ on $(x\#1_B)$. Analogously, one can check
\begin{equation}
    \begin{split}
        &\Delta_{\Gamma_{\#}} \circ \theta(y\#1) =q^4\mathrm{d}_By \otimes D^{-1}a \otimes D^{-1}a + q^4\mathrm{d}_By \otimes D^{-1}b \otimes D^{-1}c +\\
        &-q^2\mathrm{d}_Bx \otimes D^{-1}c \otimes D^{-1}a -q^3\mathrm{d}_Bx \otimes D^{-1}d \otimes D^{-1}c =\\
        &=(\theta \otimes id) \circ \rho_V(y\#1).
    \end{split}
    \label{eq:equivarianceonyinexampleriassunto}
\end{equation}
Thus, $\theta$ in Equation (\ref{eq:thetainourexample}) is well-defined and is a good candidate in order to make $(B\#H,H,V,\theta)$ a quantum frame resolution. If that's the case, then the restriction of
\begin{equation}
    s_{\theta} := \cdot \circ (id \otimes \theta)
    \label{eq:sthetaexample}
\end{equation}
to elements in $\mathcal{E}$ is a left $B\#1_H$-module isomorphism $\mathcal{E} \cong \Gamma_{B\#1_H} = (B\#1_H) \cdot \mathrm{d}_{\#}(B\#1_H)$. In fact, one can verify that $s_{\theta}$ has an inverse $s_{\theta}^{-1} \colon \Gamma_{B\#1_H} \to \mathcal{E}$ defined as
\begin{equation}
    \begin{split}
        &s_{\theta}^{-1}\big((z\#1_H) \cdot \mathrm{d}_{\#}(x\#1_H)\big) := (z\#a) \otimes (x\#1) + (z\#b) \otimes (y\#1),\\
        &s_{\theta}^{-1}\big((z\#1_H) \cdot \mathrm{d}_{\#}(y\#1_H)\big) := (z\#c) \otimes (x\#1) + (z\#d) \otimes (y\#1),
    \end{split}
    \label{eq:theinversefrommajidworks}
\end{equation}
for any $z \in B$. We observe that this formula seems to come from the left-coaction of $H$ on $B$. We will further analyze this in later studies
First of all, one can check that $s_{\theta}^{-1}$ has values in $\mathcal{E}$, which is the set of coinvariants under the right $H$-coaction $\Delta_{\mathcal{E}}$ on $B\#H \otimes V$ defined as
\begin{equation}
    \Delta_{\mathcal{E}} = (id_{B\#H} \otimes id_V \otimes \mu) \circ (id_{B\#H} \otimes flip \otimes id_V) \circ (\Delta_{B\#H} \otimes \rho_V),
    \label{eq:DeltamathcalEexample}
\end{equation}
with $\Delta_{B\#H}$ as in (\ref{eq:rightHcoactiononBsmashH}) and $\rho_V$ as in $\ref{eq:rightHcoactiononVinourexample}$. For example, on $(x\#1) \cdot \mathrm{d}_{\#}(y\#1) \in \Gamma_{B\#1_H}$, one gets
\begin{equation}
    \begin{split}
        &\Delta_{\mathcal{E}}\circ s_{\theta}^{-1}\big((x\#1) \cdot \mathrm{d}_{\#}(y\#1)\big) = \\
        &= \Delta_{\mathcal{E}}\big((x\#c) \otimes (x\#1) + (x\#d) \otimes (y\#1) \big) =\\
        &=\Delta_{\mathcal{E}}\big( (x\#c) \otimes (x\#1) \big) + \Delta_{\mathcal{E}}\big( (x\#d) \otimes (y\#1) \big),
    \end{split}
    \label{eq:exampleofwelldefinitenessoftheinverseinourexample}
\end{equation}
and after some calculations, the first term reads
\begin{equation}
    \begin{split}
        &\Delta_{\mathcal{E}}\big( (x\#c) \otimes (x\#1) \big) = (x\#c) \otimes (x\#1) \otimes D^{-1}ad + \\
        &- q^{-1}(x\#c) \otimes (y\#1) \otimes D^{-1}ab +(x\#d) \otimes (x\#1) \otimes D^{-1}cd +\\
        &- q^{-1}(x\#d) \otimes (y\#1) \otimes D^{-1}cb,
    \end{split}
    \label{eq:secondnowfirsttermexample}
\end{equation}
while the second
\begin{equation}
    \begin{split}
        &\Delta_{\mathcal{E}}\big( (x\#d) \otimes (y\#1) \big) = -q(x\#c) \otimes (x\#1) \otimes D^{-1}bc +\\
        &+ (x\#c) \otimes (y\#1) \otimes D^{-1}ba -q(x\#d) \otimes (x\#1) \otimes D^{-1}dc +\\
        &+(x\#d) \otimes (y\#1) \otimes D^{-1}da.
    \end{split}
    \label{eq:thirdnowsecondtermexample}
\end{equation}
Summing them up, by using $q$-commutation relations for $H$, one sees that indeed $s_{\theta}^{-1}\big((x\#1) \cdot \mathrm{d}_{\#}(y\#1)\big) \in \mathcal{E}$.\\
Moreover, by direct calculation, one verifies that
\begin{equation}
    \begin{split}
        &s_{\theta}\circ s_{\theta}^{-1}\big((x\#1) \cdot \mathrm{d}_{\#}(y\#1)\big) = \\
        &= s_{\theta}\big( (xy\#1) \otimes (1\#1) + (x\#c) \otimes (x\#1) + (x\#d) \otimes (y\#1) \big) = \\
        &= q^2(x\#c) \cdot (\mathrm{d}_Bx \otimes 1) - q^3(x\#c) \cdot (\mathrm{d}_By \otimes D^{-1}b) +\\
        &+ q^4(x\#d) \cdot (\mathrm{d}_By \otimes 1_H) - q^3(x\#d) \cdot (\mathrm{d}_Bx \otimes D^{-1}c) = \\
        &= q^2 \big( (q^{-2} - 1)x\mathrm{d}_By \otimes aD^{-1}d + q^{-1} x\mathrm{d}_Bx \otimes cD^{-1}d \big) + \\
        &-q^3\big( q^{-2}x\mathrm{d}_By \otimes cD^{-1}b \big) + q^4\big(q^{-2} x\mathrm{d}_By \otimes dD^{-1}a \big) + \\
        &-q^3 \big( (q^{-2} - 1)x\mathrm{d}_By \otimes bD^{-1}c + q^{-1}x\mathrm{d}_Bx \otimes dD^{-1}c\big) = \\
        &= x\mathrm{d}_By \otimes D^{-1}ad - q^2 x\mathrm{d}_By \otimes D^{-1}da - q^2(q - q^{-1})x\mathrm{d}_By \otimes D^{-1}bc +\\
        &+q^2x\mathrm{d}_Bx \otimes D^{-1}dc - q x\mathrm{d}_By \otimes D^{-1}cb + q^2x\mathrm{d}_By \otimes D^{-1}da +\\
        &-q x\mathrm{d}_By \otimes D^{-1}bc + q^3x\mathrm{d}_By \otimes D^{-1}bc - q^2 x\mathrm{d}_Bx \otimes D^{-1}dc =\\
        &= x\mathrm{d}_By \otimes D^{-1}ad - q x\mathrm{d}_By \otimes D^{-1}cb = x\mathrm{d}_By \otimes 1 = \\
        &=(x\#1) \cdot \mathrm{d}_{\#}(y\#1).
    \end{split}
    \label{eq:sthetaisindeedaninverseinthisexample}
\end{equation}
Thus, $\big(B\#H,H,V,\theta)$, with $B=\mathbb{C}_q^2$, $H=GL_q(2)$, $V = \ker{\epsilon_{\#}}\cap \mathbb{C}_q^2\#1_H$ and $\theta(v) = S_{\#}(v_1)\mathrm{d}_{\#}v_2$ for every $v \in V$ is a quantum frame resolution.

\section{Quantum $G$-structures} \label{quantumGstructure}

In this section, let $(A,H,V,\theta)$ be a quantum frame resolution. Recapping, this means that $s_{\theta} := \cdot \circ (id_A \otimes \theta)$ is an isomorphism $\mathcal{E}=(A \otimes V)^H \xrightarrow[]{\sim} \Gamma_B$ of left $B$-modules, $V$ is a right $H$-comodule with right $H$-coaction $\rho_V \colon V \to V \otimes H$, and $B:=A^{coH} \subseteq A$ is a principal $H$-comodule algebra, with $\Delta_A \colon A \to A \otimes H$ the corresponding right $H$-coaction. In order to define such a quantum frame resolution, we assumed a right $H$-covariant first order differential calculus $(\Gamma_A, \mathrm{d}_A)$, with $\Delta_{\Gamma_A} \colon \Gamma_A \to \Gamma_A \otimes H$ such that $\Delta_{\Gamma_A} \circ \mathrm{d}_A = (\mathrm{d}_A \otimes id_{H}) \circ \Delta_A$. We construct the pullback calculus $(\Gamma_B, \mathrm{d}_B)$ of base forms on $B$ by the inclusion $\iota \colon B \to A, b \mapsto b$, with $\Gamma_B := \iota(B)\mathrm{d}_A|_B\iota(B) = B\mathrm{d_A}|_BB$, and $\mathrm{d}_B := \mathrm{d}_A \circ \iota = \mathrm{d}_A|_B$.\\
We want to study quantum reductions $A_0$ of $A$ to a Hopf algebra $H_0 := H/J$ for some Hopf ideal $J$, with $\pi \colon H \to H_0$ a Hopf algebra map. In particular, let $\phi \colon A \to A_0$ be the surjective morphism of right $H_0$-comodule algebras defining the quantum reduction. Note that, as discussed, $A$ is a right $H_0$-comodule algebra under the coaction $\Delta_A^0 := (id_A \otimes \pi)$. Using Propositions \ref{pro:pullbackcalculusandquotientcalculus} and $\ref{pro:pullbackandquotientcovariant}$, we construct the quotient calculus $(\Gamma_{A_0}, \mathrm{d}_{A_0})$ on $A_0$, with right $H_0$-coaction $\Delta_{\Gamma_{A_0}} \colon \Gamma_{A_0} \to \Gamma_{A_0} \otimes H_0$. In particular, $\Gamma_{A_0} := \Gamma_A / \Gamma_I$, with $\Gamma_I = I\mathrm{d}_A + A\mathrm{d}_AI$ for $I = \ker{\phi}$, and
\begin{equation}
    \begin{split}
        \mathrm{d}_{A_0} \colon &A_0 \to \Gamma_{A_0},\\
        &\phi(a) \mapsto \mathrm{d}_{A_0}\phi(a) = [\mathrm{d}_Aa]
    \end{split}
    \label{eq:differentialforGammaA0}
\end{equation}
for any $a \in A$. This means we have defined a morphism of differential calculi $(\phi, \Phi_{\Gamma})$, with
\begin{equation}
    \begin{split}
        \Phi_{\Gamma} \colon &\Gamma_A \to \Gamma_{A_0},\\
        &a\mathrm{d}_Aa^{\prime} \mapsto \phi(a)\mathrm{d}_{A_0}\phi(a^{\prime})
    \end{split}
    \label{eq:PhiGamma}
\end{equation}
for any $a,a^{\prime} \in A$, such that the $A$-actions on $\Gamma_A$ descend to $A_0$-actions on $\Gamma_{A_0}$. From this, we have the pullback calculus $(\Gamma_{B_0}, \mathrm{d}_{B_0})$ on $B_0$ by the inclusion $\iota_0 \colon B_0 \to A_0, \bar{b} \mapsto \bar{b}$, with $\Gamma_{B_0} := \iota_0(B_0)\mathrm{d_{A_0}}|_{B_0}\iota_0(B_0) = B_0\mathrm{d_{A_0}}|_{B_0}B_0$ and $d_{B_0} := \mathrm{d_{A_0}} \circ \iota_0 = \mathrm{d}_{A_0}|_{B_0}$.\\
By this construction, since $\phi(B) = B_0$ by Definition \ref{def:quantumreduction}, the restriction of $\Phi_{\Gamma}$ to $\Gamma_B \subseteq \Gamma_A$ reads
\begin{equation}
    \begin{split}
        \Phi_{\Gamma}\Big|_{\Gamma_B} \colon &\Gamma_B \to \Gamma_{B_0},\\
        &b\mathrm{d}_Bb^{\prime} \mapsto \phi(b)\mathrm{d}_{A_0}|_{B_0}\phi(b^{\prime})
    \end{split}
    \label{eq:restrictionofGammaPhitoGammaB}
\end{equation}
for all $b,b^{\prime} \in B$. This map is invertible, since $\mathrm{d}_B = \mathrm{d}_A|_B = \mathrm{d}_{A_0}|_{B_0} \circ \phi = \mathrm{d}_{B_0} \circ \phi$ and, by Definition \ref{def:quantumreduction}, we assumed $B \cong B_0$ as algebras, so that $\phi|_B$ is a bijection. Thus, $\Gamma_B \cong \Gamma_{B_0}$.\\
We are now ready to define the quanutum analogue of a $G$-structure.
\begin{definition}
    Let $(A,H,V,\theta)$ a quantum frame resolution. Let $A_0$ be a quantum reduction of $A$ to the Hopf algebra $H_0 := H/J$, with $J$ an Hopf ideal. We call the tuple $(A_0,H_0,V,\theta_0)$, with
    \begin{equation}
        \theta_0 := \Phi_{\Gamma} \circ \theta,
        \label{eq:theta0definition}
    \end{equation}
    a \textsl{quantum} $G$\textsl{-structure} of $(A,H,V,\theta)$.
    \label{def:quantumGstructure}
\end{definition}
Note that the above definition is well-posed only if $\theta_0$ is right strongly $H_0$-tensorial. Let us show this in the following lemma.
\begin{lemma}
    Let $\theta \colon V \to A\Gamma_B$ a right strongly $H$-tensorial form, i.e., it has values in $\Gamma_A^{hor} = A\Gamma_B$ and it satisfies $\Delta_{\Gamma_A} \circ \theta = (\theta \circ id_H) \circ \rho_V$. Then $\theta_0 := \Phi_{\Gamma} \circ \theta \colon V \to \Gamma_{A_0}$ is right strongly $H_0$-tensorial, i.e., it has values in $\Gamma_{A_0}^{hor} = A_0\Gamma_{B_0}$ and satisfies $\Delta_{\Gamma_{A_0}} \circ \theta_0 = (\theta_0 \otimes id_{H_0}) \circ \rho_V^0$, where $\rho_V^0 := (id_V \otimes \pi) \circ \rho_V$.
    \label{lem:theta0isequivariant}
\end{lemma}
\begin{proof}
    Since we know $\theta$ has values in $A\Gamma_B$, in order to show $\theta_0$ has values in $A_0\Gamma_{B_0}$ it is enough to check that the restriction of $\Phi_{\Gamma}$ to $A\Gamma_B$ has values in $A_0\Gamma_{B_0}$. This is true, since we can write a generic element of $A\Gamma_B$ as $a\mathrm{d}_A|_Bb = a\mathrm{d}_Bb$, with $a \in A$ and $b \in B$, so that clearly, knowing that $\phi(B) = B_0$,
    \begin{equation}
        \Phi_{\Gamma}\big|_{A\Gamma_B}(a\mathrm{d}_Bb) = \phi(a)\mathrm{d}_{A_0}\phi(b) = \phi(a)\mathrm{d}_{A_0}|_{B_0}\phi(b) = \phi(a)\mathrm{d}_{B_0}\phi(b) \in A_0\Gamma_{B_0},
        \label{eq:horizontalcheckPhiGamma}
    \end{equation}
    From $\Delta_{\Gamma_A} \circ \theta = (\theta \circ id_H) \circ \rho_V$, we deduce $H_0$-equivariance of $\theta$ with respect to $\Delta_{\Gamma_A}^0:=(id_{\Gamma_A} \otimes \pi) \circ \Delta_{\Gamma_A}$ and $\rho_V^0$, i.e.
    \begin{equation}
        \Delta_{\Gamma_A}^0 \circ \theta = (id_{\Gamma_A} \otimes \pi) \circ \Delta_{\Gamma_A} \circ \theta = (\theta \otimes \pi) \circ \rho_V = (\theta \otimes id_{H_0}) \circ \rho_V^0,
        \label{eq:theta0isgood1}
    \end{equation}
    which for all $v \in V$ reads
    \begin{equation}
        \Delta_{\Gamma_A}^0 \circ \theta(v) = (id_{\Gamma_A} \otimes \pi) \circ \Delta_{\Gamma_A} \circ \theta(v) = \theta(v_0) \otimes \pi(v_1).
        \label{eq:theta0isgood2}
    \end{equation}
    The right $H_0$-coaction on the covariant first order differential calculus $\Gamma_{A_0}$ is defined as
    \begin{equation}
        \begin{split}
            \Delta_{\Gamma_{A_0}} \colon &\Gamma_{A_0} \to \Gamma_{A_0} \otimes H_0,\\
            &\phi(a)\mathrm{d}_{A_0}\phi(a^{\prime}) \mapsto \phi(a)_0\mathrm{d}_{A_0}\phi(a^{\prime})_0 \otimes \phi(a)_1\phi(a^{\prime})_1.
        \end{split}
        \label{eq:equivarianceconsiderations1}
    \end{equation}
    Since $\phi \colon A \to A_0$ is a morphism of right $H_0$-comodules, i.e.,
    \begin{equation}
        \Delta_{A_0} \circ \phi(a) = \phi(a)_0 \otimes \phi(a)_1 = (\phi \otimes id_{H_0}) \circ \Delta_{A}^0(a) = \phi(a_0) \otimes \pi(a_1)
        \label{eq:equivarianceconsiderations2}
    \end{equation}
    for all $a \in A$, the right $H_0$-coaction (\ref{eq:equivarianceconsiderations1}) acts as
    \begin{equation}
        \Delta_{\Gamma_{A_0}}\big(\phi(a)\mathrm{d}_{A_0}\phi(a^{\prime})\big) = \phi(a_0)\mathrm{d}_{A_0}\phi(a^{\prime}_0) \otimes \pi(a_1)\pi(a^{\prime}_1).
        \label{eq:equivarianceconsiderations3}
    \end{equation}
    We want to verify that
    \begin{equation}
        \Delta_{\Gamma_{A_0}} \circ \theta_0 = (\theta_0 \otimes id_{H_0}) \circ \rho_V^0 = (\theta_0 \otimes \pi) \circ \rho_V.
        \label{eq:theta0isgood3}
    \end{equation}
    Let us call $\theta(v) := a\mathrm{d}_A|_Bb$ for some $a \in A$, $b \in B$ (summation understood). Then, using Equation (\ref{eq:equivarianceconsiderations3}),  the left-hand side of (\ref{eq:theta0isgood3}) reads
    \begin{equation}
        \begin{split}
            &\Delta_{\Gamma_{A_0}} \circ \theta_0(v) = \Delta_{\Gamma_{A_0}}\big|_{A_0\Gamma_{B_0}} \circ  \Phi_{\Gamma} \circ \theta(v) = \Delta_{\Gamma_{A_0}}\big|_{A_0\Gamma_{B_0}} \circ  \Phi_{\Gamma}(a\mathrm{d}_Bb) = \\
            &= \Delta_{\Gamma_{A_0}}\big|_{A_0\Gamma_{B_0}}\big(\phi(a)\mathrm{d}_{A_0}|_{B_0}\phi(b)\big) = \Delta_{\Gamma_{A_0}}\big|_{A_0\Gamma_{B_0}}\big(\phi(a)\mathrm{d}_{B_0}\phi(b)\big) = \\
            &= \phi(a_0)\mathrm{d}_{B_0}\phi(b) \otimes \pi(a_1)\pi(1_H) = \phi(a_0)\mathrm{d}_{B_0}\phi(b) \otimes \pi(a_1),
        \end{split}
        \label{eq:theta0isgood4}
    \end{equation}
    where we used that $\pi$ is a Hopf algebra map, so that $\pi(1_H) = 1_{H_0}$. Using Equation (\ref{eq:theta0isgood2}), the right-hand side of Equation (\ref{eq:theta0isgood3}) reads
    \begin{equation}
        \begin{split}
            &(\theta_0 \otimes id_{H_0}) \circ \rho_V^0(v) = (\theta_0 \otimes id_{H_0}) \circ (id_V \otimes \pi) \circ \rho_V(v) = \\
            &= (\theta_0 \otimes \pi)(v_0 \otimes v_1) = \Phi_{\Gamma}\big( \theta(v_0) \big) \otimes \pi(v_1) =\\
            &= (\Phi_{\Gamma} \otimes id_{H_0}) \circ \big( \theta(v_0) \otimes \pi(v_1) \big) = (\Phi_{\Gamma} \otimes id_{H_0}) \circ (id_{\Gamma_A} \otimes \pi) \circ \Delta_{\Gamma_A} \circ \theta(v) = \\
            &= (\Phi_{\Gamma} \otimes id_{H_0}) \circ (id_{\Gamma_A} \otimes \pi) \circ \Delta_{\Gamma_A}(a\mathrm{d}_Bb) = \\
            &= (\Phi_{\Gamma} \otimes id_{H_0}) \circ (id_{\Gamma_A} \otimes \pi) \circ \Delta_{\Gamma_A}\big|_{A\Gamma_B}(a\mathrm{d}_Bb) =\\
            &= (\Phi_{\Gamma} \otimes id_{H_0}) \circ (id_{\Gamma_A} \otimes \pi)(a_0\mathrm{d}_Bb \otimes a_1) =  (\Phi_{\Gamma} \otimes id_{H_0})\big(a_0\mathrm{d}_Bb \otimes \pi(a_1)\big) =\\
            &= \phi(a_0)\mathrm{d}_{B_0}\phi(b) \otimes \pi(a_1),
        \end{split}
        \label{eq:theta0isgood5}
    \end{equation}
    completing the proof.
\end{proof}
We can now show the main result of this thesis. We show that a quantum $G$-structure is a quantum frame resolution, analogously to what one expects in the classical case.
\begin{theorem}
    Let $(A_0, H_0, V, \theta_0)$, with $\theta_0 := \Phi_{\Gamma} \circ \theta$, be a quantum $G$-structure of $(A,H,V,\theta)$ a quantum frame resolution of $(B,\Gamma_B)$. Then $(A_0,H_0,V,\theta_0)$ is a quantum frame resolution of $(B_0,\Gamma_{B_0})$.
    \label{thm:aquantumGstructureisaquantumframeresolution}
\end{theorem}
\begin{proof}
    Let $\phi \colon A \to A_0$ be the surjective right $H_0$-comodule algebras morphism defining the reduction, satisfying
    \begin{equation}
        \begin{split}
            &\Delta_{A_0} \circ \phi = (\phi \otimes \pi) \circ \Delta_A, \;\;\; \text{that is}\\
            &\phi(a)_0 \otimes \phi(a)_1 = \phi(a_0) \otimes \pi(a_1)
        \end{split}
        \label{eq:phicomodulemorph}
    \end{equation}
    for all $a \in A$, where $\pi \colon H \to H_0 := H/J$ for some Hopf ideal $J$, is a Hopf algebra morphism.\\
    We have $\mathcal{E}_0 = (A_0 \otimes V)^{coH_0}$, the associated vector bundle to $B_0 := A_0^{coH_0} \subseteq A_0$, is the set of coinvariants under the right $H_0$-coaction on $A \otimes V$, defined, for all $\phi(a) \otimes v \in A_0 \otimes V$, as:
    \begin{equation}
        \Delta_{\mathcal{E}_0}\big(\phi(a) \otimes v\big) = \phi(a)_0 \otimes v_0 \otimes \phi(a)_1\pi(v_1).
        \label{eq:DeltaE0def}
    \end{equation}
    Using Equation (\ref{eq:phicomodulemorph}), for all $a \otimes v \in \mathcal{E} = (A \otimes V)^{coH}$, we can write
    \begin{equation}
        \begin{split}
            &\Delta_{\mathcal{E}_0} \circ (\phi \otimes id_V)(a \otimes v) = \Delta_{\mathcal{E}_0}\big(\phi(a) \otimes v\big) =  \\
            &= \phi(a)_0 \otimes v_0 \otimes \phi(a)_1\pi(v_1) = \phi(a_0) \otimes v_0 \otimes \pi(a_1)\pi(v_1) = \\
            &=  \phi(a_0) \otimes v_0 \otimes \pi(a_1v_1) = (\phi \otimes id_V \otimes \pi) (a_0 \otimes v_0 \otimes a_1v_1) =\\
            &= (\phi \otimes id_V \otimes \pi) \circ \Delta_{\mathcal{E}} (a \otimes v) = (\phi \otimes id_V \otimes \pi)(a \otimes v \otimes 1_H) =\\
            &= \phi(a) \otimes v \otimes \pi(1_{H}) = \phi(a) \otimes v \otimes 1_{H_0},
        \end{split}
        \label{eq:DeltaE0phiotimesid}
    \end{equation}
    proving that the restriction of $\phi \otimes id_V$ to $\mathcal{E}$ has values in $\mathcal{E}_0$.\\
    Consider the following diagram:
    \[
\begin{tikzcd}
\mathcal{E} \arrow[ddd, "(\phi \otimes id)"] \arrow[rrrrr, "s_{\theta}", shift left] &  &  &  &  & \Gamma_B \arrow[lllll, "s_{\theta^{-1}}", shift left=2] \arrow[ddd, "\Phi_{\Gamma}\big|_{\Gamma_B}", shift left] \\
                                                                                     &  &  &  &  &                                                                                                                  \\
                                                                                     &  &  &  &  &                                                                                                                  \\
\mathcal{E}_0 \arrow[rrrrr, "s_{\theta_0}"]                                          &  &  &  &  & \Gamma_{B_0} \arrow[uuu, "\Phi_{\Gamma}\big|_{\Gamma_B}^{-1}", shift left=2]                                    
\end{tikzcd},
\]
    where, as discussed, the restriction of $\Phi_{\Gamma}$ to base forms in Equation (\ref{eq:restrictionofGammaPhitoGammaB}) is an isomorphism, and $s_{\theta} = \cdot \circ (id_A \otimes \theta) \colon \mathcal{E} \to \Gamma_B$ is an isomorphism of left $B$-modules, since by assumption $(A,H,V,\theta)$ is a quantum frame resolution, with $\cdot$ the left $A$-action on $\Gamma_A$. Thus, we can prove that the map
    \begin{equation}
        \Psi := (\phi \otimes id_V) \circ s_{\theta}^{-1} \circ \Phi_{\Gamma}\Big|_{\Gamma_B}^{-1} \colon \Gamma_{B_0} \to \mathcal{E}_0,
        \label{eq:LAMAPPA}
    \end{equation}
    satisfies
    \begin{enumerate}
        \item $\Psi \circ s_{\theta_0} = id_{\mathcal{E}_0}$,
        \item $s_{\theta_0} \circ \Psi = id_{\Gamma_{B_0}}$,
    \end{enumerate}
    so that $s_{\theta_0} := \cdot^{\prime} \circ (id_{A_0} \otimes \theta_0)$, with $\cdot^{\prime}$ the left $A_0$-action on $\Gamma_{A_0}$, identifies the desired isomorphism $\mathcal{E}_0 \cong \Gamma_{B_0}$ of left $B_0$ modules. In order:
    \begin{enumerate}
        \item Let $\bar{a} \otimes v \in \mathcal{E}_0$ and let $\bar{a} := \phi(\tilde{a})$ for some $\tilde{a} \in A$. We have
        \begin{equation}
            \begin{split}
                &\Psi \circ s_{\theta_0} (\bar{a} \otimes v) = (\phi \otimes id_V) \circ s_{\theta}^{-1} \circ \Phi_{\Gamma}\big|_{\Gamma_B}^{-1} \circ s_{\theta_0} (\phi(\tilde{a}) \otimes v) = \\
                &= (\phi \otimes id_V) \circ s_{\theta}^{-1} \circ \Phi_{\Gamma}\big|_{\Gamma_B}^{-1} \circ \cdot^{\prime} \circ (id_{A_0} \otimes \theta_0) (\phi(\tilde{a}) \otimes v) = \\
                &= (\phi \otimes id_V) \circ s_{\theta}^{-1} \circ \Phi_{\Gamma}\big|_{\Gamma_B}^{-1} \circ \cdot^{\prime} \circ \big(id_{A_0} \otimes (\Phi_{\Gamma} \circ \theta)\big) (\phi(\tilde{a}) \otimes v) =\\
                &=(\phi \otimes id_V) \circ s_{\theta}^{-1} \circ \Phi_{\Gamma}\big|_{\Gamma_B}^{-1} \circ \cdot^{\prime} \circ (id_{A_0} \otimes \Phi_{\Gamma}) \circ (\phi \otimes \theta)(\tilde{a} \otimes v) = \\
                &= (\phi \otimes id_V) \circ s_{\theta}^{-1} \circ \Phi_{\Gamma}\big|_{\Gamma_B}^{-1} \circ \cdot^{\prime} \circ (\phi \otimes \Phi_{\Gamma}) \circ (id_A \otimes \theta)(\tilde{a} \otimes v)
            \end{split}
            \label{eq:theproofPsisx1}
        \end{equation}
        Since $(\phi, \Phi_{\Gamma})$ is a morphism of differential calculi, we know
        \begin{equation}
            \cdot^{\prime} \circ (\phi \otimes \Phi_{\Gamma}) = \Phi_{\Gamma} \circ \cdot,
            \label{eq:morphismofdiffcalculi}
        \end{equation}
        so that Equation (\ref{eq:theproofPsisx1}) reads
        \begin{equation}
        \begin{split}
                &\Psi \circ s_{\theta_0} (\bar{a} \otimes v) = (\phi \otimes id_V) \circ s_{\theta}^{-1} \circ \Phi_{\Gamma}\big|_{\Gamma_B}^{-1} \circ \Phi_{\Gamma} \circ \cdot \circ (id_A \otimes \theta)(\tilde{a} \otimes v) = \\
                &= (\phi \otimes id_V) \circ s_{\theta}^{-1} \circ \Phi_{\Gamma}\big|_{\Gamma_B}^{-1} \circ \Phi_{\Gamma} \circ s_{\theta}(\tilde{a} \otimes v) =\\
                &= (\phi \otimes id_V) \circ s_{\theta}^{-1} \circ \Phi_{\Gamma}\big|_{\Gamma_B}^{-1} \circ \Phi_{\Gamma}\big|_{\Gamma_B} \circ s_{\theta}(\tilde{a} \otimes v) =\\
                &= (\phi \otimes id_V)(\tilde{a} \otimes v) = \bar{a} \otimes v,
            \end{split}
            \label{eq:theproofPsisx2}
        \end{equation}
        where we used that $s_{\theta} =  \cdot \circ (id_A \otimes \theta)$ has values in $\Gamma_B$.
        \item On the other hand, for any $\phi(b)\mathrm{d}_{B_0}\phi(b^{\prime}) \in \Gamma_{B_0}$ with $b,b^{\prime} \in B$, we have
        \begin{equation}
            \begin{split}
                &s_{\theta_0} \circ \Psi\big(\phi(b)\mathrm{d}_{B_0}\phi(b)\big) = s_{\theta_0} \circ (\phi \otimes id_V) \circ s_{\theta}^{-1} \circ \Phi_{\Gamma}\Big|_{\Gamma_B}^{-1}\big(\phi(b)\mathrm{d}_{B_0}\phi(b)\big) =\\
                &=\cdot^{\prime} \circ \big(id_{A_0} \otimes (\Phi_{\Gamma} \circ \theta)\big) \circ (\phi \otimes id_V) \circ s_{\theta}^{-1} \circ \Phi_{\Gamma}\Big|_{\Gamma_B}^{-1}\big(\phi(b)\mathrm{d}_{B_0}\phi(b)\big) =\\
                &= \cdot^{\prime} \circ (\phi \otimes \Phi_{\Gamma}) \circ (id_A \otimes \theta) \circ s_{\theta}^{-1} \circ \Phi_{\Gamma}\Big|_{\Gamma_B}^{-1}\big(\phi(b)\mathrm{d}_{B_0}\phi(b)\big) =\\
                &= \Phi_{\Gamma} \circ \cdot \circ (id_A \otimes \theta) \circ s_{\theta}^{-1} \circ \Phi_{\Gamma}\Big|_{\Gamma_B}^{-1}\big(\phi(b)\mathrm{d}_{B_0}\phi(b)\big) =\\
                &= \Phi_{\Gamma} \circ s_{\theta} \circ s_{\theta}^{-1}\circ \Phi_{\Gamma}\Big|_{\Gamma_B}^{-1}\big(\phi(b)\mathrm{d}_{B_0}\phi(b)\big) = \phi(b)\mathrm{d}_{B_0}\phi(b_0).
            \end{split}
            \label{eq:theproofPsidx1}
        \end{equation}
    \end{enumerate}
    This shows $s_{\theta_0}$ is a bijection, proving that $(A_0,H_0,V,\theta_0)$ is a quantum frame resolution.
\end{proof}

\subsection{An example: quantum Hermitian structures} \label{exampleofqGstruc}

Given the quantum frame resolution $(B\#H,H,V,\theta)$ in \ref{exampleofqfr}, with $B = \mathbb{C}_q^2$, $H = GL_q(2)$, we want to establish that $(B\#H_0,H_0,V,\theta_0)$, with $H_0 = U_q(2)$, is a quantum frame resolution. In particular, this is shown to be an example of quantum $G$-structure, which we call a quantum Hermitian structure. \\
Here, $V$ as in Equation (\ref{eq:Vinourexample}) is a right $H_0$-comodule by $\rho_V^0 := (id \otimes \pi) \circ \rho_V$, and $\theta_0 := \Phi_{\Gamma_{\#}} \circ \theta$, with $\theta$ as in (\ref{eq:thetainourexample}). The morphism of differential calculi $\Phi_{\Gamma_{\#}}$ is that between the right $H$-covariant smashed product calculus $(\Gamma_{\#},\mathrm{d}_{\#})$ and the quotient calculus on $B\#H_0$. Thus, by the discussion in the previous section and Theorem \ref{thm:aquantumGstructureisaquantumframeresolution}, it is enough to define a surjective morphism of Hopf algebras $\pi \colon H \to H_0$ and a surjective morphism of right $H_0$-comodule algebras $\phi \colon B\#H \to B\#H_0$, since in our case the coinvariants of $B\#H$ coincides with those of $B\#H_0$.\\
The Hopf algebra $U_q(2)$ is generated by elements $\alpha, \alpha^{\ast}, \gamma, \gamma^{\ast}$ and the central invertible element $\delta$, satisfying the following relations:
\begin{equation}
    \begin{split}
        &\alpha\gamma = q\gamma\alpha, \;\; \alpha\gamma^{\ast} = q\gamma\alpha, \;\; \gamma^{\ast}\gamma = \gamma\gamma^{\ast},\\
        &\alpha\alpha^{\ast} + q^2\gamma\gamma^{\ast} = 1, \;\; \alpha^{\ast}\alpha + \gamma\gamma^{\ast} = 1.
    \end{split}
    \label{ref:qcommrelforUq2}
\end{equation}
The coalgebra structure is given by the coproduct $\Delta_0 \colon H_0 \to H_0 \otimes H_0$ defined as
\begin{equation}
    \begin{split}
        &\Delta_0(\alpha) = \alpha \otimes \alpha - q\delta\gamma \otimes \gamma,\\
        &\Delta_0(\alpha^{\ast}) = \alpha^{\ast} \otimes \alpha^{\ast} - q\delta^{-1}\gamma^{\ast} \otimes \gamma,\\
        &\Delta_0(\gamma) = \gamma \otimes \alpha + \delta\alpha^{\ast} \otimes \gamma,\\
        &\Delta_0(\gamma^{\ast}) = \gamma^{\ast} \otimes \alpha^{\ast} + \delta\alpha \otimes \gamma^{\ast},\\
        &\Delta_0(\delta) = \delta \otimes \delta.
    \end{split}
    \label{eq:coproductforUq2}
\end{equation}
We require $q^{\ast}=q$, so that compatibility with the $\ast$-structure is satisfied. The counit $\epsilon_0 \colon H_0 \to \mathbb{K}$ is
\begin{equation}
    \epsilon_0(\alpha) = \epsilon_0(\alpha^{\ast}) = \epsilon_0(\delta) = 1, \;\; \epsilon_0(\gamma) = \epsilon_0(\gamma^{\ast}) = 0.
    \label{eq:counitforUq2}
\end{equation}
The antipode $S_0 \colon H_0 \to H_0$ is
\begin{equation}
    \begin{split}
        &S_0(\alpha) = \alpha^{\ast}.\\
        &S_0(\alpha^{\ast}) = \alpha,\\
        &S_0(\gamma) = -q\delta^{-1}\gamma,\\
        &S_0(\gamma^{\ast}) = q\delta\gamma^{\ast},\\
        &S_0(\delta) = \delta^{-1}.
    \end{split}
    \label{eq:antipodeforUq2}
\end{equation}
One can check that these maps satisfy the Hopf algebra axioms, coproduct and counit are algebra maps and the antipode is an antialgebra map, and they are compatible with the $\ast$-structure.\\
We define a morphism of Hopf algebras
\begin{equation}
    \begin{split}
        \pi \colon &GL_q(2) \to U_q(2),\\
        &a \mapsto \alpha,\\
        &b \mapsto -q\delta\gamma^{\ast},\\
        &c \mapsto \gamma,\\
        &d \mapsto \delta\alpha^{\ast},\\
        &D^{-1} \to \delta^{-1}.
    \end{split} 
    \label{eq:morphismofHopfalgebrasUq2}
\end{equation}
It is easy to check that this is amorphism of Hopf algebras, i.e.:
\begin{equation}
    \begin{split}
        &\pi \circ \mu = \mu_0 \circ (\pi \otimes \pi), \;\; \pi \circ \eta = \eta_0,\\
        &(\pi \otimes \pi) \circ \Delta = \Delta_0 \otimes \pi, \;\; \epsilon = \epsilon_0 \otimes \pi,\\
        &\pi \circ S = S_0 \otimes \pi.
    \end{split}
    \label{eq:pipropertiesHermitianjusttomakeitsimple}
\end{equation}
Moreover, this map is surjective, since one can check that any generator of $U_q(2)$ can be written as $\pi(h)$ for some $h \in GL_q(2)$, namely
\begin{equation}
    \begin{split}
        \alpha = \pi(a), \;\; \alpha^{\ast} = \pi(D^{-1}d), \;\; \gamma = \pi(c), \;\; \gamma^{\ast} = -q^{-1}\pi(D^{-1}b), \;\; \delta =\pi(D).
    \end{split}
    \label{eq:piofsomethingtheymustbe}
\end{equation}
We observe that $A = B\#H$ is a right $H_0$-comodule by the right $H_0$-coaction 
\begin{equation}
    \Delta^0_{B\#H} := (id_B\#id_H \otimes \pi) \circ \Delta_{B\#H},
    \label{eq:robacheserveunpo}
\end{equation}
where $\Delta_{B\#H} = (id \otimes \bar{\pi}) \circ \Delta_{\#}$ as in Equation (\ref{eq:rightHcoactiononBsmashH}). Notice that $\Delta_{B\#H} = id_B\#\Delta$, with $\Delta$ the coproduct of $H$, so that (\ref{eq:robacheserveunpo}) reads
\begin{equation}
    \Delta_{B\#H}^0 := (id_B\#id_H \otimes \pi) \circ (id_B\#\Delta).
    \label{eq:rightH0coactiononBsmashHexamplenuovo}
\end{equation}
Similarly, $A_0 = B\#H_0$ is a right $H_0$-comodule via the right $H_0$-coaction
\begin{equation}
    \Delta_{B\#H_0} := id_B\#\Delta_0.
    \label{eq:rightH0coactionforBsmashH0examplenuovo}
\end{equation}
We define the map
\begin{equation}
    \begin{split}
        \phi := (id_B\#\pi) \colon &B\#H \to B\#H_0,\\
        &b\#h \to b\#\pi(h).
    \end{split}
    \label{eq:thesurjectivemorphismofrightH0comodulealgebrasexample}
\end{equation}
This is a morphism of right $H_0$-comodules, since
\begin{equation}
    \begin{split}
        &\Delta_{B\#H_0} \circ \phi = (id_B\#\Delta_0) \circ (id_B\#\pi) = \\
        &=id_B\#(\Delta_0 \circ \pi) = \\
        &=(\phi \otimes id_{H_0}) \circ \Delta_{B\#H}^0 = (id_B\#\pi \otimes id_{H_0}) \circ (id_B\#id_H \otimes \pi) \circ (id_B\#\Delta) = \\
        &=id_B\#(\pi \otimes \pi) \circ \Delta,
    \end{split}
    \label{eq:thisphiisamorphismofH0comodule}
\end{equation}
which given that $\pi$ is a Hopf algebra morphism. Moreover, we require that $\pi$ respects the left $H$-module action on $B$ in Lemma \ref{lem:leftHmoduleanctiononB}, such that $\phi$ is an algebra morphism too. In particular, the left $H_0$-module action on $B$ is defined as:
\begin{equation}
    \begin{split}
        &\alpha \triangleright x = q^{-2}x, \;\; \alpha^{\ast} \triangleright x = q^2x, \;\; \gamma \triangleright x = (q^{-2} - 1)y, \;\; \gamma^{\ast} \triangleright x = 0, \;\; \delta \triangleright x = q^{-3}x,\\
        &\alpha \triangleright y = q^{-1}y, \;\; \alpha^{\ast} \triangleright y = qy, \;\; \gamma \triangleright y = 0, \;\; \gamma^{\ast} \triangleright y = 0, \;\; \delta \triangleright y = q^{-3}y.
    \end{split}
    \label{eq:leftUq2actiononB}
\end{equation}
Surjectivity of $\phi$ follows from that of $\pi$. Thus, 
by Theorem \ref{thm:aquantumGstructureisaquantumframeresolution}, we conclude that $(B\#H_0, H_0, V, \theta_0)$ is a quantum frame resolution of $(B\#1_{H_0}, \Gamma_{B\#1_{H_0}})$.
\bibliographystyle{plain}  
\bibliography{bibliography}   

\end{document}